\documentclass[UTF-8,reqno]{amsart}
\usepackage{enumerate}
\usepackage{amssymb,url,color, booktabs}

\usepackage{mathrsfs}
\usepackage{amsmath}

\usepackage{fancyhdr}
\usepackage[nobysame]{amsrefs}
\BibSpec{article}{%
+{}{\PrintAuthors} {author}
+{,}{ \textrm} {title}
+{.}{ \textit} {journal}
+{,}{ \textbf} {volume}
+{}{ \parenthesize} {date}
+{,}{ } {pages}
+{.}{ arXiv:} {eprint}
+{.}{} {transition}
}
\BibSpec{book}{%
+{}{\PrintAuthors} {author}
+{,}{ \textit} {title}
+{.}{ \textrm} {series} 
+{,}{ Vol.} {volume}
+{.}{ } {publisher}
+{,}{ } {date}
+{.}{} {transition}
}
\usepackage{color}
\usepackage[colorlinks=true]{hyperref}
\hypersetup{
    linkcolor=blue,          
    citecolor=red,        
    filecolor=blue,      
    urlcolor=cyan
}

\numberwithin{equation}{section}

\newcommand{\be}{\begin{eqnarray}}
\newcommand{\ee}{\end{eqnarray}}
\newcommand{\ce}{\begin{eqnarray*}}
\newcommand{\de}{\end{eqnarray*}}
\newtheorem{theorem}{Theorem}[section]
\newtheorem{lemma}[theorem]{Lemma}
\newtheorem{remark}[theorem]{Remark}
\newtheorem{definition}[theorem]{Definition}
\newtheorem{proposition}[theorem]{Proposition}
\newtheorem{Examples}[theorem]{Example}
\newtheorem{corollary}[theorem]{Corollary}
\newtheorem{assumption}[theorem]{Assumption}

\newenvironment{proof of theorem 1.2}{{\it Proof of Theorem 1.2}.}{{\hfill 	
$\square$\hskip - \parfillskip}}
\newenvironment{proof of theorem 1.3}{{\it Proof of Theorem 1.3}.}{{\hfill 	
		$\square$\hskip - \parfillskip}}
\newenvironment{proof of theorem 1.5}{{\it Proof of Theorem 1.5}.}{{\hfill 	
		$\square$\hskip - \parfillskip}}
\newenvironment{proof of theorem 1.4}{{\it Proof of Theorem 1.4}.}{{\hfill 	
		$\square$\hskip - \parfillskip}}
\newenvironment{proof of theorem 1.6}{{\it Proof of Theorem 1.6}.}{{\hfill 	
		$\square$\hskip - \parfillskip}}
\newenvironment{proof of theorem 5.3}{{\it Proof of Theorem 5.3}.}{{\hfill 	
		$\square$\hskip - \parfillskip}}
\newenvironment{proof of (1.3)}{{\it Proof of (1.3)}.}{{\hfill 	
		$\square$\hskip - \parfillskip}}
\newenvironment{proof of theorem 1.8}{{\it Proof of Theorem 1.8}.}{{\hfill 	
			$\square$\hskip - \parfillskip}}
\newenvironment{proof of corollary 1.6 and 1.7}{{\it Proof of Corollary 1.6 and 1.7}.}{{\hfill 	
		$\square$\hskip - \parfillskip}}

\makeatletter
\newcommand{\rmnum}[1]{\romannumeral #1}
\newcommand{\Rmnum}[1]{\expandafter\@slowromancap\romannumeral #1@}
\makeatother

\def\eps{\varepsilon}

\def\e{\mathrm{e}}

\def\a{\alpha}
\def\om{\omega}
\def\Om{\Omega}

\def\p{\partial}

\def\g{\gamma}
\def\l{\lambda}

\def\[{{\Big[}}
\def\]{{\Big]}}
\def\<{{\langle}}
\def\>{{\rangle}}
\def\({{\Big(}}
\def\){{\Big)}}

\def\bx{{\mathbf{x}}}

\def\osc{{\rm osc}}

\def\min{{\mathord{{\rm min}}}}
\def\Vol{\mathord{{\rm Vol}}}

\def\={&\!\!=\!\!&}

\def\cJ{{\mathcal J}}
\def\cK{{\mathcal K}}
\def\cL{{\mathcal L}}

\def\mR{{\mathbb R}}
\def\mS{{\mathbb S}}

\def\1{{\mathbf{1}}}

\def\sA{{\mathscr A}}

\def\sF{{\mathscr F}}

\def\geq{\geqslant}
\def\leq{\leqslant}
\def\ge{\geqslant}
\def\le{\leqslant}

\def\k{\kappa}

\def\eps{\varepsilon}

\def\e{\mathrm{e}}

\def\a{\alpha}
\def\om{\omega}
\def\Om{\Omega}

\def\p{\partial}

\def\g{\gamma}
\def\l{\lambda}

\def\[{{\Big[}}
\def\]{{\Big]}}
\def\<{{\langle}}
\def\>{{\rangle}}
\def\({{\Big(}}
\def\){{\Big)}}

\def\bx{{\mathbf{x}}}

\def\osc{{\rm osc}}

\def\min{{\mathord{{\rm min}}}}
\def\Vol{\mathord{{\rm Vol}}}

\def\={&\!\!=\!\!&}
\def\bt{\begin{theorem}}
\def\et{\end{theorem}}
\def\bl{\begin{lemma}}
\def\el{\end{lemma}}
\def\br{\begin{remark}}
\def\er{\end{remark}}
\def\bx{\begin{Examples}}
\def\ex{\end{Examples}}
\def\bd{\begin{definition}}
\def\ed{\end{definition}}
\def\bp{\begin{proposition}}
\def\ep{\end{proposition}}
\def\bc{\begin{corollary}}
\def\ec{\end{corollary}}

\def\geq{\geqslant}
\def\leq{\leqslant}
\def\ge{\geqslant}
\def\le{\leqslant}

 \def\nn{\nabla}

\def\<{\langle} \def\>{\rangle}

\def\bpf{\begin{proof}}
\def\epf{\end{proof}}

\allowdisplaybreaks

\begin{document}
	
\title{a class of inverse curvature flows and $L^p$ dual Christoffel-Minkowski problem}\thanks{\it {This research was partially supported by NSFC (No. 11871053).}}
\author{Shanwei Ding and Guanghan Li}


\address{School of Mathematics and Statistics, Wuhan University, Wuhan 430072, China.
}

\begin{abstract}
In this paper, we consider a large class of expanding flows of closed, smooth, star-shaped hypersurface in Euclidean space $\mathbb{R}^{n+1}$ with speed $\psi u^\alpha\rho^\delta f^{-\beta}$, where $\psi$ is a smooth positive function on unit sphere, $u$ is the support function of the hypersurface, $\rho$ is the radial function, $f$ is a smooth, symmetric, homogenous of degree one, positive function of the principal curvatures of the hypersurface on a convex cone. When $\psi=1$, we prove that the flow exists for all time and converges to infinity if $\a+\delta+\beta\le1$, and $\a\le0<\beta$, while in case $\a+\delta+\beta>1$, $\a,\delta\le0<\beta$, the flow blows up in finite time, and where we assume the initial hypersurface to be strictly convex. In both cases the properly rescaled flows converge to a sphere centered at the origin. In particular, the results of Gerhardt \cite{GC,GC3} and Urbas \cite{UJ2} can be recovered by putting $\a=\delta=0$. Our previous works \cite{DL,DL2} and Hu, Mao, Tu and Wu \cite{HM} can be recovered by putting $\delta=0$ and $\a=0$ respectively. By the convergence of these flows, we can give a new proof of uniqueness theorems for solutions to $L^p$-Minkowski problem and $L^p$-Christoffel-Minkowski problem with constant prescribed data. Similarly, we consider the $L^p$ dual Christoffel-Minkowski problem and prove a uniqueness theorem for solutions to $L^p$ dual Minkowski problem and $L^p$ dual Christoffel-Minkowski problem with constant prescribed data. At last, we focus on the longtime existence and convergence of
 a class of anisotropic flows (i.e. for general function $\psi$).
The final result not only gives a new proof of many previously known solutions to $L^p$ dual Minkowski problem, $L^p$-Christoffel-Minkowski problem, etc. by such anisotropic flows, but also provides solutions to $L^p$ dual Christoffel-Minkowski problem with some conditions.

\end{abstract}

\maketitle
\setcounter{tocdepth}{2}
\tableofcontents

\section{Introduction}
Flows of convex hypersurfaces  by a class of speed functions which are homogenous and symmetric in principal curvatures have been extensively studied in the past four decades. Well-known examples include the mean curvature flow \cite{HG}, and the Gauss curvature flow \cite{BS,FWJ}. In \cite{HG} Huisken showed that the flow has a unique smooth solution and the hypersurface converges to a round sphere if the initial hypersurface is closed and convex. Later, a range of flows with the speed of homogenous of degree one in principal curvatures were established,  see \cite{B0,B1,CB1,CB2} and references therein.

For star-shaped hypersurface $M_0$, Gerhardt \cite{GC3,GC} and Urbas \cite{UJ2} studied the flow with concave curvature function $f$ which satisfies $f\vert_{\Gamma}>0$ and $f\vert_{\p\Gamma}=0$ for an open convex symmetric cone $\Gamma$ containing the positive cone $\Gamma_+$, and proved a similar convergence result. Scheuer \cite{SJ} improved the asymptotical behavior of the flow considered in \cite{GC} by showing that the flow becomes close to a flow of a sphere.

The inverse curvature flow has also been studied in other ambient spaces, in particular in the hyperbolic space and in sphere, See \cite{GC4,SJ2,GC5,LW,SJ3} etc..

Flow with speed depending not only on the curvatures has recently begun to be considered. For example, flows that deform hypersurfaces by their curvature and support function were studied in \cite{IM,SWM,SJ4,GL} etc..

For a certain range of $\alpha,\beta$, the limit of flows with speed $u^\alpha f^\beta$ can be an ellipsoid. For example, Andrews \cite{A9} proved that the solution will converge in $C^\infty$ to an ellipsoid along the contracting flow with the speed of $\frac{1}{n+2}$-power of the Gauss-Knonecker curvature after scaling. In \cite{IM2,IM3}, the authors studied flows of the convex hypersurfaces at the speeds of $-u^\alpha K^\beta$ and $\psi u^{2-m} K^{-1}$ respectively 
 and the solutions converge to an ellipsoid.

A class of curvature flows were introduced by \cite{IM,LSW}, where the speed of the flow depends on an anisotropic factor, support function or radial function, and a curvature function. These flows can solve the $L_p$-Christoffel-Minkowski problems or dual Minkowski problems. Whether these flows can be extended is an interesting problem. In the present works \cite{DL,DL2} we considered this kind of flow,$$\frac{\p X}{\p t}=u^\a f^{-\beta}\nu,$$
in the Euclidean space $\mathbb{R}^{n+1}$, $n\geq2$. When $f=(\frac{\sigma_n}{\sigma_k})^{\frac{1}{n-k}}$, the flow has been studied by Sheng-Yi in \cite{SWM}. In \cite{DL,DL2} we got a solution which exists for all time if $\a+\beta\le1$. Then there is a natural problem how the hypersurface will evolve along this flow with speed $\psi u^\alpha\rho^\delta f^{-\beta}$. This flow can be used to solve the $L_p$-Christoffel-Minkowski problem naturally.

 Let $M_0$ be a closed, smooth and star-shaped hypersurface in $\mathbb{R}^{n+1}$ ($n\geq2$), and $M_0$ encloses the origin. In this paper, we study the following expanding flow
 \begin{equation}
 	\label{x1.1}
 	\begin{cases}
 		&\frac{\partial X}{\partial t}(x,t)=\psi u^\alpha\rho^\delta f^{-\beta}(x,t) \nu(x,t),\\
 		&X(\cdot,0)=X_0,
 	\end{cases}
 \end{equation}
 where $f(x,t)=f(\kappa(x,t))$ is a suitable curvature function of the hypersurface $M_t$ parameterized by $X(\cdot,t): M^n\times[0,T^*)\to \mR^{n+1}$, $\kappa=(\kappa_1,...,\kappa_n)$ are the principal curvatures of the  hypersurface $M_t$, $\beta>0$, $\psi>0$ is a positive function defined in $\mS^n$, $u$ is the support function, $\rho$ is the radial function and $\nu(\cdot,t)$ is the outer unit normal vector field to $M_t$.


We obtain convergence results for a large class of speeds if $\psi\equiv1$ and therefore make the following assumption.
\begin{assumption}\label{a1.1}
	Let $\Gamma\subseteq\mathbb{R}^n$ be a symmetric, convex, open cone containing
\begin{equation*}
	\Gamma_+=\{(\k_i)\in\mathbb{R}^n:\k_i>0\},
\end{equation*}
and suppose that $f$ is positive in $\Gamma$, homogeneous of degree $1$, and concave with
\begin{equation*}
	\frac{\p f}{\p\k_i}>0,\quad f\vert_{\p\Gamma}=0,\quad f^{-\beta}(1,\cdots,1)=\eta.
\end{equation*}
\end{assumption}

We first prove the long time existence and convergence of the flow (\ref{x1.1}) with $\psi\equiv1$, i.e.
\begin{equation}
	\label{1.1}
	\begin{cases}
		&\frac{\partial X}{\partial t}(x,t)=u^\alpha\rho^\delta f^{-\beta}(\kappa) \nu(x,t),\\
		&X(\cdot,0)=X_0.
	\end{cases}
\end{equation}
\begin{theorem}\label{t1.2}
	Assume $\alpha, \delta, \beta\in \mR$ satisfying $\alpha \le 0<\beta\le 1-\alpha-\delta$. Let $f\in C^2(\Gamma)\cap C^0(\p\Gamma)$ satisfy Assumption \ref{a1.1}, and let $X_0(M)$ be the embedding of a closed $n$-dimensional manifold $M^n$ in $\mathbb{R}^{n+1}$ such that $X_0(M)$ is a graph over $\mathbb{S}^n$, and such that $\k\in\Gamma$ for all n-tuples of principal curvatures along $X_0(M)$. Then the flow (\ref{1.1}) has a unique smooth solution $M_t$ for all time $t>0$. For each $t\in[0,\infty)$, $X(\cdot,t)$ is a parameterization of a smooth, closed, star-shaped hypersurface $M_t$ in $\mR^{n+1}$ by $X(\cdot,t)$: $M^n\to \mR^{n+1}$. After a proper rescaling $X\to \varphi^{-1}(t)X$, where
	\begin{equation}\label{1.2}
		\begin{cases}
			\varphi(t)=e^{\eta t} &\text{ if }\alpha=1-\delta-\beta,\\
			\varphi(t)=(1+(1-\beta-\delta-\alpha)\eta t)^{\frac{1}{1-\beta-\delta-\alpha}} &\text{ if }\alpha\not=1-\delta-\beta,
		\end{cases}
	\end{equation}
	the hypersurface $\widetilde M_t=\varphi^{-1}M_t$ converges exponentially to a round sphere centered at the origin in the $C^\infty$-topology.
\end{theorem}

We remark that there are only a few results for hypersurface blowing up in finite time along some flows. In this paragraph, we mention these cases with respect to flow (\ref{1.1}). In dimension $n=2$, let $\a=\delta=0$ and $f=\sigma_n^\frac{1}{n}$, where $\sigma_n$ is defined in Page 4, then Schn$\ddot{u}$rer \cite{SO} considered the case $\beta=2=n$ and Li \cite{LQ} the case $1<\beta\le2=n$. The restriction to two dimensions is due  to the method of proof for the crucial curvature estimates which relies on the monotonicity of a certain rather artificial expression depending on the principal curvatures of the flow hypersurfaces. In \cite{CL}, Chen-Li proved that the flow blows up in finite time if $\a=0$, $\delta\le0$, $f=\sigma_n^\frac{1}{n}$ and $\delta+\beta>1$. In \cite{BIS}, Bryan-Ivaki-Scheuer proved that the flow $\frac{\p X}{\p t}=\psi u^\a\sigma_n^{-1}\nu$ blows up in finite time if $\psi\in C^\infty(\mS^n)$ is even and $\a+n>1$. In \cite{GC}, Gerhardt has proved that the flow blows up in finite time if $\a=\delta=0$ and $\beta>1$. In \cite{KS}, a similar result was proved if the initial hypersurface satisfies the pinching condition $\vert\vert A\vert\vert^2-H^2<c_0H^2$ if $\a=\delta=0$ and $\beta>1$, where $0<c_0=c_0(f,n,\beta)<\frac{1}{n(n-1)}$ is sufficiently small. In the second part of this paper, we shall show how the hypersurface will evolve along these flows with $\a+\delta+\beta>1$. The results in \cite{CL,GC,SO,LQ} are covered by Theorem \ref{t1.3} and the result in \cite{BIS} is covered by Theorem \ref{xt1.5}.

\begin{theorem}\label{t1.3}
	Assume $\alpha, \delta, \beta\in \mR$ satisfying $\alpha,\delta\le 0,\beta>0,\a+\delta+\beta>1$. Let $f\in C^2(\Gamma_+)\cap C^0(\p\Gamma_+)$ satisfy Assumption \ref{a1.1} with $\Gamma=\Gamma_+$, and let $X_0(M)$ be the embedding of a closed, smooth, strictly convex $n$-dimensional manifold $M^n$ in $\mathbb{R}^{n+1}$ enclosing the origin. Then the flow (\ref{1.1}) has a unique smooth solution $X_t$ which is defined on a maximal finite interval $[0,T^*)$. Moreover, the leaves $X(t)$ can be written as graphs of a function $\rho$ over $\mS^n$ such that
\begin{equation}\label{1.3}
\lim_{t\rightarrow T^*}\inf_{\mS^n}\rho(\cdot,t)=\infty.
\end{equation}
After a proper rescaling $X\to \varphi^{-1}(t)X$, where
	\begin{equation}\label{1.4}
	\varphi(t)=\((\alpha+\delta+\beta-1)\eta(T^*-t)\)^{\frac{1}{1-\delta-\beta-\alpha}},
	\end{equation}
the hypersurface $\widetilde M_t=\varphi^{-1}M_t$ converges exponentially to the unit sphere centered at the origin in the $C^\infty$-topology.
\end{theorem}

Flow (\ref{1.1}) can be described by an ODE of the radial function if $\beta=0$. So we do not state that result here. We apply the Aleksandrov reflection principle to derive the $C^0$ estimates in the proof of Theorem \ref{t1.3}. Thus, $\a,\delta\le0$ and convexity are essential in Theorem \ref{t1.3}. However, we can find the convexity and $\delta\le0$ are not essential in the proof of $C^2$ estimates. Thus, if initial hypersurface is not convex or $\delta>0$, whether we can derive similar results is a natural question. When $\a=\delta=0$, the reflection principle for (\ref{1.1}) has been studied by Chow and Gulliver \cite{CG}. In \cite{GC,KS},  the Aleksandrov reflection principle was also used to derive $C^0$ estimates for $\a=\delta=0$. In \cite{CL}, Chen-Li follow the ideas in \cite{GC,KS} to solve the case $\a=0$, $f=K^\frac{1}{n}$, where $K$ is Gauss curvature.

When the initial hypersurface is a sphere,  (\ref{1.1}) is equivalent to an ODE, since the leaves of the flow will then be spheres too, and the spherical flow will develop a blow up in finite time if $\a+\beta+\delta>1$. For $\a+\beta+\delta\le1$ the spherical flow will exist for all time and converge to infinity.

The $k$th elementary symmetric function $\sigma_k$ is defined by
$$\sigma_k(\k_1,...,\k_n)=\sum_{1\leq i_1<\cdot\cdot\cdot<i_k\leq n}\k_{i_1}\cdot\cdot\cdot\k_{i_k},$$
and let $\sigma_0=1$.

Let us make some remarks about our conditions. The convex cone $\Gamma$ that contains the positive cone in Assumption \ref{a1.1} is decided by $f$, and $\frac{\p f}{\p\k_i}>0$ ensures that this equation is parabolic.  Star-shaped initial hypersurface means it can be written as a graph over  $\mathbb{S}^n$. Examples of functions $f$ include the curvature quotients $f=(\frac{\sigma_l}{\sigma_k})^{\frac{1}{l-k}}$ with $0\leq k< l\leq n$ and the power means  $f=(\sum_{i=1}^n\k_i^{k})^{\frac{1}{k}}$ with $k<0$. More examples can be constructed as following: if $f_1,\cdots,f_k$ satisfy our conditions, then $f=\prod_{i=1}^kf_i^{\alpha_i}$ also satisfies our conditions, where $\alpha_i\ge 0$ and $\sum_{i=1}^k\alpha_i=1$. The details can be seen in \cite{DL2} and more example can be found in \cite{B3,B4}.





The study of the asymptotic behaviour of the flow (\ref{1.1}) is equivalent to the long time behaviour of the normalized flows. Let $\widetilde X(\cdot,\tau)=\varphi^{-1}(t)X(\cdot,t)$, where
\begin{equation}\label{1.5}
	\tau=\begin{cases}
		t &\text{  if }\alpha+\delta+\beta=1,\\
		\frac{\log((1-\alpha-\delta-\beta)\eta t+1)}{(1-\alpha-\delta-\beta)\eta} &\text{  if }\alpha+\delta+\beta<1,\\
	\frac{\log(\frac{T^*-t}{T^*})}{(1-\alpha-\delta-\beta)\eta} &\text{  if }\alpha+\delta+\beta>1.
	\end{cases}
\end{equation}
Obviously, $\frac{\p\tau}{\p t}=\varphi^{\a+\delta+\beta-1}$ and $\tau(0)=0$. Then $\widetilde X(\cdot,\tau)$ satisfies the following normalized flow,
\begin{equation}\label{1.6}
	\begin{cases}
		\frac{\partial \widetilde X}{\p \tau}(x,\tau)=\widetilde u^\alpha\widetilde\rho^\delta\widetilde f^{-\beta}(\widetilde\k)\nu-\eta \widetilde X,\\[3pt]
		\widetilde X(\cdot,0)=\widetilde X_0.
	\end{cases}
\end{equation}
For convenience we still use $t$ instead of $\tau$ to denote the time variable and omit the ``tilde'' if no confusions arise and we mention the scaled flow or normalized flow. We can find that the flow (\ref{1.6}) is equivalent (up to an isomorphism) to
\begin{equation}\label{1.7}
	\begin{cases}
		&\frac{\partial X}{\partial t}=(u^\alpha\rho^\delta f^{-\beta}(\k)-\eta u) \nu(x,t),\\
		&X(\cdot,0)=X_0.
	\end{cases}
\end{equation}

In order to prove Theorem \ref{t1.2}, we shall establish the a priori estimates for the normalized flow (\ref{1.7}), and show that if $X(\cdot,t)$ solves (\ref{1.7}), then the radial function $\rho$ converges exponentially to a constant as $t\to\infty$.

The proof of Theorem \ref{t1.3} is different from Theorem \ref{t1.2}. We first need to prove the existence of $T^*$ in the proof of Theorem \ref{t1.3}. Therefore in order to prove Theorem \ref{t1.3}, we shall establish the a priori estimates for the original flow (\ref{1.1}) if $M_t\subset B_R(0)$ first. This allows to continue the solution $M_t$ smoothly past $t=T^*$. After that we establish the a priori estimates for the normalized flow (\ref{1.7}) and show that if $X(\cdot,t)$ solves (\ref{1.7}), then the radial function $\rho$ converges to $1$ as $t\to\infty$.

By the convergence of flow (\ref{1.7}) we can derive the existence and uniqueness results of the $L^p$-Minkowski problem and $L^p$-Christoffel-Minkowski problem mentioned below. As applications, anisotropic flows \cite{B5,CL,IM,IM3,LSW} etc.. usually provided alternative proofs and smooth category approach of the existence of solutions to elliptic PDEs arising in convex body geometry. In these references, the authors usually consider expanding or contracting flows of the convex hypersurfaces at the speeds of $\psi u^\a\sigma_{k}^\beta(\l)$ or $-\psi \rho^\a\sigma_{n}^\beta(\k)$ respectively, where $\psi$ is a smooth positive function on $\mS^n$, $\k$ are the principal curvature of the hypersurface, $\l=\frac{1}{\k}$ are the principal curvature radii of the hypersurface.
 One advantage of this method is that there is no need to employ the constant rank theorem. For example, if $f^{-1}=\sigma_{k}^\frac{1}{k}(\l)$, by the bound of $f^{-1}=\sigma_{k}^\frac{1}{k}(\l)=(\frac{\sigma_{n-k}}{\sigma_{n}})^\frac{1}{k}(\k)$ one has that the hypersurfaces naturally preserve convexity. As a natural extension, in the last part of this paper, we let $f(\k)=(\frac{\sigma_{n}(\k)}{\sigma_{n-k}(\k)})^\frac{1}{k}=\sigma_{k}^{-\frac{1}{k}}(\l)$ in flow (\ref{x1.1}) and consider the anisotropic expanding flow 
 \begin{equation}\label{x1.8}
 	\begin{cases}
 	&\frac{\partial X}{\partial t}=\psi u^{\alpha}\rho^\delta \sigma_{k}^{\frac{\beta}{k}}(\l) \nu(x,t),\\
 	&X(\cdot,0)=X_0,
 \end{cases}
 \end{equation}
where $k$ is an integer and $1\le k\le n$. Usually, we need to get at least two integral monotone quantities to derive the convergence of anisotropic flows. If $k=n$, we can define a functional (\ref{x1.16}) to derive the convergence of anisotropic flows. However, if $k<n$ and $\delta\ne0$, so far there is no way to handle this situation. 
One reason is that if we consider $\int_{\mS^n}\sF\rho(\xi) d\xi$, there will appear Gauss curvature $K$ by (\ref{2.28}) that we don't want. If we consider $\int_{\mS^n}\sF\rho(\xi(x))dx=\int_{\mS^n}\sF\sqrt{u^2+|Du|^2}dx$, the evolution equation of it is hard to handle along (\ref{x1.8}), where we let $\xi:\mS^n\rightarrow M$ and $x:\mS^n\rightarrow M$ be the radial map and inverse Gauss map respectively. Inspired by Theorem \ref{t1.3}, we study the normalized flow  after a proper rescaling.

 For $1\le k\le n-1$, $\a+\delta+\beta<1$ and $\a\le0<\beta$, we prove that the solution to this flow exists for all time and converges smoothly after normalization to a soliton which is a solution of $\psi u^{\a-1}\rho^\delta \sigma_{k}^{\frac{\beta}{k}}=c$ if $\psi$ is a smooth positive function on $\mS^n$ and satisfies the condition that $(D_iD_j\psi^{\frac{1}{1+\beta-\a}}+\delta_{ij}\psi^{\frac{1}{1+\beta-\a}})$ is positive definite, where  $D$ is the Levi-Civita connection of the Riemannian metric $e$ of $\mathbb S^n$. If $1\le k\le n-1$, $\a+\delta+\beta<1$ and $\a>1+\beta>1$, we have the same result with $(D_iD_j\psi^{\frac{1}{1+\beta-\a}}+\delta_{ij}\psi^{\frac{1}{1+\beta-\a}})<0$.
We also have the same result for $k=n$ when $u$ and $\psi$ are even or $\a+\delta+\beta\le1$ without any constraint on positive smooth function $\psi$, which recovers the results in  \cite{CW,IM,HZ,CHZ} and partial results in \cite{CL}. When $\delta=0,\beta=k$, the flow has been studied by Ivaki \cite{IM}. If $\a=0,k=n$, the flow has been studied by Chen and Li \cite{CL}. If $\beta=k$, in \cite{LJL} they studied this flow with $k+1<\delta<1-k-\a$ and specific initial hypersurface. In this paper we give a wider range of coefficients and derive wider conclusions. 

For general function $\psi$, under the flow (\ref{x1.8}), by \cite{DL} Section 2 the support function $u$ satisfies
$$\frac{\partial u}{\partial t}=\psi u^\alpha\rho^\delta \sigma_{k}^{\frac{\beta}{k}}.$$
We mention that $\sigma_{k}=\sigma_{k}(\l)$ in the anisotropic flow, where $\l$ are the principal curvature radii of hypersurface. By Section 2.5, $\l$ are also the eigenvalues of matrix $[D^2u+u\Rmnum{1}]$. Similarly to Theorem \ref{t1.3}, let $\widetilde u(\cdot,\tau)=\varphi^{-1}(t)u(\cdot,t)$, where
\begin{equation*}
	\tau=\frac{\log((1-\alpha-\delta-\beta)\eta t+1)}{(1-\alpha-\delta-\beta)\eta},
\end{equation*}
\begin{equation*}
	\varphi(t)=(C_0+(1-\beta-\delta-\alpha)\eta t)^{\frac{1}{1-\beta-\delta-\alpha}}.
\end{equation*}
We let $C_0$ sufficiently large to make 
\begin{equation*}
\psi\widetilde u^{\alpha-1}\widetilde\rho^\delta\widetilde\sigma_{k}^{\frac{\beta}{k}}\vert_{M_0}>\eta.
\end{equation*}
In fact, $\psi\widetilde u^{\alpha-1}\widetilde\rho^\delta\widetilde\sigma_{k}^{\frac{\beta}{k}}\vert_{M_0}=C_0\psi u^{\alpha-1}\rho^\delta \sigma_{k}^{\frac{\beta}{k}}\vert_{M_0}>\eta$. 

Obviously, $\frac{\p\tau}{\p t}=\varphi^{\a+\delta+\beta-1}$ and $\tau(0)=0$. Then $\widetilde u(\cdot,\tau)$ satisfies the following normalized flow,
\begin{equation}\label{x1.9}
	\begin{cases}
		\frac{\partial \widetilde u}{\p \tau}(x,\tau)=\psi\widetilde u^\alpha\widetilde\rho^\delta\widetilde \sigma_{k}^{\frac{\beta}{k}}(\widetilde\l)-\eta\widetilde u,\\[3pt]
		\widetilde u(\cdot,0)=\widetilde u_0,
	\end{cases}
\end{equation}
where $\eta=(C_n^k)^\frac{\beta}{k}$.
For convenience we still use $t$ instead of $\tau$ to denote the time variable and omit the ``tilde'' if no confusions arise and we mention the scaled flow or normalized flow.

 Next, we prove the following theorem about the asymptotic behavior of the flow.
\begin{theorem}\label{xt1.4}
Let $M_0$ be a smooth, closed and uniformly convex hypersurface in $\mR^{n+1}$, $n\ge2$, enclosing the origin. Suppose $\a,\delta,\beta\in \mR^1$, $\beta>0$ and $\a+\delta+\beta<1$, $k$ is an integer, $\psi$ is a smooth positive function on $\mS^n$. If

(\rmnum{1}) $1\le k<n$, $\a\le0$ and $(D_iD_j\psi^{\frac{1}{1+\beta-\a}}+\delta_{ij}\psi^{\frac{1}{1+\beta-\a}})$ is positive definite;

or (\rmnum{2}) $1\le k<n$, $1+\beta<\a$ and $(D_iD_j\psi^{\frac{1}{1+\beta-\a}}+\delta_{ij}\psi^{\frac{1}{1+\beta-\a}})$ is negative definite;

 or (\rmnum{3}) $k=n$.

 Then the flow (\ref{x1.8}) has a unique smooth and uniformly convex solution $M_t$ for all time $t>0$. After normalization, the rescaled hypersurfaces $\widetilde{M_t}$ converge smoothly to a smooth solution of $\psi u^{\a-1}\rho^\delta\sigma_{k}^\frac{\beta}{k}=\eta$.
\end{theorem}
\textbf{Remark}: (\rmnum{1}) If $\a+\delta+\beta=1$, there may  be no hypersurface satisfying the condition $\widetilde u^{\alpha-1}\widetilde\rho^\delta\widetilde\sigma_{k}^{\frac{\beta}{k}}\vert_{M_0}>\eta$. The $C^0$ estimates or convergence are hard to derive in this case. We give a uniqueness result in Section 7 to show that the proof of this theorem can't be used to deal with this case.

(\rmnum{2})  The condition that $(D_iD_j\psi^{\frac{1}{1+\beta-\a}}+\delta_{ij}\psi^{\frac{1}{1+\beta-\a}})$ is positive definite can be seen in some papers, such as \cite{GX,HMS,IM,SWM}.  In their results, the condition $(D_iD_j\psi^{\frac{1}{1+\beta-\a}}+\delta_{ij}\psi^{\frac{1}{1+\beta-\a}})$ is positive definite or semi-positive definite is  essential if we want to solve the $L^p$ Christoffel-Minkowski problem. In \cite{IM}, for $1\le k<n$, $\delta=0$, $\a<k+1$, $\beta=k$, Ivaki showed the existence of a rotationally symmetric $\psi$ with $((\psi^\frac{1}{k+1-\a})_{\theta\theta}+\psi^\frac{1}{k+1-\a})|_{\theta=0}<0$ and a smooth, closed, strictly convex initial hypersurface for which the solution to the flow (\ref{x1.8}) with $k<n$ will lose smoothness. Therefore $(D_iD_j\psi^{\frac{1}{1+\beta-\a}}+\delta_{ij}\psi^{\frac{1}{1+\beta-\a}})\ge0$ is essential to ensure the smoothness of the solution is preserved if $\a\le \beta+1$. Up to now, all results by use of anisotropic flows to solve Christoffel-Minkowski problem require the condition $(D_iD_j\psi^{\frac{1}{1+\beta-\a}}+\delta_{ij}\psi^{\frac{1}{1+\beta-\a}})>0$. More details can be seen in \cite{BIS2,IM}.

For $k=n$, Gauss curvature has pretty good properties therefore we can wide the range of coefficients. We give another scaling to solve this flow better. Let $\Om_{t}$ be the convex body enclosed by $M_t$. Firstly we mention that the definitions of the support function $u=u(x):\mS^n\rightarrow\mR$ and radial function $\rho=\rho(\xi):\mS^n\rightarrow\mR$ of  a convex body which contains the origin are equal to the definitions of the support function and radial function of the hypersurface enclosing the convex body respectively. Let $\vec{\rho}(\xi):=\rho(\xi)\xi$.
We then introduce two set-valued mappings, the radial Gauss mapping $\mathscr{A}=\sA_\Om$ and the reverse radial Gauss mapping $\sA^*=\sA^*_\Om$, which are given by, for any Borel set $\om\subset\mS^n$,
\begin{align*}
\sA(\om)=&\{\nu(\vec{\rho}(\xi)):\xi\in\om\},\\
\sA^*(\om)=&\{\xi\in\mS^n:\nu(\vec{\rho}(\xi))\in\om\}.
\end{align*}
Note that $\sA(\xi)$ (resp. $\sA^*(x)$) is a unique vector for almost all $\xi\in\mS^n$ (resp. for almost all $x\in\mS^n$) \cite{SR}.

 Let $\widetilde{X}(\cdot,\tau)=e^{-\tau}X(\cdot,t(\tau))$, where $t(\tau)$ is the inverse function of $\tau(t)$ given by
\begin{equation}\label{x1.12}
	\tau=\tau(t)=\begin{cases}
		\frac{1}{q}\log{\int\hspace{-0.9em}-}_{\mS^n} \rho^q_{\Om_{t}}d\mu_{\mS^n}-\frac{1}{q}\log{\int\hspace{-0.9em}-}_{\mS^n} \rho^q_{\Om_{0}}d\mu_{\mS^n},&\quad q\ne0,\\[3pt]
		{\int\hspace{-0.9em}-}_{\mS^n} \log\rho_{\Om_{t}}d\mu_{\mS^n}-{\int\hspace{-0.9em}-}_{\mS^n} \log\rho_{\Om_{0}}d\mu_{\mS^n},&\quad q=0,
	\end{cases}
\end{equation}
while $q=n+1+\frac{n\delta}{\beta}$, and for any measure $d\mu$ and integrable function $g$, we use the convention throughout of the paper
\begin{equation*}
{\int\hspace{-1em}-}_{\mS^n}gd\mu=\frac{\int_{\mS^n}gd\mu}{\int_{\mS^n}d\mu}.
\end{equation*}
It can be seen in Section 8 later on that $\widetilde{X}(\cdot,\tau)$ satisfies the following normalized flow
\begin{equation}\label{x1.13}
	\begin{cases}
		&\frac{\partial X}{\partial t}(x,t)=\frac{\phi(t)\psi u^\alpha\rho^\delta}{K^\frac{\beta}{n}}\nu-X,\\[4pt]
		&X(x,0)=X_0(x),
	\end{cases}
\end{equation}
where $K$ is Gauss curvature,
\begin{equation}\label{x1.14}
\phi(t)=\[\int_{\mS^n}\rho^q(\xi,t)d\xi\]\[\int_{\mS^n}\frac{\psi u^\alpha\rho^{\delta+n\delta/\beta}}{K^{\frac{\beta}{n}+1}}dx\]^{-1}.
\end{equation}

In what follows we use the notations for convenience
\begin{equation*}
dx=d\mu_{\mS^n}(x) \text{ and }d\xi=d\mu_{\mS^n}(\xi).
\end{equation*}
For convenience we still use $t$ instead of $\tau$ to denote the time variable and omit the “tilde” if no confusion arises. Let
\begin{equation}\label{x1.15}
\begin{split}
q&=n+1+\frac{n\delta}{\beta},\qquad
p=1+\frac{n(1-\a)}{\beta},\\[2pt]
q^*&=\begin{cases}
\dfrac{q}{q-n} \quad \text{  if } q>n+1,\\[2pt]
n+1 \quad \text{  if } q=n+1,\\[2pt]
\dfrac{nq}{q-1} \quad \text{  if } 1<q<n+1,\\[2pt]
+\infty \qquad \text{  if } 0<q\le1.
\end{cases}
\end{split}
\end{equation}
 Then we study the normalized flow (\ref{x1.13}) and prove the following theorem.

\begin{theorem}\label{xt1.5}
Let $M_0$ be a smooth, closed and uniformly convex hypersurface in $\mR^{n+1}$, $n\ge2$, enclosing the origin and $\psi$ be a smooth positive function on $\mS^n$. Suppose $\a,\delta,\beta\in \mR^1$, $\beta>0$.

(\rmnum{1}) If $\a+\delta+\beta<1$ or $\a=1-\delta-\beta\ne\frac{\beta}{n}+1$, then the normalized flow (\ref{x1.13}) has a unique smooth solution $X$ for all time $t\ge0$. For each $t\in[0,\infty)$, $M_t=X(\mS^n,t)$ is a closed, smooth, uniformly convex hypersurface and $u(\cdot,t)$, the support function of $M_t=X(\mS^n,t)$, converges smoothly as $t\rightarrow\infty$ to the unique positive, smooth and uniformly convex solution of (\ref{1.11}) with $\psi$ replaced by $c_0\psi^{-n/\beta}$ for some $c_0>0$;

(\rmnum{2}) If $\a=\frac{\beta}{n}+1$, $\delta=\frac{-(n+1)\beta}{n}$, $\int_{\mS^n}  c_0\psi^{-\frac{n}{\beta}}=\vert\mS^n\vert$ and $\int_\om  c_0\psi^{-\frac{n}{\beta}}<\vert\mS^n\vert-\vert\om^*\vert$ for some positive constant $c_0$, then the normalized flow (\ref{x1.13}) has a unique smooth solution $X$ for all time $t\ge0$. For each $t\in[0,\infty)$, $M_t=X(\mS^n,t)$ is a closed, smooth, uniformly convex hypersurface and $u(\cdot,t)$, the support function of $M_t=X(\mS^n,t)$, converges smoothly as $t\rightarrow\infty$ to the unique positive, smooth and uniformly convex solution of (\ref{1.11}) with $\psi$ replaced by $c_0\psi^{-n/\beta}$ for some $c_0>0$;

 (\rmnum{3}) If (1) $\a\le\frac{\beta}{n}+1$ (i.e. $p\ge0$), or (2) $\delta\le\frac{-(n+1)\beta}{n}$ (i.e. $q\le0$), or (3) $\delta>\frac{-(n+1)\beta}{n}$(i.e. $q>0$) and $-q^*<p<0$, where $q^*$ is defined in (\ref{x1.15}). Let $\psi$ be in addition an even function and the initial hypersurface $M_0$ is origin-symmetric, then the normalized flow (\ref{x1.13}) has a unique smooth solution $X$ for all existing time $[0,T^*)$. For each $t\in[0,T^*)$, $M_t=X(\mS^n,t)$ is a closed, smooth, uniformly convex and origin-symmetric hypersurface, and $u(\cdot,t)$, the support function of $M_t=X(\mS^n,t)$, converges smoothly for a sequence of times to a positive, smooth, uniformly convex and even solution of (\ref{1.11}) with $\psi$ replaced by $c_0\psi^{-n/\beta}$ for some $c_0>0$.
\end{theorem}
\textbf{Remark}: (\rmnum{1}) $\a+\delta+\beta\le1$, $\a\le\frac{\beta}{n}+1$, $\delta\le\frac{-(n+1)\beta}{n}$ means $p\ge q$, $p\ge0$, $q\le0$ respectively. If $\a=\frac{\beta}{n}+1$ and $\delta=\frac{-(n+1)\beta}{n}$, we have $p=q=0$. We can always derive the ranges of $p,q$ by the ranges of $\a,\delta,\beta$.

 (\rmnum{2}) In (\rmnum{3}) of Theorem \ref{xt1.5}, $T^*$ is a finite number if $p<q$; $T^*=\infty$ if $p\ge q$. In fact, if $p<q$, by comparing with suitable inner balls, the flow (\ref{x1.8}) with $k=n$ exists only on a finite-time interval. Consider an origin-centered ball $B_{r(0)}$ such that $B_{r(0)}\subset \Om_0$, where $\Om_0$ is a convex body enclosed by $M_0$. Then $B_{r(t)}\subset \Om_t$, where $r(t)=(r(0)^{\frac{\beta}{n}(p-q)}+t\frac{\beta}{n}(p-q)\min\psi)^\frac{n}{\beta(p-q)}$ and $B_{r(t)}$ expands to infinity as $t$ approaches a finite number $T^*$. If $p\ge q$, by comparing with suitable outer balls, the flow exists on $[0,\infty)$. This result covers the main result in \cite{BIS}.

(\rmnum{3}) If $\int_{\mS^n}\psi'=\vert\mS^n\vert$ and $\int_\om\psi'<\vert\mS^n\vert-\vert\om^*\vert$, the existence of Alexsandrov's problem $\det(D^2u+ug_{\mS^n})=(u^2+|Du|^2)^\frac{n+1}{2}u^{-1}\psi'(x)$ had been proved in \cite{AA}.
We point out that $\int_{\mS^n}\psi'=\vert\mS^n\vert$ and $\int_\om\psi'<\vert\mS^n\vert-\vert\om^*\vert$ are necessary for Alexsandrov's problem \cite{AA}.  In this paper we use the generalized solution to the Alexsandrov problem with $\psi'$ replaced by $c_0\psi^{-n/\beta}$ to establish the uniform estimate for the corresponding Gauss curvature flow. Therefore we need to let $\int_{\mS^n}  c_0\psi^{-\frac{n}{\beta}}=\vert\mS^n\vert$ and $\int_\om  c_0\psi^{-\frac{n}{\beta}}<\vert\mS^n\vert-\vert\om^*\vert$. Case (\rmnum{2}) gives a proof for the classical Alexsandrov problem in the smooth category by a curvature flow approach, i.e., it provides the smooth convergence of the flow and an alternative proof for the regularity of the solution.


To prove this theorem, a key observation is that the functional $\cJ_{p,q}$ below is monotone along the normalized flow (\ref{x1.13}),
\begin{equation}\label{x1.16}
\cJ_{p,q}(X(\cdot,t))=U_p(\Om_t)-V_q(\Om_{t}),
\end{equation}
where
\begin{equation*}
U_p(\Om_t)=\begin{cases}
		\frac{1}{p}{\int\hspace{-0.9em}-}_{\mS^n} u^p(x,t)d\mu_{\psi,\beta}(x),&\quad p\ne0,\\[2pt]
{\int\hspace{-0.9em}-}_{\mS^n} \log u(x,t)d\mu_{\psi,\beta}(x),&\quad p=0,
\end{cases}
\end{equation*}
\begin{equation*}
V_q(\Om_{t})=\begin{cases}
		\frac{1}{q}{\int\hspace{-0.9em}-}_{\mS^n} \rho^q(\xi,t)d\xi,&\quad q\ne0,\\[2pt]
		{\int\hspace{-0.9em}-}_{\mS^n} \log\rho(\xi,t)d\xi,&\quad q=0,
	\end{cases}
\end{equation*}
$p,q$ are given by (\ref{x1.15}) and $d\mu_{\psi,\beta}$ is a spherical measure given by
$$d\mu_{\psi,\beta}=\psi^{-n/\beta}d\mu_{\mS^n}(x).$$

Convex geometry plays important role in the development of fully nonlinear partial differential equations. The classical Minkowski problem, the Christoffel-Minkowski problem, the $L^p$ Minkowski problem, the $L^p$ Christoffel-Minkowski problem in general, are beautiful examples of such interactions (e.g., \cite{BC,FWJ2,GM,FWJ3,LE,LO}).

By the theory of convex bodies, to solve the $L^p$ Minkowski problem is equivalent to solve the following PDE:
\begin{equation}\label{1.8}
\sigma_{n}(D^2u+ug_{\mS^n})=u^{p-1}\psi \text{ on } \mS^n,
\end{equation}
where $\psi$ is a function defined on $\mS^n$. After the development of $L^p$-Minkowski problem, it is natural to consider the $L^p$ Christoffel-Minkowski problem, i.e., the problem of prescribing the $k$-th $p$-area measure for general $1\le k\le n-1$ and $p\ge 1$. As before, this problem can be reduced to the following nonlinear PDE:
\begin{equation}\label{1.9}
	\sigma_{k}(D^2u+ug_{\mS^n})=u^{p-1}\psi \text{ on } \mS^n.
\end{equation}

Notice that the solution to $u^{\alpha-1}\rho^\delta f^{-\beta}=\eta$ remains invariant under flow (\ref{1.7}). In summary, Theorem \ref{t1.2} and \ref{t1.3} imply the following uniqueness results of the $L^p$ Minkowski problem and $L^p$ Christoffel-Minkowski problem.
\begin{corollary}\label{t1.4}
Assume $u$ is a positive, smooth, strictly convex solution to (\ref{1.8}) or (\ref{1.9}), then $u\equiv constant$ for $p>1$ with $\psi=1$. Here the strict convexity of a solution, $u$, means that the matrix $[u_{ij}+u\delta_{ij}]$ is positive definite on $\mS^n$.
\end{corollary}
\textbf{Remark}: If $u\equiv constant$, the convex body $\Om$ is ball. At this time the constant can be calculated. Similarly, if $\psi=constant$ and $k\neq p-1$, We can still get the result. For simplicity, we will not describe in detail here. All the theorems below are the same.

The $L^p$-Minkowski problem $(\ref{1.8})$ has been extensively studied during the last twenty years, see \cite{CW,LE,LE2,LO} and see also \cite{SR} for the most comprehensive list of results. When $p\ge1$, the existence and uniqueness of solutions have been proved. For $k<n$, $L^p$-Christoffel-Minkowski problem is difficult to deal with, since the admissible
solution to equation (\ref{1.9}) is not necessary a geometric solution to
$L^p$-Christoffel-Minkowski problem if $k<n$. So, one needs to deal with the convexity of
the solutions of (\ref{1.9}). Under a sufficient condition on the prescribed function, Guan-Ma \cite{GM} proved the existence of a unique convex solution. The key tool to handle the convexity is the constant rank theorem for fully nonlinear partial differential equations. Later, the equation (\ref{1.9}) has been studied by Hu-Ma-Shen \cite{HMS} for $p\ge k+1$ and Guan-Xia \cite{GX} for $1<p<k+1$ and even prescribed data, by using the constant rank theorem. In this paper we give a new proof of the $L^p$-Minkowski problem and $L^p$-Christoffel-Minkowski problem. We choose the appropriate curvature function so that we do not need to use constant rank theorem. We only need the convergence of flows (\ref{1.7}), (\ref{x1.9}) and (\ref{x1.13}) to prove it.

Next we study a characterization problem for the $L^p$ dual curvature measures
introduced recently by Lutwak, Yang and Zhang \cite{LYZ2}. This problem has been proved partially in \cite{LYZ2,CL,HZ,CCL,CHZ,BF} etc..
This problem is reduced to solving a Monge-Ampère type equation
\begin{equation}\label{1.11}
\det(D^2u+ug_{\mS^n})=(u^2+|Du|^2)^\frac{n+1-q}{2}u^{p-1}\psi(x)\quad\text{ on } \mS^n.
\end{equation}

 We shall also consider the $L^p$ dual Christoffel-Minkowski problem. As before, this problem can be reduced to the following nonlinear PDE:
\begin{equation}\label{1.12}
\sigma_{k}(D^2u+ug_{\mS^n})=(u^2+|Du|^2)^\frac{k+1-q}{2}u^{p-1}\psi(x)\quad\text{ on } \mS^n
\end{equation}
with $1\le k\le n-1$, and we want to derive the existence and uniqueness of the convex solution $u$ of (\ref{1.12}). By Theorem \ref{t1.2} and \ref{t1.3}, we have the following existence and uniqueness results of the $L^p$ dual Minkowski problem and $L^p$ dual Christoffel-Minkowski problem.
\begin{corollary}\label{t1.5}
(1) Assume $u$ is a positive, smooth, strictly convex solution to (\ref{1.11}), then $u\equiv constant$ with $\psi=1$ if (\rmnum{1}) $p\ge q$;
or (\rmnum{2}) $1<p<q\le n+1$.

(2) Assume $u$ is a positive, smooth, strictly convex solution to (\ref{1.12}), then $u\equiv constant$ with $\psi=1$ if (\rmnum{1}) $p>1$, $q\le k+1$; or (\rmnum{2}) $p>1$, $q\le p$.
\end{corollary}
\textbf{Remark}: The case  (1) covers Theorem 1.4 in  Chen and Li \cite{CL}. The case $p>q$ of (1) has been proved in Huang and Zhao \cite{HZ}. When $p<q$ and $k=n$, the uniqueness is difficult, and some non-uniqueness results were obtained in \cite{CCL}. In this paper, we give a new proof of the uniqueness with $\psi=1$ for $1<p<q\le n+1$.

In above, we let $\psi=1$ and derive the existence and uniqueness results of the $L^p$ dual Minkowski problem and $L^p$ dual Christoffel-Minkowski problem. More generally, by Theorem \ref{xt1.4} and \ref{xt1.5}, we prove the following existence and uniqueness results of the $L^p$ dual Minkowski problem and $L^p$ dual Christoffel-Minkowski problem with general $\psi$ by the asymptotic behavior of the anisotropic curvature flow.
\begin{theorem}\label{t1.8}
Let $\psi$ be a smooth and positive function on the sphere $\mS^n$.

(\rmnum{1}) If $p> q$ or $p=q\ne0$, there is a unique, positive, smooth, uniformly convex solution to (\ref{1.11});

(\rmnum{2}) If $p=q=0$, $\int_{\mS^n}\psi=\vert\mS^n\vert$ and $\int_\om\psi<\vert\mS^n\vert-\vert\om^*\vert$, there is a unique, positive, smooth, uniformly convex solution to (\ref{1.11});

(\rmnum{3}) If (1) $p\ge0$, or (2) $q\le0$, or (3) $q>0$ and $-q^*<p<0$, where $q^*$ is defined in (\ref{x1.15}). Let $\psi$ be in addition even, then there is an even, positive, smooth and uniformly convex solution to (\ref{1.11});

(\rmnum{4}) If $p> q$, $p>1$, and $(D_iD_j\psi^{\frac{-1}{p+k-1}}+\delta_{ij}\psi^{\frac{-1}{p+k-1}})$ is positive definite, there is a unique smooth, uniformly convex solution to (\ref{1.12}).

(\rmnum{5}) If $q<p<0$, and $(D_iD_j\psi^{\frac{-1}{p+k-1}}+\delta_{ij}\psi^{\frac{-1}{p+k-1}})$ is negative definite, there is a unique smooth, uniformly convex solution to (\ref{1.12}).
\end{theorem}
\textbf{Remark}: (\rmnum{1}) If $p=q$, the solution of (\ref{1.11}) is unique up to dilation.

(\rmnum{2}) Smooth solutions to (\ref{1.11}) was obtained in \cite{HZ} for $p>q$ by elliptic method, both in \cite{CL} for $p=q\ne0$, and in \cite{CHZ} for even $\psi$ and $pq\ge0$ by anisotropic curvature flows. For $q>0$ and $-q^*<p<0$, it was obtained in \cite{CCL} by elliptic method.  In our theorem,  we give a new proof by anisotropic inverse curvature flows.

(\rmnum{3}) If we let $p$ or $q$ be a special value, the $L^p$ dual Minkowski problem and $L^p$ dual Christoffel-Minkowski problem will be converted into $L^p$ Minkowski problem, dual Minkowski problem, $L^p$ Christoffel-Minkowski problem etc.. In fact, if $q=0$ in (\ref{1.11}), there is a unique smooth, uniformly convex solution to $L^p$ Aleksandrov problem $\det(D^2u+ug_{\mS^n})=(u^2+|Du|^2)^\frac{n+1}{2}u^{p-1}\psi(x)$ with $p\ge0$ or $p<0$, $\psi$ is even, which was introduced in \cite{HLY}. If $q=n+1$  in (\ref{1.11}), there is a unique smooth, uniformly convex solution to $L^p$ Minkowski problem $\det(D^2u+ug_{\mS^n})=u^{p-1}\psi(x)$ with $p\ge n+1$ or $p>-(n+1)$, $\psi$ is even.  If $p=0$  in (\ref{1.11}), there is a unique smooth, uniformly convex solution to dual Minkowski problem $\det(D^2u+ug_{\mS^n})=(u^2+|Du|^2)^\frac{n+1-q}{2}u^{-1}\psi(x)$ with $q\le0$ or $\forall q$, $\psi$ is even, which covers Theorem 1.6 in \cite{LSW}. If $q=k+1$  in (\ref{1.12}), there is a unique smooth, uniformly convex solution to $L^p$ Christoffel-Minkowski problem $\sigma_{k}(D^2u+ug_{\mS^n})=u^{p-1}\psi(x)$ with $p>k+1$, which covers the results in \cite{IM}, and it also covers the results in \cite{HMS} if $(D_iD_j\psi^{\frac{-1}{p+k-1}}+\delta_{ij}\psi^{\frac{-1}{p+k-1}})$ is positive definite. In \cite{HMS}, they provided a elliptic proof which only needs $(D_iD_j\psi^{\frac{-1}{p+k-1}}+\delta_{ij}\psi^{\frac{-1}{p+k-1}})\ge0$. Our argument provides a parabolic proof in the smooth category for this problem.

The rest of the paper is organized as follows. We first recall some notations and known results in Section 2 for later use. In Section 3, we establish the a priori estimates for the normalized flow (\ref{1.7}) when $\a+\beta+\delta\le1$, which ensure the long time existence of this flow. In Section 4, we show the convergence of the flow (\ref{1.7}) when $\a+\beta+\delta\le1$, and complete the proof of Theorem \ref{t1.2}. In section 5, we establish the a priori estimates for the original flow (\ref{1.1}) if $M_t\subset B_R(0)$ first. And then we prove (\ref{1.3}).  In Section 6, we show the convergence of the flows (\ref{1.7}) when $\a+\beta+\delta>1$, and complete the proof of Theorem \ref{t1.3}. We prove Theorem \ref{xt1.4} and \ref{xt1.5} in Section 7 and 8 respectively. Finally, in Section 9, we complete the proof of Corollary \ref{t1.4}, \ref{t1.5} and Theorem \ref{t1.8}.

\section{Preliminary}
\subsection{Intrinsic curvature}
We now state some general facts about hypersurfaces, especially those that can be written as graphs. The geometric quantities of ambient spaces will be denoted by $(\bar{g}_{\alpha\beta})$, $(\bar{R}_{\alpha\beta\gamma\delta})$ etc., where Greek indices range from $0$ to $n$. Quantities for $M$ will be denoted by $(g_{ij})$, $(R_{ijkl})$ etc., where Latin indices range from $1$ to $n$.

Let $\nabla$, $\bar\nabla$ and $D$ be the Levi-Civita connection of $g$, $\bar g$ and the Riemannian metric $e$ of $\mathbb S^n$  respectively. All indices appearing after the semicolon indicate covariant derivatives. The $(1,3)$-type Riemannian curvature tensor is defined by
\begin{equation}\label{2.1}
	R(U,Y)Z=\nabla_U\nabla_YZ-\nabla_Y\nabla_UZ-\nabla_{[U,Y]}Z,
\end{equation}
or with respect to a local frame $\{e_i\}$,
\begin{equation}\label{2.2}
	R(e_i,e_j)e_k={R_{ijk}}^{l}e_l,
\end{equation}
where we use the summation convention (and will henceforth do so). The coordinate expression of (\ref{2.1}), the so-called Ricci identities, read
\begin{equation}\label{2.3}
	Y_{;ij}^k-Y_{;ji}^k=-{R_{ijm}}^kY^m
\end{equation}
for all vector fields $Y=(Y^k)$. We also denote the $(0,4)$ version of the curvature tensor by $R$,
\begin{equation}\label{2.4}
	R(W,U,Y,Z)=g(R(W,U)Y,Z).
\end{equation}
\subsection{Extrinsic curvature}
The induced geometry of $M$ is governed by the following relations. The second fundamental form $h=(h_{ij})$ is given by the Gaussian formula
\begin{equation}\label{2.5}
	\bar\nabla_ZY=\nabla_ZY-h(Z,Y)\nu,
\end{equation}
where $\nu$ is a local outer unit normal field. Note that here (and in the rest of the paper) we will abuse notation by disregarding the necessity to distinguish between a vector $Y\in T_pM$ and its push-forward $X_*Y\in T_p\mathbb{R}^{n+1}$. The Weingarten endomorphism $A=(h_j^i)$ is given by $h_j^i=g^{ki}h_{kj}$, and the Weingarten equation
\begin{equation}\label{2.6}
	\bar\nabla_Y\nu=A(Y),
\end{equation}
holds there, or in coordinates
\begin{equation}\label{2.7}
	\nu_{;i}^\alpha=h_i^kX_{;k}^\alpha.
\end{equation}
We also have the Codazzi equation in $\mathbb{R}^{n+1}$
\begin{equation}\label{2.8}
	\nabla_Wh(Y,Z)-\nabla_Zh(Y,W)=-\bar{R}(\nu,Y,Z,W)=0
\end{equation}
or
\begin{equation}\label{2.9}
	h_{ij;k}-h_{ik;j}=-\bar R_{\alpha\beta\gamma\delta}\nu^\alpha X_{;i}^\beta X_{;j}^\gamma X_{;k}^\delta=0,
\end{equation}
and the Gauss equation
\begin{equation}\label{2.10}
	R(W,U,Y,Z)=\bar{R}(W,U,Y,Z)+h(W,Z)h(U,Y)-h(W,Y)h(U,Z)
\end{equation}
or
\begin{equation}\label{2.11}
	R_{ijkl}=\bar{R}_{\alpha\beta\gamma\delta}X_{;i}^\alpha X_{;j}^\beta X_{;k}^\gamma X_{;l}^\delta+h_{il}h_{jk}-h_{ik}h_{jl},
\end{equation}
where
\begin{equation}\label{2.12}
	\bar{R}_{\alpha\beta\gamma\delta}=0.
\end{equation}
\subsection{Hypersurface in $\mathbb{R}^{n+1}$}
It is known that the space form $\mathbb{K}^{n+1}$ can be viewed as Euclidean space $\mathbb{R}^{n+1}$ equipped with a metric tensor, i.e., $\mathbb{K}^{n+1}=(\mathbb{R}^{n+1},ds^2)$ with proper choice $ds^2$. More specifically, let $\mathbb S^n$ be the unit sphere in Euclidean space $\mathbb{R}^{n+1}$ with standard induced metric $dz^2$, then
\begin{equation*}
	\bar{g}:=ds^2=d\rho^2+\phi^2(\rho)dz^2,
	\end{equation*}
where
\begin{equation*}
\phi(\rho)=\rho\text{   in  } \mR^{n+1},
\end{equation*}
 $\rho\in[0,\infty)$.  By \cite{GL}, we have the following lemma.
\begin{lemma}\label{l2.1}
Arbitrary vector field $Y$ satisfies $\bar{\nabla}_YX=Y$.
\end{lemma}
We call the inner product $u:=<X,\nu>$ to be the support function of a hypersurface in $\mathbb{R}^{n+1}$, where $<\cdot,\cdot>=\bar{g}(\cdot,\cdot)$. Then we can derive the gradient and hessian of the support function $u$ under the induced metric $g$ on $M$.
\begin{lemma}\label{l2.2}
	The support function $u$ satisfies
	\begin{equation}\label{2.13}
		\begin{split}
		\nabla_iu=&g^{kl}h_{ik}\nabla_l\Phi,  \\
		 \nabla_i\nabla_ju=&g^{kl}\nabla_kh_{ij}\nabla_l\Phi+\phi'h_{ij}-(h^2)_{ij}u,
		 \end{split}
	\end{equation}
where $(h^2)_{ij}=g^{kl}h_{ik}h_{jl},$ and
\begin{equation*}
	\Phi(\rho)=\int_{0}^{\rho}\phi(r)dr=\frac{1}{2}\rho^2.
\end{equation*}
\end{lemma}
The proof of Lemma \ref{l2.2} can be seen in \cite{GL,BLO,JL}.
\subsection{Graphs in $\mathbb{R}^{n+1}$}
Let $(M,g)$ be a hypersurface in $\mathbb{R}^{n+1}$ with induced metric $g$. We now give the local expressions of the induced metric, second fundamental form, Weingarten curvatures etc., when $M$ is a graph of a smooth and positive function $\rho(z)$ on $\mathbb{S}^n$. Let $\p_1,\cdots,\p_n$ be a local frame along $M$ and $\p_\rho$ be the vector field along radial direction. Then the support function, induced metric, inverse metric matrix, second fundamental form can be expressed as follows (\cite{GL}).
\begin{align*}
	u &= \frac{\rho^2}{\sqrt{\rho^2+|D\rho|^2}},\;\; \nu=\frac{1}{\sqrt{1+\rho^{-2}|D\rho|^2}}(\frac{\p}{\p\rho}-\rho^{-2}\rho_i\frac{\p}{\p x_i}),   \\
	g_{ij} &= \rho^2e_{ij}+\rho_i\rho_j,  \;\;   g^{ij}=\frac{1}{\rho^2}(e^{ij}-\frac{\rho^i\rho^j}{\rho^2+|D\rho|^2}),\\
	h_{ij} &=\(\sqrt{\rho^2+|D\rho|^2}\)^{-1}(-\rho D_iD_j\rho+2\rho_i\rho_j+\rho^2e_{ij}),\\
	h^i_j &=\frac{1}{\rho^2\sqrt{\rho^2+|D\rho|^2}}(e^{ik}-\frac{\rho^i\rho^k}{\rho^2+|D\rho|^2})(-\rho D_kD_j\rho+2\rho_k\rho_j+\rho^2e_{kj}),
	\end{align*}
where $e_{ij}$ is the standard spherical metric. It will be convenient if we introduce a new variable $\gamma$ satisfying $$\frac{d\gamma}{d\rho}=\frac{1}{\rho}.$$
Let $\omega:=\sqrt{1+|D\gamma|^2}$, one can compute the unit outward normal $$\nu=\frac{1}{\omega}(1,-\frac{\gamma_1}{\rho},\cdots,-\frac{\gamma_n}{\rho})$$ and the general support function $u=<X,\nu>=\frac{\rho}{\omega}$. Moreover,
\begin{align}\label{2.14}
	g_{ij} &=\rho^2(e_{ij}+\gamma_i\gamma_j), \;\; g^{ij}=\frac{1}{\rho^2}(e^{ij}-\frac{\gamma^i\gamma^j}{\omega^2}),\notag\\
	h_{ij} &=\frac{\rho}{\omega}(-\gamma_{ij}+\gamma_i\gamma_j+e_{ij}),\notag\\
	h^i_j &=\frac{1}{\rho\omega}(e^{ik}-\frac{\gamma^i\gamma^k}{\omega^2})(-\gamma_{kj}+\gamma_k\gamma_j+e_{kj})\notag\\
	&=\frac{1}{\rho\omega}(\delta^i_j-(e^{ik}-\frac{\gamma^i\gamma^k}{\omega^2})\gamma_{kj}).
	\end{align}
Covariant differentiation with respect to the spherical metric is denoted by indices.

There is also a relation between the second fundamental form and the radial function on the hypersurface. Let $\widetilde{h}=\rho e$. Then
\begin{equation}\label{2.15}
	\omega^{-1}h=-\nn^2\rho+\widetilde{h}
\end{equation}
holds; cf. \cite{GC2}. Since the induced metric is given by
$$g_{ij}=\rho^2 e_{ij}+\rho_i\rho_j, $$
we obtain
\begin{equation}\label{2.16}
	\omega^{-1}h_{ij}=-\rho_{;ij}+\frac{1}{\rho}g_{ij}-\frac{1}{\rho}\rho_{i}\rho_{j}.
\end{equation}

We now consider these flow equations (\ref{1.1}) and (\ref{1.7}) of radial graphs over $\mathbb{S}^n$ in $\mathbb{R}^{n+1}$. It is known (\cite{GC2}) if a closed hypersurface which is a radial graph and satisfies
$$\p_tX=\mathscr{F}\nu,$$
then the evolution of the scalar function $\rho=\rho(X(z,t),t)$ satisfies
\begin{equation}\label{2.17}
\p_t\rho=\mathscr{F}\omega.
\end{equation}
 Equivalently, the equation for $\gamma$ satisfies
 $$\p_t\g=\mathscr{F}\frac{\omega}{\rho}.$$
  Thus we only need to consider these following parabolic initial value problems on $\mathbb{S}^n$,
\begin{equation}\label{2.18}
	\begin{cases}
		\p_t\rho&=u^\alpha\rho^\delta f^{-\beta}\omega, \;\;(z,t)\in\mathbb{S}^n\times [0,\infty),\\
		\rho(\cdot,0)&=\rho_0,
	\end{cases}
\end{equation}
\begin{equation}\label{2.19}
	\begin{cases}
		\p_t\g&=u^\alpha\rho^\delta f^{-\beta}\frac{\omega}{\rho}, \;\;(z,t)\in\mathbb{S}^n\times [0,\infty),\\
		\g(\cdot,0)&=\g_0,
	\end{cases}
\end{equation}
\begin{equation}\label{2.20}
     \begin{cases}
     	\p_t\rho&=(u^\alpha\rho^\delta f^{-\beta}-\eta u)\omega, \;\;(z,t)\in\mathbb{S}^n\times [0,\infty),\\
     	 \rho(\cdot,0)&=\rho_0,
     	\end{cases}
\end{equation}
\begin{equation}\label{2.21}
	\begin{cases}
		\p_t\g&=(u^\alpha\rho^\delta f^{-\beta}-\eta u)\frac{\omega}{\rho}, \;\;(z,t)\in\mathbb{S}^n\times [0,\infty),\\
		\g(\cdot,0)&=\g_0,
	\end{cases}
\end{equation}
   where $\rho_0$ is the radial function of the initial hypersurface.

Next, we can derive a connection between $\vert\nn\rho\vert$ and $\vert D\gamma\vert$.
\begin{lemma}\label{l2.3}\cite{DL2}
If $M$ is a star-shaped hypersurface, we can derive that $\vert\nn\rho\vert^2=1-\frac{1}{\omega^2}$.
\end{lemma}

\subsection{Convex hypersurface parametrized by the inverse Gauss map}
Let $M$ be a smooth, closed and uniformly convex hypersurface in $\mR^{n+1}$. Assume that $M$ is parametrized by the inverse Gauss map $X: \mS^n\to M\subset \mR^{n+1}$ and encloses origin. The support function $u: \mS^n\to \mR^1$ of $M$ is defined by $$u(x)=\sup_{y\in M}<x,y>.$$ The supremum is attained at a point $y=X(x)$ because of convexity, $x$ is the outer normal of $M$ at $y$. Hence $u(x)=<x,X(x)>$. Then we have
\begin{equation*}
X=Du+ux\quad \text{    and    } \quad\rho=\sqrt{u^2+|Du|^2}.
\end{equation*}

Let $e_1,\cdot\cdot\cdot,e_n$ be a smooth local orthonormal frame field on $\mS^n$, and $D$ be the covariant derivative with respect to the standard metric $e_{ij}$ on $\mS^n$. Denote by $g_{ij}, g^{ij}, h_{ij}$ the induced metric, the inverse of the induced metric, and the second fundamental form of $M$, respectively. Then the second fundamental form of $M$ is given by
$$h_{ij}=D_iD_ju+ue_{ij},$$
and the proof can be seen in Urbas \cite{UJ}. To compute the metric $g_{ij}$ of $M$ we use the Gauss-Weingarten relations $D_ix=h_{ik}g^{kl}D_lX$, from which we obtain$$e_{ij}=<D_ix,D_jx>=h_{ik}g^{kl}h_{jm}g^{ms}<D_lx,D_sx>=h_{ik}h_{jl}g^{kl}.$$
Since $M$ is uniformly convex, $h_{ij}$ is invertible and the inverse is denoted by $h^{ij}$, hence $g_{ij}=h_{ik}h_{jk}$. The principal radii of curvature are the eigenvalues of the matrix
\begin{equation}
	b_{ij}=h^{ik}g_{jk}=h_{ij}=D_{ij}^2u+u\delta_{ij}.
\end{equation}
 We see therefore that the support function before scaling satisfies the initial value problem by \cite{DL} Section 2,
\begin{equation}\label{2.24}
	\begin{cases}
		&\frac{\p u}{\p t}=u^\alpha\rho^\delta F^\beta([D^2u+u\uppercase\expandafter{\romannumeral1}])\; \text{ on } \mS^n\times[0,\infty),\\
		&u(\cdot,0)=u_0,
	\end{cases}
\end{equation}
where $\uppercase\expandafter{\romannumeral1}$ is the identity matrix, $u_0$ is the support function of $M_0$, and
\begin{equation}
	F([a_{j}^i])=f(\mu_1,\cdots,\mu_n),
\end{equation}
where $\mu_1,\cdots,\mu_n$ are the eigenvalues of matrix $[a^i_{j}]$. It is not difficult to see that the eigenvalues of $[F^{j}_i]=[\frac{\p F}{\p a^i_{j}}]$ are $\frac{\p f}{\p\mu_1},\cdots,\frac{\p f}{\p\mu_n}$. Thus from Assumption \ref{a1.1} we obtain
\begin{equation}
	[F^{j}_i]>0 \text{ on } \Gamma^+,
\end{equation}
which yields that the equation (\ref{2.24}) is parabolic for admissible solutions. By \cite{DL} Section 2 and (\ref{2.18}), we can find this formula
\begin{equation}\label{2.26}
\frac{\p_t\rho}{\rho}(\xi)=\frac{\p_tu}{u}(x).
\end{equation}
This formula can be also found in \cite{LSW,CL}.

\subsection{Convex body and mixed volume}
Knowledge in this subsection is used to solve (\ref{x1.9}) and (\ref{x1.13}) and we let $M_0$ be a smooth, closed and uniformly convex hypersurface. Let $\sigma_{k}(A)$ be the $k$-th elementary symmetric function defined on the set $\mathcal{M}_n$ of $n\times n$ matrices and $\sigma_{k}(A_1,\cdots,A_k)$ be the complete polarization of $\sigma_{k}$ for $A_i\in\mathcal{M}_n$, $i=1,\cdots,k$, i.e.
$$\sigma_{k}(A_1,\cdots,A_k)=\frac{1}{k!}\sum\tiny_{\begin{array}{c}i_1,\cdots,i_k=1\\j_1,\cdots,j_k=1\end{array}}^{n}\delta^{i_1,\cdots,i_k}_{j_1,\cdots,j_k}(A_1)_{i_1j_1}\cdots(A_k)_{i_kj_k}.$$
If $A_1=\cdots=A_k=A,$ $\sigma_{k}(A_1,\cdots,A_k)=\sigma_{k}(A)$, where $\sigma_{k}(A)$ has been defined in Section 1. Let $\Gamma_k$ be Garding's cone
$$\Gamma_k=\{A\in\mathcal{M}_n:\sigma_{i}(A)>0,i=1,\cdots,k\}.$$

It is well-known that the determinants of the Jacobian of radial Gauss mapping $\sA$ and reverse radial Gauss mapping $\sA^*$ of $\Om$ are given by, see e.g. \cite{HLY2,CCL,LSW},
\begin{equation}\label{2.27}
\vert Jac\sA\vert(\xi)=\vert\frac{dx}{d\xi}\vert=\frac{\rho^{n+1}(\xi)K(\vec{\rho}(\xi))}{u(\sA(\xi))},
\end{equation}
and
\begin{equation}\label{2.28}
\vert Jac\sA^*\vert(x)=\vert\frac{d\xi}{dx}\vert=\dfrac{u(x)}{\rho^{n+1}(\sA^*(x))K(\nu^{-1}_\Om(x))}.
\end{equation}
Before closing this section, we give the following basic properties for any given $\Om$ contained the origin.
\begin{lemma}\cite{CL}\label{l2.5}
Let $\Om$ contain the origin. Let $u$ and $\rho$ be the support function and radial function of $\Om$, and $x_{\max}$ and $\xi_{\min}$ be two points such that $u(x_{\max})=\max_{\mS^n}u$ and $\rho(\xi_{\min})=\min_{\mS^n}\rho$. Then
\begin{align}
\max_{\mS^n}u=\max_{\mS^n}\rho\quad \text{    and    }\quad \min_{\mS^n}u=\min_{\mS^n}\rho,\label{2.29}\\
u(x)\ge x\cdot x_{\max}u(x_{\max}),\quad \forall x\in\mS^n,\label{2.30}\\
\rho(\xi)\xi\cdot\xi_{\min}\le\rho(\xi_{\min}),\quad \forall \xi\in\mS^n.\label{2.31}
\end{align}
\end{lemma}
Let $\Om_{t}$ be the convex body whose support function is $u(\cdot,t)$, $\Om_{t}^*$ be its polar body
$$\Om_{t}^*=\{z\in\mR^{n+1}:z\cdot y\le1 \text{ for all } y\in\Om_{t}\},$$
and $u^*(\cdot,t)$ be the support function of $\Om_{t}^*$. If $\rho(\cdot,t)$ is the radial function of $\Om_{t}$, then
\begin{equation}\label{2.32}
u^*(\xi,t)=\frac{1}{\rho(\xi,t)}.
\end{equation}
It is well-known that $\sA_{\Om_t^*}=\sA_{\Om_t}^*$, see e.g. \cite{HLY2}. This implies $1=\vert Jac\sA_{\Om_t^*}\vert\vert Jac\sA_{\Om_t}\vert$, and so by (\ref{2.27}) and (\ref{2.28}),
\begin{equation}\label{2.33}
\dfrac{u^{n+2}(\sA_{\Om_t}(\xi),t)(u^*(\xi,t))^{n+2}}{K(\vec{\rho}_{\Om_{t}}(\xi))K^*(\nu^{-1}_{\Om_{t}^*}(\xi))}=1, \quad\forall \xi\in\mS^n.
\end{equation}
By (\ref{2.26}) and (\ref{2.32}), we have
\begin{equation}\label{2.34}
\frac{\p_tu^*}{u^*}(\xi,t)=-\frac{\p_tu}{u}(\sA_{\Om_t(\xi)},t).
\end{equation}


\section{The Priori Estimates of the normalized flow with $\a+\delta+\beta\le1$}
In this section, we establish the priori estimates with $\a+\delta+\beta\le1$ and show that the flow exists for long time.
For convenience, we denote that $\Psi=u^\a\rho^\delta,G=f^{-\beta}$, then the equation (\ref{1.1}) can be written in the following form
$$\frac{\p X}{\p t}=u^\alpha\rho^\delta f^{-\beta}(x,t) \nu(x,t)=\Psi G\nu.$$

We first show the $C^0$-estiamte of the solution to (\ref{1.7}) or (\ref{2.20}).
\begin{lemma}\label{l3.1}
	Let $\rho(x,t)$, $t\in[0,T)$, be a smooth, star-shaped solution to (\ref{2.20}). If $\a+\delta+\beta\le1$, then there is a positive constant $C_1$ depending only on $\alpha,\delta,\beta$ and the lower and upper bounds of $\rho(\cdot,0)$ such that
	\begin{equation*}
		\frac{1}{C_1}\leq \rho(\cdot,t)\leq C_1.
	\end{equation*}
\end{lemma}
\begin{proof}
	Let $\rho_{\max}(t)=\max_{z\in \mS^n}\rho(\cdot,t)=\rho(z_t,t)$. For fixed time $t$, at the point $z_t$, we have $$D_i\rho=0 \text{ and } D^2_{ij}\rho\leq0.$$
	Note that $\omega=1$, $u=\frac{\rho}{\omega}=\rho$ and
	\begin{equation}\label{3.1}
		\begin{split}
			h^i_j &=\frac{1}{\rho^2\sqrt{\rho^2+|D\rho|^2}}(e^{ik}-\frac{\rho^i\rho^k}{\rho^2+|D\rho|^2})(-\rho D_kD_j\rho+2\rho_k\rho_j+\rho^2e_{kj})\\
			&=-\rho^{-2}\rho^i_{j}+\rho^{-1}\delta^i_{j}.
		\end{split}
	\end{equation}
At the point $z_t$, we have $F^{-\beta}(h^i_j)\leq\eta(\frac{1}{\rho})^{-\beta}$. Then
$$\frac{d}{dt}\rho_{\max}\leq\eta\rho(\rho^{\alpha+\delta+\beta-1}-1),$$
hence $\rho_{\max}\leq \max\{C_0,\rho_{\max}(0)\}$.
Similarly, $$\frac{d}{dt}\rho_{\min}\geq\eta\rho(\rho^{\alpha+\delta+\beta-1}-1).$$
By $\alpha+\delta+\beta\le1$ we have $\rho_{\min}\geq \min\{\frac{1}{C_0},\rho_{\min}(0)\}$.
\end{proof}

Let $M(t)$ be a smooth family of closed hypersurfaces in $\mR^{n+1}$. Let $X(\cdot,t)$ denote a point on $M(t)$. In general, we have the following evolution property.
\begin{lemma}\label{l3.2}
	Let $M(t)$ be a smooth family of closed hypersurfaces in $\mR^{n+1}$ evolving along the flow$$\p_tX=\mathscr{F}\nu,$$
	where $\nu$ is the unit outward normal vector field and $\mathscr{F}$ is a function defined on $M(t)$. Then we have the following evolution equations.
	\begin{equation}\label{3.2}
		\begin{split}
			\p_tg_{ij}&=2\mathscr{F}h_{ij},\\
			\p_t\nu&=-\nn\mathscr{F},\\
			\p_td\mu_g&=\mathscr{F}Hd\mu_g,\\
			\p_th_{ij}&=-\nn_i\nn_j\mathscr{F}+\mathscr{F}(h^2)_{ij},\\
			\p_tu&=\mathscr{F}-<\nn \Phi,\nn\mathscr{F}>,
		\end{split}
	\end{equation}
	where $d\mu_g$ is the volume element of the metric $g(t)$, $(h^2)_{ij}=h_i^kh_{kj}$.
\end{lemma}
\begin{proof}
	Proof is standard, see for example, \cite{HG}.
\end{proof}

\begin{lemma}\label{l3.3}
	Let $0<\beta\leq1-\alpha-\delta$, and $X(\cdot,t)$ be the solution to the flow (\ref{1.7}) which encloses the origin for $t\in[0,T)$. Then there is a positive constant $C_2$ depending on the initial hypersurface and $\alpha,\delta,\beta$,  such that
	$$\frac{1}{C_2}\leq u^{\alpha-1}\rho^\delta F^{-\beta}\leq C_2.$$
\end{lemma}
\begin{proof}
	We rewrite the flow equation (\ref{2.21}),
\begin{equation}\label{3.3}
\begin{split}
\p_t\gamma=&\(u^\a\rho^\delta F^{-\beta}(\frac{1}{\rho\omega}(\delta^i_j-\gamma^i_{j}+\frac{1}{\om^2}\g^i(\frac{|D\g|^2}{2})_j))-\eta u\)\frac{\omega}{\rho}\\
=&\frac{\rho^{\a+\delta+\beta-1}}{\om^{\a-\beta-1}} F^{-\beta}(a_j^i)-\eta,
\end{split}
\end{equation}
where $a^i_j=\delta^i_j-\gamma^i_{j}+\frac{1}{\om^2}\g^i(\frac{|D\g|^2}{2})_j$. For simplicity, we denote by $F=F(\delta^i_j-\gamma^i_{j}+\frac{1}{\om^2}\g^i(\frac{|D\g|^2}{2})_j)=F(a^i_j)$ and $F^j_i$ is the derivative of $F$ with respect to its argument. At the spatial extremum point $P$ of $\g_t$, we have
$$D\g_t=0,\qquad \p_t\om=0.$$
We then have
$$\p_t\g_t=\frac{\rho^{\a+\delta+\beta-1}}{\om^{\a-\beta-1}}F^{-\beta-1}\((\a+\delta+\beta-1)F\g_t-\beta\p_tF\).$$
	At $P$,
	\begin{align*}
		\p_t F=&F^j_i\p_t[\delta^i_j-\gamma^i_{j}+\frac{1}{\om^2}\g^i(\frac{|D\g|^2}{2})_j]\\
		=&F^j_i[-(\gamma_t)^i_{j}+\frac{1}{\om^2}\g^i\g^k(\g_t)_{kj}]\\
		=&-F^j_i(\delta^{ik}-\frac{1}{\om^2}\g^i\g^k)(\g_t)_{kj}.
	\end{align*}
	Put the above together, we have
	\begin{align}\label{f}
		&\p_t\g_t-\frac{\beta\rho^{\a+\delta+\beta-1}}{\om^{\a-\beta-1}}F^{-\beta-1}F^j_i(\delta^{ik}-\frac{1}{\om^2}\g^i\g^k)(\g_t)_{kj}\notag\\
		=&\frac{\rho^{\a+\delta+\beta-1}}{\om^{\a-\beta-1}}F^{-\beta}(\a+\delta+\beta-1)\g_t.
	\end{align}
	Since we already have uniform $C^0$ bound and $\a+\delta+\beta-1\le0$, then apply maximal principle, $\g_t$ can't blow up to infinity or become negative.  Thus $\g_t$ and $u^{\a-1}\rho^\delta F^{-\beta}=\g_t+\eta$ have uniform upper and lower bounds.
\end{proof}

We would like to get the upper bound of $|D\gamma|$ to show that $u$ has a lower bound.

\begin{lemma}\label{l3.4}
	Let $0<\beta\leq1-\alpha-\delta$, and $X(\cdot,t)$ be the solution to the flow (\ref{1.7}) which encloses the origin for $t\in[0,T)$. Then there is a positive constant $C_3$ depending on the initial hypersurface and $\alpha,\delta,\beta$,  such that
	$$\vert D \g\vert\leq C_3.$$
\end{lemma}
\begin{proof}
	Consider the auxiliary function $O=\frac{1}{2}\vert D\g\vert^2$. At the point where $O$ attains its spatial maximum, we have
	\begin{gather*}
		D\omega=0,\\
		0=D_iO=\sum_{l}\g_{li}\g^{l},\\
		0\geq D_{ij}^2O=\sum_{l}\g^l_{i}\g_{lj}+\sum_{l}\g^{l}\g_{lij},
	\end{gather*}
where $\g^l=\g_ke^{kl},\g_{i}^l=\g_{ik}e^{kl}$. We mention that $\g_{ij}$ is symmetric.	By (\ref{2.14}) and (\ref{2.21}), we deduce
	\begin{equation}\label{3.5}
		\begin{split}
			\p_t\gamma=&(u^\alpha\rho^\delta F^{-\beta}-\eta u)\frac{\omega}{\rho}\\
			=&\frac{\rho^{\a+\delta-1}}{\omega^{\a-1}}F^{-\beta}\(\frac{1}{\rho\omega}(\delta_{j}^i-(e^{ik}-\frac{\gamma^i\gamma^k}{\omega^2})\gamma_{kj})\)-\eta\\
			=&\frac{\rho^{\a+\delta+\beta-1}}{\omega^{\a-\beta-1}}G-\eta.
		\end{split}
	\end{equation}
	We remark that here $G=F^{-\beta}([\delta_{j}^i-(e^{ik}-\frac{\gamma^i\gamma^k}{\omega^2})\gamma_{kj}])$ and $$G^{j}_i=G^{j}_i([\delta_{j}^i-(e^{ik}-\frac{\gamma^i\gamma^k}{\omega^2})\gamma_{kj}]).$$
	Due to (\ref{3.5}), we have
	\begin{align}\label{3.6}
		\p_tO_{\max}=&\sum\g^l\g_{tl}\notag\\
		=&\g^l\(\frac{(\a+\delta+\beta-1)\rho^{\a+\delta+\beta-1}\g_l}{\omega^{\a-\beta-1}}G\\\notag
		&-\frac{\beta\rho^{\a+\delta+\beta-1}F^{-\beta-1}}{\omega^{\a-\beta-1}}F^{j}_i(-\g_{kjl}(e^{ik}-\frac{\gamma^i\gamma^k}{\omega^2}))\),
	\end{align}
	where we have used $\sum_{l}\g_{li}\g^{l}=0$. By the Ricci identity,$$D_l\g_{ij}=D_j\g_{li}+\e_{il}\g_j-\e_{ij}\g_l,$$
	we get
	\begin{align}\label{3.7}
\beta F^{-\beta-1}F^{j}_i\g^l\g_{kjl}(e^{ik}-\frac{\gamma^i\gamma^k}{\omega^2})=&\beta F^{-\beta-1}F^{j}_i\g^l(\g_{lkj}+\g_j\e_{lk}-\g_l\e_{kj})(e^{ik}-\frac{\gamma^i\gamma^k}{\omega^2})\notag\\
		\leq&\beta F^{-\beta-1}F^{j}_i(-\g^l_{k}\g_{lj}+\g^l\g_je_{lk}-\vert D\g\vert^2e_{kj})(e^{ik}-\frac{\gamma^i\gamma^k}{\omega^2})\notag\\
		\leq&\beta F^{-\beta-1}F^{j}_i(-\g^{il}\g_{lj}+\g^i\g_j-\vert D\g\vert^2\delta^i_{j}).
	\end{align}
	Note that $[F^{j}_i]$ is positive definite. According to (\ref{3.6}) and (\ref{3.7}), we can derive
	\begin{align}\label{3.8}
		\p_tO_{\max}\leq&\frac{\rho^{\a+\delta+\beta-1}}{\omega^{\a-\beta-1}}\((\a+\delta+\beta-1)\vert D\g\vert^2G-\beta\vert D\g\vert^2F^{-\beta-1}\sum_{i=1}^{n}F^{ii}\notag\\
		&-\beta F^{-\beta-1}F^{j}_i\g^{li}\g_{lj}+\beta F^{-\beta-1}F^{j}_i\g^i\g_j\).
	\end{align}
	In terms of the positive definition of the symmetric matrix $[F^{j}_i]$, we have $F^{j}_i\g^{li}\g_{lj}\geq0$ and $F^{j}_i\g^i\g_j\leq\max_iF^{ii}\vert D\g\vert^2$. Thus  by $\a+\delta+\beta-1\leq0$, we have
	$$\p_tO_{\max}\leq0.$$
	Then we have $\vert D\g(\cdot,t)\vert\leq C_3$ for a positive constant $C_3$.
\end{proof}
By the above Lemmas, we can get the bounds of $u$ and $f$.
\begin{corollary}\label{c3.5}
	Let $0<\beta\leq1-\alpha-\delta$, and $X(\cdot,t)$ be the solution to the flow (\ref{1.7}) which encloses the origin for $t\in[0,T)$. Then there are positive constants $C_4$ and $C_5$ depending on the initial hypersurface and $\alpha,\delta,\beta$,  such that
	$$\frac{1}{C_4}\leq u\leq C_4\quad\text{and}\quad\frac{1}{C_5}\leq f\leq C_5.$$
\end{corollary}
\begin{proof}
	By Lemma \ref{l3.4} and $u=\frac{\rho}{\omega}\leq\rho_{\max}\leq C_4$, we can derive the bounds of $u$. By Lemma \ref{l3.1}, Lemma \ref{l3.3} and the bounds of $u$, we can get the bounds of $f$.
\end{proof}
The next step in our proof is the derivation of the principal curvature boundary.
\begin{lemma}\label{l3.6}
	Let $\alpha\leq0<\beta\leq1-\alpha-\delta$, and $X(\cdot,t)$ be a smooth, closed and star-shaped solution to the flow (\ref{1.7}) which encloses the origin for $t\in[0,T)$. Then there is a positive constant $C_6$ depending on the initial hypersurface and $\alpha,\delta,\beta$,  such that the principal curvatures of $X(\cdot,t)$ are uniformly bounded from above $$\k_i(\cdot,t)\leq C_6 \text{ \qquad } \forall1\le i\le n,$$
	and hence, are compactly contained in $\Gamma$, in view of Corollary \ref{c3.5}.
\end{lemma}
\begin{proof}
	
We shall prove that $\k_i$  is bounded from above by a positive constant. The principal curvatures of $M_t$ are the eigenvalues of $\{h_{il}g^{lj}\}$.

First, we need the evolution equation of $h_j^i$. Note that $$F^{j}_i=\frac{\p F}{\p h^i_{j}},\quad (F^j_{i})^n_{m}=\frac{\p^2F}{\p h^i_{j}\p h^m_{n}},\quad G^{j}_i=\frac{\p G}{\p h^i_{j}},\quad (G^j_{i})^n_{m}=\frac{\p^2G}{\p h^i_{j}\p h^m_{n}}.$$
For convenience, we denote $F^{j}_i,(F^j_{i})^n_{m}$ by $f^{j}_i,(f^j_{i})^n_{m}$ respectively. From Lemma \ref{l3.2} we can get
\begin{align}\label{h}
	\frac{\p}{\p t}h_j^i=&\frac{\p(g^{ik}h_{kj})}{\p t}\notag\\
	=&g^{ik}\frac{\p h_{kj}}{\p t}+h_{kj}\frac{\p g^{ik}}{\p t}\notag\\
	=&g^{ik}(-\nn_j\nn_k(\Psi G-\eta u)+(\Psi G-\eta u)h^2_{kj})-2(\Psi G-\eta u)(h^2)_{j}^i\notag\\
	=&-g^{ik}(\Psi_{;kj}G+\Psi_{;k}G_{;j}+\Psi_{;j}G_{;k}+\Psi G_{;kj})+\eta g^{ik}u_{;kj}-(\Psi G-\eta u)(h^2)_{j}^i\notag\\
=&-g^{ik}G\(\a\frac{\Psi}{u}u_{;kj}+\a(\a-1)\Psi(\log u)_{;k}(\log u)_{;j}+\delta\frac{\Psi}{\rho}\rho_{;kj}\notag\\
&+\delta(\delta-1)\Psi(\log \rho)_{;k}(\log \rho)_{;j}+\a\delta\Psi((\log u)_{;k}(\log\rho)_{;j}+(\log u)_{;j}(\log\rho)_{;k})\)\notag\\
&+\beta g^{ik}G\(\a\Psi((\log u)_{;k}(\log F)_{;j}+(\log F)_{;k}(\log u)_{;j})+\delta\Psi((\log \rho)_{;k}(\log F)_{;j}\notag\\
&+(\log F)_{;k}(\log \rho)_{;j})\)-\Psi g^{ik}(G^{n}_mh^m_{n;kj}+(G^n_{m})^s_{r}h^m_{n;k}h^r_{s;j})+\eta g^{ik}u_{;kj}\notag\\
&-(\Psi G-\eta u)(h^2)_{j}^i.
\end{align}
Since $(G^n_{m})^s_{r}=-\beta f^{-\beta-1}(f^n_{m})^s_{r}+\beta(\beta+1)f^{-\beta-2}f^n_{m}f^s_{r}$, and by (\ref{2.11}), (\ref{2.12}) and Ricci identity, we have
\begin{equation*}
	\begin{split}
		\nn_j\nn_kh^m_{n}=&g^{ml}\nn_j\nn_kh_{ln}\\
		=&g^{ml}\(h_{kj;ln}+{R_{njk}}^ah_{al}+{R_{njl}}^ah_{ak}\)\\
		=&g^{ml}\(h_{kj;ln}+(h^2)_{ln}h_{jk}-(h^2)_{jl}h_{nk}+(h^2)_{kn}h_{jl}-(h^2)_{kj}h_{ln} \).
	\end{split}
\end{equation*}
Thus,
\begin{equation}\label{x1}
	\begin{split}
\frac{\p h_j^i}{\p t}=&g^{ik}(\eta-\a\frac{ \Psi G}{u})u_{;kj}-\a(\a-1)\Psi G(\log u)^{;i}(\log u)_{;j}-\delta\frac{\Psi G}{\rho}g^{ik}\rho_{;kj}\\
&-\delta(\delta-1)\Psi G(\log \rho)^{;i}(\log \rho)_{;j}-\a\delta\Psi G((\log u)^{;i}(\log\rho)_{;j}+(\log u)_{;j}(\log\rho)^{;i})\\
&+\a\beta\Psi G((\log u)^{;i}(\log F)_{;j}+(\log F)^{;i}(\log u)_{;j})+\delta\beta\Psi G((\log \rho)^{;i}(\log F)_{;j}\\
&+(\log F)^{;i}(\log \rho)_{;j})-\Psi g^{ik}G^n_{m}g^{ml}\(h_{kj;ln}+(h^2)_{ln}h_{jk}-(h^2)_{jl}h_{nk}\\
&+(h^2)_{kn}h_{jl}-(h^2)_{kj}h_{ln}\)-(\Psi G-\eta u)(h^2)_{j}^i\\
&-\Psi g^{ik}\(-\beta f^{-\beta-1}(f^n_{m})^s_{r}+\beta(\beta+1)f^{-\beta-2}f^n_{m}f^s_{r}\)h^m_{n;k}h^r_{s;j}.
	\end{split}
\end{equation}

Next, we define the functions
\begin{gather}\label{3.10}
	W(x,t)=\max\{h_{ij}(x,t)\zeta^i\zeta^j: g_{ij}(x)\zeta^i\zeta^j=1\},\\
	p(u)=-\log(u-\frac{1}{2}\min u),
\end{gather}
and
\begin{equation}\label{3.12}
	\theta=\log W+p(u)+N\rho,
\end{equation}
where $N$ will be chosen later. Note that
\begin{equation}\label{3.13}
	1+p'u=\frac{-\frac{1}{2}\min u}{u-\frac{1}{2}\min u}<0.
\end{equation}
In fact, the auxiliary function in this proof doesn't need to be as complicated as above. We need this calculation in the later lemma, so we construct this auxiliary function here in advance. We hope to bound $\theta$ from above. Thus, suppose $\theta$ attains a maximal value at $(x_0,t_0)\in M\times(0,T_0]$, $T_0<T^*$. Choose Riemannian normal coordinates at $(x_0,t_0)$, such that at this point we have
\begin{equation}\label{3.14}
	g_{ij}=\delta_{ij},\quad h^i_j=h_{ij}=\k_i\delta_{ij},\quad\k_1\geq\cdots\geq\k_n.
\end{equation}
Since $W$ is only continuous in general, we need to find a differentiable version instead. Take a vector $\tilde{\zeta}=(\tilde{\zeta}^i)=(1,0,\cdots,0)$ at $(x_0, t_0)$, and extend it to a parallel vector field in a neighborhood of $x_0$ independent of $t$, still denoted by $\tilde{\zeta}$. Set
$$\widetilde{W}=\frac{h_{ij}\tilde{\zeta}^i\tilde{\zeta}^j}{g_{ij}\tilde{\zeta}^i\tilde{\zeta}^j}.$$

At $(x_0,t_0)$, there holds
\begin{equation}\label{3.15}
	h_{11}=h^1_1=\k_1=W=\widetilde{W}.
\end{equation}
By a simple calculation, we find
\begin{equation*}
	\frac{\p}{\p t}\widetilde{W}=\frac{(\frac{\p}{\p t}h_{ij})\tilde{\zeta}^i\tilde{\zeta}^j}{g_{ij}\tilde{\zeta}^i\tilde{\zeta}^j}-\frac{h_{ij}\tilde{\zeta}^i\tilde{\zeta}^j}{(g_{ij}\tilde{\zeta}^i\tilde{\zeta}^j)^2}\(\frac{\p}{\p t}g_{ij}\)\tilde{\zeta}^i\tilde{\zeta}^j,
\end{equation*}
\begin{equation*}
	\frac{\p}{\p t}h_1^1=\frac{\p}{\p t}(h_{1k}g^{k1})=(\frac{\p}{\p t}h_{1k})g^{k1}-g^{ki}(\frac{\p}{\p t}g_{ij})g^{j1}h_{1k},
\end{equation*}
\begin{equation*}
	\nn_l\widetilde{W}=\frac{(\nn_lh_{ij})\tilde{\zeta}^i\tilde{\zeta}^j}{g_{ij}\tilde{\zeta}^i\tilde{\zeta}^j}
\end{equation*}
and
\begin{equation*}
	\nn_lh_1^1=(\nn_lh_{1k})g^{k1}.
\end{equation*}
By the choice of $\tilde{\zeta}$, we find that at $(x_0,t_0)$
$$\frac{\p\widetilde{W}}{\p t}=\frac{\p h_1^1}{\p t},\quad \nn_l\widetilde{W}=\nn_lh_1^1$$
and the second order spatial derivatives also coincide. In a neighborhood of $(x_0,t_0)$ there holds
$$\widetilde{W}\leq W.$$
This implies $\widetilde{W}$ satisfies the same evolution as $h_1^1$ ($i=j=1$ in (\ref{x1})).
Replacing $\theta$ by $\widetilde{\theta}=\log \widetilde{W}+p(u)+N\rho$, we can assume $\widetilde{W}(x_0,t_0)$ is large enough so that $\widetilde{\theta}$ attains a same maximal value as $\theta$ still at $(x_0,t_0)$, where $\widetilde{W}$ satisfies the same differential equation at this point as $h_1^1$. Thus, without loss of generality, we may treat $h_1^1$ like a scalar and $\theta$ to be given by
\begin{equation}\label{3.16}
	\theta=\log h_1^1+p(u)+N\rho.
\end{equation}
In order to calculate the evolution equations of $\theta$, we should deduce the evolution equations of $h_1^1,$ $u$ and $\rho$. We claim that there is no need to distinguish between the upper and lower indexes in the following calculations (except when one takes $\p_t$ derivative). In fact, at $(x_0,t_0)$, we have
$$h_j^i=g^{ki}h_{kj}=h_{ij},\quad \nn_lh_j^i=g^{ki}\nn_lh_{kj}=\nn_lh_{ij},\quad \nn^i=g^{ik}\nn_k=\nn_i.$$
The second order spatial derivatives also coincide. For ease of reading, we denote $G_i^j$, $(G_i^j)_m^n$, $f_i^j$, $(f_i^j)_m^n$ by $G^{ij}$, $G^{ij,mn}$, $f^{ij}$, $f^{ij,mn}$ respectively. We mention that $(\log u)_{;1}^2=((\log u)_{;1})^2$, $(\log\rho)_{;1}^2=((\log \rho)_{;1})^2$, $(\log F)_{;1}^2=((\log F)_{;1})^2$ in this paper.

By (\ref{x1}), we have
\begin{equation*}
	\begin{split}
		\frac{\p h_1^1}{\p t}=&(\eta-\a\frac{ \Psi G}{u})u_{;11}-\a(\a-1)\Psi G(\log u)_{;1}^2-\delta\frac{\Psi G}{\rho}\rho_{;11}\\
		&-\delta(\delta-1)\Psi G(\log \rho)_{;1}^2-2\a\delta\Psi G(\log u)_{;1}(\log\rho)_{;1}\\
		&+2\a\beta\Psi G(\log u)_{;1}(\log F)_{;1}+2\delta\beta\Psi G(\log \rho)_{;1}(\log F)_{;1}\\
		&+\frac{\beta\Psi G}{F}F^{nm}\(h_{11;mn}+(h^2)_{mn}h_{11}-(h^2)_{11}h_{mn}\)-(\Psi G-\eta u)(h^2)_{11}\\
		&+\beta\frac{\Psi G}{F}(f^n_{m})^s_{r}h^m_{n;1}h^r_{s;1}-\beta(\beta+1)\Psi G(\log F)_{;1}^2.
	\end{split}
\end{equation*}
Due to (\ref{2.13}) and (\ref{2.16}), we have
\begin{equation*}
	\begin{split}
		\frac{\p h_1^1}{\p t}=&\beta\Psi f^{-\beta-1}f^{ij}h_{11,ij}+\rho(\eta-\a \frac{ \Psi G}{u})\rho_{;k}{h_{11;}}^k+(\eta-\a \frac{ \Psi G}{u})h_{11}\\
		&+(\a-\beta-1)\Psi Gh_{11}^2-\a(\a-1)\Psi G(\log u)_{;1}^2-\delta(\delta-2)\Psi G(\log \rho)_{;1}^2\\
&-\delta\frac{\Psi G}{\rho}(\frac{1}{\rho}-\frac{u}{\rho}h_{11})-2\a\delta\Psi G(\log u)_{;1}(\log\rho)_{;1}+2\a\beta\Psi G(\log u)_{;1}(\log F)_{;1}\\
&+2\delta\beta\Psi G(\log \rho)_{;1}(\log F)_{;1}+\frac{\beta\Psi G}{F}F^{ij}(h^2)_{ij}h_{11}+\beta\frac{\Psi G}{F}(f^n_{m})^s_{r}h^m_{n;1}h^r_{s;1}\\
&-\beta(\beta+1)\Psi G(\log F)_{;1}^2.
	\end{split}
\end{equation*}
Define the operator $\mathcal{L}$ by
\begin{equation}\label{3.17}
	\mathcal{L}=\p_t-\beta\Psi f^{-\beta-1}f^{ij}\nn_{ij}^2-\rho(\eta-\a u^{\a-1}\rho^\delta G)\rho_{;k}\nn^k.
\end{equation}
First, we shall estimate the terms involving the first derivative. Note that $(\log\rho)_{;1}$ has uniform bound by Lemma \ref{l2.3} and Lemma \ref{l3.1}. Therefore it is a good choice to convert all the terms involving the first derivative into $(\log\rho)_{;1}$.
We denote
\begin{align*}
 \Rmnum{1}=&-\a(\a-1)\Psi G(\log u)_{;1}^2-\delta(\delta-2)\Psi G(\log \rho)_{;1}^2-2\a\delta\Psi G(\log u)_{;1}(\log\rho)_{;1}\\
 &+2\a\beta\Psi G(\log u)_{;1}(\log F)_{;1}+2\delta\beta\Psi G(\log \rho)_{;1}(\log F)_{;1}-\beta(\beta+1)\Psi G(\log F)_{;1}^2.
\end{align*}
We claim that $\Rmnum{1}\le C(\a,\beta,\delta)\Psi G(\log \rho)_{;1}^2.$ In fact,

(\rmnum{1}) If $\a=0$, by the Cauchy-Schwarz inequality, we have
\begin{align*}
\Rmnum{1}=&-\delta(\delta-2)\Psi G(\log \rho)_{;1}^2+2\delta\beta\Psi G(\log \rho)_{;1}(\log F)_{;1}-\beta(\beta+1)\Psi G(\log F)_{;1}^2\\
\le&\frac{\delta(2\beta+2-\delta)}{\beta+1}\Psi G(\log \rho)_{;1}^2=C(\a,\beta,\delta)\Psi G(\log \rho)_{;1}^2.
\end{align*}
(\rmnum{2}) If $\a<0$, we have $\a(\a-1)>0$. Therefore by the Cauchy-Schwarz inequality, we have
\begin{align*}
2\a\beta\Psi G(\log u)_{;1}(\log F)_{;1}&\le(\a(\a-1)-\epsilon)\Psi G(\log u)_{;1}^2+\frac{\a^2\beta^2}{\a(\a-1)-\epsilon}\Psi G(\log F)_{;1}^2,\\
-2\a\delta\Psi G(\log u)_{;1}(\log\rho)_{;1}&\le\epsilon\Psi G(\log u)_{;1}^2+\frac{\a^2\delta^2}{\epsilon}\Psi G(\log \rho)_{;1}^2,\\
2\delta\beta\Psi G(\log \rho)_{;1}(\log F)_{;1}&\le\epsilon'\Psi G(\log F)_{;1}^2+\frac{\beta^2\delta^2}{\epsilon'}\Psi G(\log \rho)_{;1}^2,
\end{align*}
where $$\epsilon'=\beta(\beta+1)-\frac{\a^2\beta^2}{\a(\a-1)-\epsilon}=\frac{\beta(\a(\a-\beta-1)-\epsilon(\beta+1))}{\a(\a-1)-\epsilon},$$
$0<\epsilon<\a(\a-1)$. Moreover, we can let $\epsilon<\frac{\a(\a-\beta-1)}{\beta+1}$ to ensure $\epsilon'>0$. Note that $\epsilon$ and $\epsilon'$ only depend on $\a,\beta,\delta$. By above, we have
$$\Rmnum{1}\le C(\a,\beta,\delta)\Psi G(\log \rho)_{;1}^2.$$
In summary, for $\forall\a\le0$, we have
\begin{equation}\label{r1}
\Rmnum{1}\le C(\a,\beta,\delta)\Psi G(\log \rho)_{;1}^2.
\end{equation}
Therefore
\begin{equation}\label{3.18}
	\begin{split}
		\cL h_1^1\leq&(\eta-\a \frac{ \Psi G}{u}+\frac{\delta u\Psi G}{\rho^2})h_{11}+(\a-\beta-1)\Psi Gh_{11}^2+(C-\delta)\frac{\Psi G}{\rho^2}\\
		&+\frac{\beta\Psi G}{F}F^{ij}(h^2)_{ij}h_{11}+\beta\frac{\Psi G}{F}(f^n_{m})^s_{r}h^m_{n;1}h^r_{s;1},
	\end{split}
\end{equation}
where we used $\rho_{;1}^2\le|\nn\rho|^2\le1.$
In addition,
\begin{equation*}
	\begin{split}
		\frac{\p u}{\p t}=&\frac{\p}{\p t}<X,\nu>\\
		=&\Psi G-\eta u+(\eta-\a\frac{\Psi G}{u})<X,\nn u>+\beta\frac{\Psi G}{F}<X,\nn F>\\
		&-\delta\Psi G|\nn\rho|^2.
	\end{split}
\end{equation*}
By (\ref{2.13}), we deduce
\begin{equation}\label{3.21}
	\begin{split}
		\cL u=(1-\beta)\Psi G-\eta u+\beta \frac{u\Psi G}{F}F^{ij}(h^2)_{ij}-\delta\Psi G|\nn\rho|^2.
	\end{split}
\end{equation}
By (\ref{2.16}) and (\ref{2.20}), we have
\begin{equation}\label{3.22}
	\begin{split}
		\cL\rho=&(\Psi G-\eta u)\omega-\beta\frac{1}{\rho}\Psi f^{-\beta-1}f^{ij}g_{ij}-\rho(\eta-\a \frac{\Psi G}{u})\vert\nn\rho\vert^2\\
		&+\beta\frac{1}{\rho}\Psi f^{-\beta-1}f^{ij}\rho_{;i}\rho_{;j}+\frac{\beta\Psi G}{\omega}\\
		\le&(\Psi G-\eta u)\omega-\beta\frac{1}{\rho}\Psi f^{-\beta-1}f^{ij}g_{ij}\frac{1}{\om^2}-\rho(\eta-\a \frac{\Psi G}{u})\vert\nn\rho\vert^2+\frac{\beta\Psi G}{\omega},
	\end{split}
\end{equation}
where we used $f^{ij}\rho_{;i}\rho_{;j}\le f^{ij}g_{ij}\vert\nn\rho\vert^2=f^{ij}g_{ij}(1-\frac{1}{\om^2})$.
Note that $\a\leq0$, $\beta>0$, $\1-\a+\beta>2\beta>0$. If $\k_1$ is sufficiently large, the combination of (\ref{3.18}), (\ref{3.21}) and (\ref{3.22}) gives
\begin{equation}\label{3.23}
	\begin{split}
		\cL\theta=&\frac{\cL h_1^1}{h_1^1}+\beta\Psi f^{-\beta-1}f^{ij}\nn_i(\log h_1^1)\nn_j(\log h_1^1)+p'\cL u-p''\beta\Psi f^{-\beta-1}f^{ij}\nn_iu\nn_ju\\
		&+N\cL\rho\\
		\leq&c+(\a-\beta-1)\Psi Gh_{11}+\frac{\beta\Psi G}{F}F^{ij}(h^2)_{ij}+\beta\frac{\Psi G}{F\k_1}(f^n_{m})^s_{r}h^m_{n;1}h^r_{s;1}\\
		&+\beta\Psi f^{-\beta-1}f^{ij}(\nn_i(\log h_1^1)\nn_j(\log h_1^1)-p''\nn_iu\nn_ju)+p'\((1-\beta)\Psi G-\eta u\\
		&+\beta \frac{u\Psi G}{F}F^{ij}(h^2)_{ij}-\delta\Psi G|\nn\rho|^2\)+N\((\Psi G-\eta u)\omega-\beta\frac{1}{\rho}\Psi f^{-\beta-1}f^{ij}g_{ij}\frac{1}{\om^2}\\
		&-\rho(\eta-\a \frac{\Psi G}{u})\vert\nn\rho\vert^2+\frac{\beta\Psi G}{\omega}\)\\
		\leq&c+(\a-\beta-1)\Psi Gh_{11}+\beta(1+p'u)\Psi f^{-\beta-1} f^{ij}(h^2)_{ij}-N\beta\frac{1}{\rho}\Psi f^{-\beta-1}f^{ij}g_{ij}\frac{1}{\om^2}\\
		&+\frac{\beta\Psi f^{-\beta-1}f^{ij,mn}h_{ij;1}h_{mn;1}}{h_{11}}+\beta\Psi f^{-\beta-1}f^{ij}(\nn_i(\log h_1^1)\nn_j(\log h_1^1)-p''\nn_iu\nn_ju),
	\end{split}
\end{equation}
where we used Lemma \ref{l2.3}, Lemma \ref{l3.1} and Corollary \ref{c3.5}. Due to the concavity of $f$ it holds that
\begin{equation}\label{3.24}
	f^{kl,rs}\xi_{kl}\xi_{rs}\leq\sum_{k\neq l}\frac{f^{kk}-f^{ll}}{\k_k-\k_l}\xi_{kl}^2\leq\frac{2}{\k_1-\k_n}\sum_{k=1}^{n}(f^{11}-f^{kk})\xi_{1k}^2
\end{equation}
for all symmetric matrices $(\xi_{kl})$; cf. \cite{GC2}. Furthermore, we have
\begin{equation}\label{3.25}
	f^{11}\leq\cdots\leq f^{nn};
\end{equation}
cf. \cite{EH}. In order to estimate (\ref{3.23}), we distinguish between two cases.

Case 1: $\k_n<-\eps_1\k_1$, $0<\eps_1<\frac{1}{2}$. Then
\begin{equation}\label{3.26}
	f^{ij}(h^2)_{ij}\geq f^{nn}\k_n^2\geq\frac{1}{n}f^{ij}g_{ij}\k_n^2\geq\frac{1}{n}f^{ij}g_{ij}\eps_1^2\k_1^2.
\end{equation}
We use $\nn\theta=0$ to obtain
\begin{equation}\label{3.27}
	f^{ij}\nn_i(\log h_1^1)\nn_j(\log h_1^1)=p'^2f^{ij}u_{;i}u_{;j}+2Np'f^{ij}u_{;i}\rho_{;j}+N^2f^{ij}\rho_{;i}\rho_{;j}.
\end{equation}
In this case, the concavity of $f$ implies that
\begin{equation}\label{3.28}
	\frac{\beta\Psi f^{-\beta-1}f^{ij,mn}h_{ij;1}h_{mn;1}}{h_{11}}\leq0.
\end{equation}
By (\ref{2.13}) and note that $p'<0$, we have
\begin{equation}\label{3.29}
	\begin{split}
		2N\beta p'\Psi f^{-\beta-1}f^{ij}u_{;i}\rho_{;j}=&2N\beta p'\rho\Psi f^{-\beta-1}f^{ij}\k_i\rho_{;i}\rho_{;j}.	
	\end{split}
\end{equation}
By \cite{UJ} Lemma 3.3, we know $\frac{H}{n}\geq \frac{f}{\eta}\geq C_7$, where $H=\sum_{i}^{n}\k_i$ is the mean curvature. Thus $(n-1)\k_1+\k_n\ge H\ge C_7$. We can derive $\k_i\ge\k_n\ge C_7-(n-1)\k_1$. For fixed $i$, if $\k_i\geq0$, we derive $$2N\beta p'\rho\Psi f^{-\beta-1}f^{ij}\k_i\rho_{;i}\rho_{;j}\leq0.$$ If $\k_i<0$, we have
$$ 2N\beta p'\rho\Psi f^{-\beta-1}f^{ij}\k_i\rho_{;i}\rho_{;j}\leq2N\beta p'\rho\Psi f^{-\beta-1}\k_if^{ij}g_{ij} \leq C_{8}(C_7-(n-1)\k_1)f^{ij}g_{ij},$$
where we have used $f^{ij}\rho_{;i}\rho_{;j}\leq f^{ij}g_{ij}\vert\nn\rho\vert^2\leq f^{ij}g_{ij}$.

Without loss of generality, we can assume that $\k_k\ge0$ and $\k_{k+1}\le0$, then
\begin{equation}\label{3.30}
	2N\beta p'\rho\Psi f^{-\beta-1}f^{ij}\k_i\rho_{;i}\rho_{;j}\leq(n-k)C_{8}(C_7-(n-1)\k_1)f^{ij}g_{ij}.
\end{equation}

Since $p'^2=p''$ and $1+p'u<0$, by the combination of Lemma \ref{l2.3}, (\ref{3.23}), (\ref{3.26}), (\ref{3.27}), (\ref{3.28}) and (\ref{3.30}), in this case (\ref{3.23}) becomes
\begin{equation}\label{3.31}
	\begin{split}
		\cL\theta\leq\beta f^{-\beta-1}\Psi f^{ij}g_{ij}\(\frac{1}{n}\eps_1^2\k_1^2(1+p'u)+C_{9}\k_1+C_{10}\)+(\a-\beta-1)\Psi G\k_1+C_{11},
	\end{split}
\end{equation}
which is negative for large $\k_1$. We also use $\a-\beta-1<-2\beta<0$ here.

Case 2: $\k_n\geq-\eps_1\k_1$. Then
$$\frac{2}{\k_1-\k_n}\sum_{k=1}^{n}(f^{11}-f^{kk})(h_{11;k})^2k_1^{-1}\leq\frac{2}{1+\eps_1}\sum_{k=1}^{n}(f^{11}-f^{kk})(h_{11;k})^2k_1^{-2}.$$
We deduce further
\begin{equation}\label{3.32}
	\begin{split}
		f^{ij}&\nn_i(\log h_1^1)\nn_j(\log h_1^1)+\frac{2}{\k_1-\k_n}\sum_{k=1}^{n}(f^{11}-f^{kk})(h_{11;k})^2k_1^{-1}\\
		\leq&\frac{2}{1+\eps_1}\sum_{k=1}^{n}f^{11}(\log h_1^1)_{;k}^2-\frac{1-\eps_1}{1+\eps_1}\sum_{k=1}^{n}f^{kk}(\log h_1^1)_{;k}^2\\
		\leq&\sum_{k=1}^{n}f^{11}(\log h_1^1)_{;k}^2\\
		\leq&f^{11}(p'^2\vert\nn u\vert^2+2Np'<\nn u,\nn\rho>+N^2\vert\nn\rho\vert^2),
	\end{split}
\end{equation}
where we have used $f^{kk}\geq f^{11}$ in the second inequality.
Note that
\begin{gather}
	\beta(1+p'u)\Psi f^{-\beta-1}f^{ij}(h^2)_{ij}\leq\beta(1+p'u)\Psi f^{-\beta-1}f^{11}\k_1^2,\label{3.33}\\
	2Np'f^{11}<\nn u,\nn\rho>=2Np'\rho f^{11}\k_i\rho_{;i}^2\leq-2\eps_1Np'\rho f^{11}\rho_{;i}^2\k_1,\label{3.34}\\
	-\beta\Psi f^{-\beta-1}p''f^{ij}\nn_iu\nn_ju+\beta\Psi f^{-\beta-1}p''f^{11}\vert\nn u\vert^2\leq0\label{3.35}.
\end{gather}
By the combination of (\ref{3.32})---(\ref{3.35}), (\ref{3.23}) becomes
\begin{equation}\label{3.38}
	\begin{split}
		\cL\theta\leq& c+(\a-\beta-1)\Psi G\k_1+\beta\Psi f^{-\beta-1}f^{11}\((1+p'u)k_1^2+C_{12}\k_1+C_{13}\)\\
		&-\beta f^{-\beta-1}f^{ij}g_{ij}\Psi N\frac{1}{\rho\omega^2},
	\end{split}
\end{equation}
which is negative for large $\k_1$.

In conclusion, $\k_i\leq C_6$, where $C_6$ depends on the initial hypersurface, $\a$, $\delta$ and $\beta$. Since $f$ is uniformly continuous on the convex cone $\overline{\Gamma}$, and $f$ is bounded from below by a positive constant, Corollary \ref{c3.5} and Assumption \ref{a1.1} imply that $\k_i$ remains in a fixed compact subset of $\Gamma$, which is independent of $t$.
\end{proof}

The estimates obtained in Lemma \ref{l3.1}, \ref{l3.4}, \ref{l3.6} and Corollary \ref{c3.5} depend on $\a$, $\delta$, $\beta$ and the geometry of the initial data $M_0$. They are independent of $T$. By Lemma \ref{l3.1}, \ref{l3.4}, \ref{l3.6} and Corollary \ref{c3.5}, we conclude that the equation (\ref{2.20}) is uniformly parabolic. By the $C^0$ estimate (Lemma \ref{l3.1}), the gradient estimate (Lemma \ref{l3.4}), the $C^2$ estimate (Lemma \ref{l3.6}) and the Krylov's and Nirenberg's theory \cite{KNV,LN}, we get the H$\ddot{o}$lder continuity of $D^2\rho$ and $\rho_t$. Then we can get higher order derivative estimates by the regularity theory of the uniformly parabolic equations. Hence we obtain the long time existence and $C^\infty$-smoothness of solutions for the normalized flow (\ref{1.7}). The uniqueness of smooth solutions also follows from the parabolic theory. In summary, we have proved the following theorem.
\begin{theorem}\label{t3.7}
Let $M_0$ be a smooth, closed and star-shaped hypersurface in $\mathbb{R}^{n+1}$, $n\geq2$, which encloses the origin. If $\alpha\leq0<\beta\le1-\alpha-\delta$, the normalized flow (\ref{1.7}) has a unique smooth, closed and star-shaped solution $M_t$ for all time $t\geq0$. Moreover, the radial function of $M_t$ satisfies the a priori estimates
$$\parallel\rho\parallel_{C^{k,\beta}(\mS^n\times[0,\infty))}\leq C,$$
where the constant $C>0$ depends only on $k,\a,\beta$ and the geometry of $M_0$.
\end{theorem}

\section{Proof of Theorem \ref{t1.2}}

In this section, we prove the asymptotical convergence of solutions to the normalized flow (\ref{1.7}). By Theorem \ref{t3.7} it is known that the flow (\ref{1.7}) exists for all time $t>0$ and remains smooth and star-shaped, provided $M_0$ is smooth, star-shaped and encloses the origin. In Section 3, we have the bound of $\rho$, $\vert D\g\vert$, $\k_i$ and $f$. It then follows by (\ref{3.8}) that $\p_tO_{\max}\leq-C_0O_{\max}$ for some positive constant $C_0$, where $O=\frac{1}{2}\vert D\g\vert^2$. This proves
\begin{equation}\label{4.1}
	\max_{\mS^n}\vert D\g\vert^2\leq Ce^{-C_0t},\;\;\forall t>0,
\end{equation}
for both $C$ and $C_0$ are positive constants. Meanwhile, according to the bound of $\rho$, we can derive
\begin{equation}\label{4.2}
	\max_{\mS^n}\vert D\rho\vert^2\leq C'e^{-C_0t},\;\;\forall t>0.
\end{equation}

\begin{proof of theorem 1.2}
	
	By (\ref{4.2}), we have that $\vert D\rho\vert\to0$ exponentially as $t\to\infty$. Hence by the interpolation and the a priori estimates, we can get that $\rho$ converges exponentially to a constant in the $C^\infty$ topology as $t\to\infty$.
	
\end{proof of theorem 1.2}

\section{The Priori Estimates of original flows and proof of (\ref{1.3})}
In this section, we establish the priori estimates of original flow (\ref{1.1}) when $\a+\delta+\beta>1$ and show that (\ref{1.3}) is correct.

We first look at the flow of geodesic spheres. Spheres centered at the origin with radius $\rho$ are umbilical, their second fundamental form is given by $\bar{h}_{ij}=\frac{1}{\rho}\bar{g}_{ij}$ and $D\rho=0$. Hence, the flow equation (\ref{1.1}) can be reduced to
\begin{equation}\label{5.1}
\p_t\rho=\rho^{\a+\delta+\beta}\eta.
\end{equation}
Thus, in case $\a+\beta+\delta\neq1$,
\begin{equation}\label{5.2}
\rho=\((1-\a-\delta-\beta)\eta t+\rho_0^{1-\a-\delta-\beta}\)^{\frac{1}{1-\a-\delta-\beta}},
\end{equation}
where $\rho(0)=\rho_0,$ and we conclude
\begin{remark}\label{r5.1}
If the initial hypersurface is a sphere centered at the origin, the flow (\ref{5.1}) converges to infinity in case $\a+\delta+\beta>1$. The flow blows up in finite time
\begin{equation}\label{5.3}
T^*=\frac{\rho_0^{1-\a-\delta-\beta}}{(\a+\delta+\beta-1)\eta}.
\end{equation}
\end{remark}
As a corollary we obtain:
\begin{corollary}\label{c5.2}
Let $M_0=\text{graph}$ $\rho_0$ be star-shaped and let $X=X(x,t)$ be a solution of flow (\ref{1.1}), and define $\rho(x,t)$ by $M(t)=\text{graph}$ $\rho(t)$. Let $r_1,r_2$ be positive constants such that
\begin{equation*}
r_1<\rho_0(x)<r_2 \qquad \forall x\in\mS^n,
\end{equation*}
then $\rho(t)$ satisfies the estimates
\begin{equation}\label{5.4}
\Theta(r_1,t)<\rho(x,t)<\Theta(r_2,t) \quad \forall0\le t<\min\{T^*,T^*(r_1),T^*(r_2)\},
\end{equation}
where $$\Theta(r,t)=\((1-\a-\delta-\beta)\eta t+r^{1-\a-\delta-\beta}\)^{\frac{1}{1-\a-\delta-\beta}}$$ and where $T^*(r_i)$ indicates the maximal time for which the spherical flow with initial sphere of radius $r_i$ will exist.
\end{corollary}
\begin{proof}
The spheres with radii $\Theta(r_i,t)$ are the spherical solutions of the flow (\ref{1.1}) with initial spheres of radius $r_i$.

The flows $X=X(x,t)$ also satisfy the same scalar flow equations by (\ref{2.18}). Hence, the result is due to the maximum principle, since these are parabolic equations.
\end{proof}

In view of the estimates (\ref{5.4}) the maximal time $T^*$ has to be finite. We shall prove uniform estimates for the second fundamental form and the maximal time $T^*$ is indeed characterized by a blow up of the flows.

Let us start with the latter result in this section.
\begin{theorem}\label{t5.3}
Let $\a+\beta+\delta>1$ and let the initial hypersurface $M_0$ be strictly convex, then the solution of the curvature flow (\ref{1.1}) exists for $[0,T^*)$. The flow (\ref{1.1}) satisfies uniform estimates as long as it stays in a compact subset $\bar{\Om}\subset\mR^{n+1}$. Moreover, the principal curvatures are strictly positive
$$0<C_{14}\le\k_i\le C_{15}\qquad \forall1\le i\le n,$$
where the constants $C_{14},C_{15}$ depends on $\bar{\Om}$, $\a,\beta,\delta$ and $M_0$. Hence the singularity $T^*$ is characterized to be the blow up time of the flow. Then
$$\lim_{t\rightarrow T^*}\sup_{M_t}|X|=\infty.$$
\end{theorem}
For the proof we need several lemmas.
\begin{lemma}\label{l5.4}
Let $M$ be a closed, strictly convex hypersurface which is represented as the graph of $\rho$ over $\mS^{n}$. Assume that $\rho$ is bounded by
$$0<r_1\le\rho\le r_2,$$
then
$$\om\le C(r_1,r_2).$$
\end{lemma}
The lemma was proved in \cite{GC2} Theorem 2.7.10.
\begin{lemma}\label{l5.5}
 Let $\a+\beta+\delta>1$ and let $M(t)$ be a solution of the flow (\ref{1.1}). Let $0<T<T^*$ be arbitrary and assume
 $$0<r_1\le\rho\le r_2 \qquad \forall0\le t\le T,$$
 then there exists a constant $C_{16}$ such that
 $$F\le C_{16}\qquad \forall0\le t\le T,$$
 where $C_{16}=C(r_1,r_2,\a,\beta,\delta,M_0)$ is independent of $T$.
\end{lemma}
\begin{proof}
The proof is almost identical to the proof of Lemma \ref{l3.3}, the only difference is that this time we omit $-\eta$ in evolution equation (\ref{3.3}). By Lemma \ref{l5.4} we have the bounds of $u$ if $0<r_1\le\rho\le r_2$. Note that at this time $\g_t>0$. By (\ref{f}), we have
$$\p_t\g_t-\frac{\beta\rho^{\a+\delta+\beta-1}}{\om^{\a-\beta-1}}F^{-\beta-1}F^j_i(\delta^{ik}-\frac{1}{\om^2}\g^i\g^k)(\g_t)_{kj}\ge0.$$
By the maximum principle, we have the lower bound of $\g_t$ and the upper bound of $F$.
\end{proof}
\begin{lemma}\label{l5.6}
	Let $\a+\beta+\delta>1$ and let $M(t)$ be a solution of the flow (\ref{1.1}). Let $0<T<T^*$ be arbitrary and assume
	$$0<r_1\le\rho\le r_2 \qquad \forall0\le t\le T,$$
	then there exists a constant $C_{17}$ such that
	$$F\ge C_{17}\qquad \forall0\le t\le T,$$
	where $C_{17}=C(r_1,r_2,\a,\beta,\delta,M_0)$ is independent of $T$.
\end{lemma}
\begin{proof}
Let $0<T<T^*$ be arbitrary. Consider the auxiliary function
$$Q=u^{\alpha-1}\rho^\delta F^{-\beta}.$$
Then $Q=u^{\alpha-1}\rho^\delta G$ and $G$ is homogenous of degree $-\beta$.
$$(u^{\alpha}\rho^\delta G)_{ij}=(Qu)_{ij}=Q_{ij}u+Q_iu_j+Q_ju_i+Qu_{ij}.$$
In order to calculate the evolution equation of $Q$, we need to deduce the evolution equations of $u$ and $F$ first.
	\begin{align}\label{5.5}
	\frac{\p u}{\p t}&=\frac{\p}{\p t}<X,\nu>\notag\\
	&=<\bar{\nabla}_{\frac{\p X}{\p t}}X,\nu>+<X,\frac{\p \nu}{\p t}>\notag\\
	&=Qu-<X,\nn(Qu)>\notag\\
	&=Qu-Q<X,\nn u>-u<X,\nn Q>.
\end{align}
\begin{equation}\label{5.6}
	\begin{split}
		\frac{\p F}{\p t}=&F^{j}_i\frac{\p h^i_j}{\p t}\\
		=&F^{j}_i\(-\nn^i\nn_j(Qu)+Qu(h^2)^i_j-2Qu(h^2)^i_j\)\\
		=&-F^{j}_i\(u\nn^i\nn_jQ+2\nn^iQ\nn_ju+Q\nn^i\nn_ju+Qu(h^2)^i_j\).
	\end{split}
\end{equation}
By Lemma \ref{l2.2}, (\ref{2.18}), (\ref{5.5}), and (\ref{5.6}), we can deduce the evolution equation of $Q$.
\begin{equation}
\begin{split}\label{5.7}
\p_tQ=&\p_t(u^{\alpha-1}\rho^\delta G)\\
=&(\alpha-1)\frac{Q}{u}\frac{\p u}{\p t}+\delta\frac{Q}{\rho}\frac{\p\rho}{\p t}-\beta \frac{Q}{F}\frac{\p F}{\p t}\\
=&(\alpha+\delta-1)Q^2-(\a-1)\frac{Q^2}{u}<X,\nn u>-(\a-1)Q<X,\nn Q>\\
&+\beta\frac{Q}{F}F^{j}_i\(u\nn^i\nn_jQ+2\nn^iQ\nn_ju+Q\nn^i\nn_ju+Qu(h^2)^i_j\)\\
=&(\alpha+\delta+\beta-1)Q^2-(\a-1)\frac{Q^2}{u}<X,\nn u>-(\a-1)Q<X,\nn Q>\\
&+\beta\frac{Q}{F}\(uF^{j}_i\nn^i\nn_jQ+2F^{j}_i\nn^iQ\nn_ju+Q<X,\nn F>\).
\end{split}
\end{equation}
Looking at the spatial maximum point $p$ for the function
$$\theta=\log Q+N\rho.$$
We have
\begin{align}
\nn\theta=\frac{\nn Q}{Q}+N\nn\rho=&(\a-1)\frac{\nn u}{u}+\delta\frac{\nn\rho}{\rho}-\beta\frac{\nn F}{F}+N\nn\rho=0,\label{5.8}\\
\nn^2_{ij}\theta=&\frac{\nn^2_{ij}Q}{Q}-\frac{\nn_iQ\nn_jQ}{Q^2}+N\nn^2_{ij}\rho.\label{5.9}
\end{align}
By (\ref{2.16}), (\ref{5.7}), (\ref{5.8}) and (\ref{5.9}), we have
\begin{equation}
\begin{split}
(\p_t-\beta\frac{uQ}{F}F^{j}_i\nn^i\nn_j)\theta=&(\alpha+\delta+\beta-1+N\rho)Q-(\a-1)<X,\nn Q>\\
&+(\delta+N\rho)Q|\nn\rho|^2-\frac{2\beta NQ}{F}F^{j}_i\nn^i\rho\nn_ju\\
&+\frac{\beta u}{F}(F^{j}_i\frac{\nn^i Q\nn_j Q}{Q}-NQF^{j}_i\nn^i\nn_j\rho)\\
\le&(\alpha+\delta+\beta-1+N\rho)Q-(\a-1)<X,\nn Q>\\
&+(\delta+N\rho)Q|\nn\rho|^2-\frac{2\beta NQ}{F}F^{j}_i\nn^i\rho\nn_ju\\
&+\frac{\beta uQ}{F}(N^2F^{j}_i\nn^i\rho\nn_j\rho-N(\frac{1}{\rho\om^2}F^{i}_i-\frac{u}{\rho}F)),
\end{split}
\end{equation}
where we have used $F^{j}_i\nn^i\rho\nn_j\rho\leq F^{i}_i\vert\nn\rho\vert^2=(1-\frac{1}{\om^2})F^{i}_i$ in the last inequality. $-\frac{2\beta NQ}{F}F^{j}_i\nn^i\rho\nn_ju$ is nonpositive by Lemma \ref{l2.2} since the leaves $M(t)$ are supposed to be strictly convex. Since we want to get a lower bound of $F$, the terms involving $\frac{Q}{F}$ is crucial. In fact, $\frac{Q}{F}$ is of the order $F^{-(\beta+1)}$, and $Q$ is of the order $F^{-\beta}$. By assumption, $\rho$, $u$ and $|\nn\rho|^2$ have uniform bounds. Note that
$$-(\a-1)<X,\nn Q>=(\a-1)NQ\rho|\nn\rho|^2.$$
Therefore we shall choose $N$ to make $\frac{N\beta uQ}{F}(NF^{j}_i\nn^i\rho\nn_j\rho-\frac{1}{\rho\om^2}F^{i}_i)$ negative.
\begin{align*}
NF^{j}_i\nn^i\rho\nn_j\rho-\frac{1}{\rho\om^2}F^{i}_i\le& F^{i}_i(N|\nn\rho|^2-\frac{1}{\rho\om^2})\\
\le&F^{i}_i(N-\frac{1}{\rho\om^2}).
\end{align*}
We can let $N=\frac{1}{2\rho_{\max}\om^2_{\max}}$, then $\frac{N\beta uQ}{F}(NF^{j}_i\nn^i\rho\nn_j\rho-\frac{1}{\rho\om^2}F^{i}_i)\le -C\frac{Q}{F}$ by $\sum_{i}F_i^i\ge1$. By the maximum principle we can obtain the lower bound of $F$.
\end{proof}
Next, we can prove Theorem \ref{t5.3}.

\begin{proof of theorem 5.3}

The proof is similar to the proof of Lemma \ref{l3.6}. The auxiliary functions, etc., are the same as Lemma \ref{l3.6}.
\begin{align}\label{5.11}
	\frac{\p}{\p t}h_j^i=&\frac{\p(g^{ik}h_{kj})}{\p t}\notag\\
	=&g^{ik}\frac{\p h_{kj}}{\p t}+h_{kj}\frac{\p g^{ik}}{\p t}\notag\\
	=&g^{ik}(-\nn_j\nn_k(\Psi G)+\Psi Gh^2_{kj})-2\Psi G(h^2)_{j}^i.
\end{align}
We can find that only $\eta\nn_j\nn^iu+\eta u(h^2)_{j}^i$ in (\ref{h}) is not in (\ref{5.11}) . Therefore, by (\ref{3.18}), we have
\begin{equation}\label{5.12}
	\begin{split}
		\cL h_1^1\leq&(-\a \frac{ \Psi G}{u}+\frac{\delta u\Psi G}{\rho^2})h_{11}+(\a-\beta-1)\Psi Gh_{11}^2+(C-\delta)\frac{\Psi G}{\rho^2}\\
		&+\frac{\beta\Psi G}{F}F^{ij}(h^2)_{ij}h_{11}+\beta\frac{\Psi G}{F}(f^n_{m})^s_{r}h^m_{n;1}h^r_{s;1},
	\end{split}
\end{equation}
Note that the difference between $\mathcal{L}$ here and Lemma \ref{l3.6} is the first derivative. Similarly,
\begin{align}\label{5.13}
	\cL u=&(1-\beta)\Psi G+\beta \frac{u\Psi G}{F}F^{ij}(h^2)_{ij}-\delta\Psi G|\nn\rho|^2,\\
	\cL\rho\le&\Psi G\omega-\beta\frac{1}{\rho}\Psi f^{-\beta-1}f^{ij}g_{ij}\frac{1}{\om^2}+\a \frac{\rho\Psi G}{u}\vert\nn\rho\vert^2+\frac{\beta\Psi G}{\omega}.\label{5.14}
\end{align}
By (\ref{5.12}), (\ref{5.13}) and (\ref{5.14}) we have
\begin{equation}\label{5.15}
	\begin{split}
		\cL\theta\leq&c+(\a-\beta-1)\Psi Gh_{11}+\beta(1+p'u)\Psi f^{-\beta-1} f^{ij}(h^2)_{ij}-N\beta\frac{1}{\rho}\Psi f^{-\beta-1}f^{ij}g_{ij}\frac{1}{\om^2}\\
		&+\frac{\beta\Psi f^{-\beta-1}f^{ij,mn}h_{ij;1}h_{mn;1}}{h_{11}}+\beta\Psi f^{-\beta-1}f^{ij}(\nn_i(\log h_1^1)\nn_j(\log h_1^1)-p''\nn_iu\nn_ju),
	\end{split}
\end{equation}
The rest is the same as Lemma \ref{l3.6}. Note that $\Gamma=\Gamma_+$ here. Therefore the principal curvatures satisfy
$$0<C_{14}\le\k_i\le C_{15}\qquad \forall1\le i\le n.$$
Thus (\ref{1.1}) can be rewritten as a uniformly parabolic equation and the estimates of  Safonov,  Evans, Schauder and Krylov's and Nirenberg's theory \cite{KNV,LN}, imply uniform a priori estimates up to $t=T^*$; for details can see e.g. \cite{GC2} Chapter 2.6. This allows to continue the solution $M_t$ smoothly past $t=T^*$, a contradiction. Note finally, that (\ref{1.1}) is an expanding flow, so $\sup_{M_t}|X|$ is monotone in $t$ and the limit exists.
\end{proof of theorem 5.3}

At the end of this section, we need to prove (\ref{1.3}).
\begin{lemma}\label{l5.7}
Let $\a,\delta\le0<\beta$ and $M_0$ be uniformly convex. Let $u(\cdot,t)$ be a positive, smooth and uniformly convex solution to (\ref{2.24}). Then there exists a positive constant $C_{18}$ depending only on $u_0,\a,\beta,\delta$, such that
\begin{equation}\label{5.16}
\rho(x,t)-C_{18}\le\Theta(\rho_0,t)\le \rho(x,t)+C_{18}\qquad \forall x\in\mS^n,
\end{equation}
where $\rho_0$ is defined by (\ref{5.3}). Hence
\begin{equation}\label{5.17}
\lim_{t\rightarrow T^*}\rho(x,t)\Theta^{-1}(\rho_0,t)=1\qquad \forall x\in\mS^n.
\end{equation}
\end{lemma}
\begin{proof}
Firstly, we shall apply the Aleksandrov's reflection principle (see, e.g. \cite{CG,CL}) to derive $\osc u\le C_{19}.$

Fixing a direction $\vec{a}\in\mS^n$, let us consider the reflection
\begin{equation}\label{5.18}
\widehat{u}(x,t)=u(\widehat{x},t),\qquad \text{where }\widehat{x}=x-2<x,\vec{a}>\vec{a}.
\end{equation}
Let $\Om_t$ and $\widehat{\Om}_t$ be the convex bodies whose support functions are respectively $u(\cdot,t)$ and $\widehat{u}(\cdot,t)$. Given $\l>0$, we define $u_\l(\cdot,t):\mS^n\rightarrow \mR$ via
\begin{equation}\label{5.19}
u_\l(x,t)=\widehat{u}(x,t)+\l<x,\vec{a}>.
\end{equation}
This also means $X_\l(x,t)=\widehat{X}(x,t)+\l\vec{a}.$ Denote by $\Om_t^\l$ the convex body whose support function is $u_\l(\cdot,t)$. Then $\Om_t^\l$ is a translation of $\widehat{\Om}_t$. Clearly $\Om_t^\l$ and $\Om_t$ are symmetric with respect to the hyperplane $P_\l=\{z\in\mR^{n+1}:z\cdot\vec{a}=\l/2\}$. Let
$$P^+_\l=\{z\in\mR^{n+1}:z\cdot\vec{a}\ge\l/2\}\quad\text{ and      }\quad P^-_\l=\{z\in\mR^{n+1}:z\cdot\vec{a}\le\l/2\}.$$
Since the initial data $\Om_0$ is compact, there is a $\l=\l_{u_0}>0$, depending only on $u_0$ (independent of $\vec{a}$), such that
\begin{equation}\label{o}
\Om_0\subset\text{Int}P^-_\l\quad\text{ and      }\quad\Om_0^\l\subset\text{Int}P^+_\l.
\end{equation}
Consequently
$$u_{\l}(x,0)\ge u(x,0),\quad\forall x\in\mS^n_+:=\{y\in\mS^n:\langle y,\vec{a}\rangle \ge0\},$$
and the equality holds only on $\p\mS^n_+$. We claim that
\begin{equation}\label{5.20}
u_{\l}(x,t)\ge u(x,t),\quad\forall(x,t)\in\mS^n_+\times[0,T).
\end{equation}
Note that $u_{\l}(x,t)=u(x,t)$ for all $(x,t)\in\p\mS^n_+\times[0,T).$

We give the proof of claim (\ref{5.20}). Let $o'=\l\vec{a}$. Denote by $u_{\l,o'}(\cdot,t)$ the support function of $\Om_t^\l$ with respect to the centre $o'$. Then
\begin{equation}\label{5.21}
u_{\l,o'}(\cdot,t)=\widehat{u}(\cdot,t).
\end{equation}
By (\ref{5.18}), (\ref{5.19}) and (\ref{5.21}), it is not hard to check that
\begin{equation}\label{5.22}
\begin{split}
\p_tu_\l(x,t)=&\p_t\widehat{u}(x,t)=\widehat{u}^\a\left(\sqrt{\widehat{u}^2+|D\widehat{u}|^2}\right)^\delta F^\beta(D^2\widehat{u}+\widehat{u}\Rmnum{1})\\
=&u_{\l,o'}^\a\left(\sqrt{u_{\l,o'}^2+|Du_{\l,o'}|^2}\right)^\delta F^\beta(D^2u_{\l,o'}+u_{\l,o'}\Rmnum{1}) \quad \text{ on }\mS^n\times[0,T).
\end{split}
\end{equation}
In fact, $\Om_t^\l$ is a translation of $\widehat{\Om}_t$. Therefore $D\widehat{u}=Du_{\l,o'}$ and $D^2\widehat{u}=D^2u_{\l,o'}$. Let $\mathcal{M}^+_t=\nu^{-1}_{\Om_t}(\mS^n_+)$ and $\mathcal{M}^{\l,+}_t=\nu^{-1}_{\Om^\l_t}(\mS^n_+)$. For $x_0\in\text{Int}\mS^n_+$, let $t_0\in[0,T)$  be the time such that
\begin{align*}
(\rmnum{1})\quad & \nu^{-1}_{\Om_{t_0}}(x_0)=\nu^{-1}_{\Om_{t_0}^\l}(x_0)=:z_0;\\
(\rmnum{2})\quad & \nu^{-1}_{\Om_{t}}(x_0)\neq\nu^{-1}_{\Om_{t}^\l}(x_0)\quad \forall t\in[0,t_0), \text{ and } u_\l(\cdot,t_0)\ge u(\cdot,t_0) \text{ near }x_0.
\end{align*}
If for each $x_0\in\text{Int}\mS^n_+$, no such $t_0$ exists, then one infers by (\ref{o}) that $\text{Int}\mathcal{M}^+_t\cap\text{Int}\mathcal{M}^{\l,+}_t=\emptyset$ remains for all $t\in[0, T)$, and therefore (\ref{5.20}) follows immediately. Suppose $t_0$ exists. Then
\begin{equation}\label{5.24}
D^2u_\l(x_0,t_0)\ge D^2u(x_0,t_0).
\end{equation}
By the symmetry, it is easy to see that $z_0\in P^+_\l$. Hence
\begin{equation}\label{5.25}
|z_0-o|\ge|z_0-o'|,
\end{equation}
where $o$ is the origin and $o'=\l\vec{a}$ as given above. Since $\a,\delta\le0$, we have by using (\ref{5.22})-(\ref{5.25}),
\begin{equation*}
\begin{split}
\p_tu_\l(x_0,t_0)=&u_{\l,o'}^\a\rho_{\l,o'}^\delta F^\beta(D^2u_{\l,o'}+u_{\l,o'}\Rmnum{1})\\
=&u_{\l,o'}^\a|z_0-o'|^\delta F^\beta(D^2u_{\l,o'}+u_{\l,o'}\Rmnum{1})\\
\ge&u^\a|z_0-o|^\delta F^\beta(D^2u+u\Rmnum{1})\\
=&\p_tu(x_0,t_0).
\end{split}
\end{equation*}
where we have used
$$D^2_{ij}x+x\delta_{ij}=-\bar{h}_{ij}x+x\delta_{ij}=-\delta_{ij}x+x\delta_{ij}=0,$$
and $\bar{h}_{ij}$ here is the second fundamental form of the unit sphere. This implies (\ref{5.20}).

Then given any two points $x_1,x_2\in\mS^n,x_1\neq x_2,$ let
$$\vec{a}=\frac{x_2-x_1}{|x_2-x_1|}.$$
Then $\widehat{x}_2=x_2-2<x_2,\vec{a}>\vec{a}=x_1,$ therefore by (\ref{5.19}) and (\ref{5.20}),
\begin{equation}\label{5.26}
u(x_1,t)+\l<x_2,\vec{a}>=u_\l(x_2,t)\ge u(x_2,t).
\end{equation}
It is easy to see that $<x_2,\vec{a}>=|x_2-x_1|/2$. Hence (\ref{5.26}) implies
$$\frac{u(x_2,t)-u(x_1,t)}{|x_2-x_1|}\le\l/2.$$
Note that $x_1,x_2\in\mS^n$. Therefore we have proved $\osc u\le C_{19}.$ Since $u(x,t)=\rho(x,t)$ when $x$ is an extremal point, we have
$$\osc\rho\le C_{20}.$$
By the argument in the proof of \cite{SO} Lemma 5.1, it follows that for any $t\in[0,T^*)$, there exists $x_t$ such that
$$\rho(x_t,t)=\Theta(\rho_0,t).$$
Conclusions (\ref{5.16}) and (\ref{5.17}) follow directly by combining the facts $\osc\rho\le C_{20}$ and $\rho(x_t,t)=\Theta(\rho_0,t).$ This completes the proof.
\end{proof}
\begin{proof of (1.3)}
(1.3) is a consequence of Theorem \ref{t5.3} and Lemma \ref{l5.7}.
\end{proof of (1.3)}

\section{Proof of Theorem \ref{t1.3}}
In this section, we shall establish the a priori estimates for normalized flow (\ref{1.7}) to show the long time existence and convergence of this flow when $\a+\beta+\delta>1$, $\a,\delta\le0$.

By (\ref{5.3}) and the expression of $\Theta(r,t)$ in Corollary \ref{c5.2}, we can find that $\varphi(t)=\Theta(\rho_0,t)$ when $\a+\beta+\delta>1$. Therefore by Lemma \ref{l5.7}, we have the $C^0$-estimate of normalized flow (\ref{1.7}).
\begin{lemma}\label{l6.1}
	Let $\rho(x,t)$, $t\in[0,T)$, be a smooth, strictly convex solution to (\ref{2.20}). If $\a,\delta\le0<\beta$, $\a+\delta+\beta>1$, then there is a positive constant $C_{21}$ depending only on $\alpha,\delta,\beta$ and the lower and upper bounds of $\rho(\cdot,0)$ such that
	\begin{equation*}
		\frac{1}{C_{21}}\leq \rho(\cdot,t)\leq C_{21}.
	\end{equation*}
\end{lemma}
The $C^1$-estimate of normalized flow (\ref{1.7}) is deduced by original flow (\ref{1.1}). Therefore in Lemma \ref{l6.2}, we distinguish between the terms in (\ref{1.7}) and (\ref{1.1}) through ``tilde''.
\begin{lemma}\label{l6.2}
Let $\a,\delta\le0<\beta$, $\a+\beta+\delta>1$, we have
\begin{equation}\label{6.1}
\om-1\le C_{22}\Theta^{-1},
\end{equation}
where $\Theta=\Theta(\rho_0,t)$ as in (\ref{5.16}), i.e.,
$$|D\widetilde{\rho}|\le C_{23}\quad \text{    and    }\quad \lim_{t\rightarrow T^*}|D\widetilde{\rho}|=0.$$
\end{lemma}
\begin{proof}
We deduce from (\ref{5.16}), (\ref{5.17}) and $\om=\frac{\rho}{u}$ that
\begin{align}
\om-1=&(\rho-u)u^{-1}=(\widetilde{\rho}-\widetilde{u})\widetilde{u}^{-1}\notag\\
=&(\widetilde{\rho}-1)\widetilde{u}^{-1}+(1-\widetilde{u})\widetilde{u}^{-1}\notag\\
\le& C_{22}\Theta^{-1},\notag
\end{align}
where we have used that $u(x,t)=\rho(x,t)$ when $x$ is an extremal point and $|\widetilde{\rho}-1|\le c\Theta^{-1}$. We mention that
$$D\widetilde{\g}=\frac{D\widetilde{\rho}}{\widetilde{\rho}}=\frac{D\rho}{\rho}=D\g.$$
 By (\ref{6.1}), we have
\begin{align*}
|D\widetilde{\g}|^2\le(\om+1)C_{22}\Theta^{-1}\le C_{23}\Theta^{-1}.
\end{align*}
Note that $\lim_{t\rightarrow T^*}\Theta^{-1}=0$. We have completed the proof by Lemma \ref{l6.1}.
\end{proof}
\textbf{Remark}: In fact, by Lemma \ref{l6.2}, we have the exponential convergence of the normalized flow. By $|D\widetilde{\g}|^2\le C_{23}\Theta^{-1}$, (\ref{5.3}) and (\ref{1.5}), we have
$$|D\widetilde{\rho}|^2\le C_{23}\widetilde{\rho}^2((\a+\delta+\beta-1)\eta T^*)^\frac{1}{1-\a-\delta-\beta}e^{-\eta\tau}\le ce^{-\eta\tau}.$$

Next we can derive the lower bound of $F$.
\begin{lemma}\label{l6.3}
	Let $\a,\delta\le0<\beta$, $1<\a+\beta+\delta$ and let $M(t)$ be a solution of the flow (\ref{1.7}). There exists a constant $C_{24}$ such that
	$$F\ge C_{24}>0.$$
\end{lemma}
\begin{proof}
The proof is similar to the proof of Lemma \ref{l5.6}. The auxiliary functions, etc. are the same as Lemma \ref{l5.6}. The calculation is similar to Theorem \ref{t5.3}. We directly give the evolution equation of $\theta$.
\begin{equation}
	\begin{split}
		(\p_t-\beta\frac{uQ}{F}F^{j}_i\nn^i\nn_j)\theta
		\le&(\alpha+\delta+\beta-1+N\rho)(Q-\eta)+(\a-1)NQ\rho|\nn\rho|^2\\
		&+(\delta+N\rho)(Q-\eta)|\nn\rho|^2-\frac{2\beta NQ}{F}F^{j}_i\nn^i\rho\nn_ju\\
		&+\frac{\beta uQ}{F}(N^2F^{j}_i\nn^i\rho\nn_j\rho-N(\frac{1}{\rho\om^2}F^{i}_i-\frac{u}{\rho}F)),
	\end{split}
\end{equation}
The rest is the same as Lemma \ref{l5.6}.
\end{proof}
Finally we will deduce the $C^2$-estimate of normalized flow (\ref{1.7}). Note that we don't have the upper bound of $F$ at this time. Therefore we need to classify all items in the evolution equations with respect to $F$.
\begin{lemma}\label{l6.4}
	Let $\alpha,\delta\leq0<\beta$, and $X(\cdot,t)$ be a smooth, closed and strictly convex solution to the flow (\ref{1.7}) which encloses the origin for $t\in[0,T)$. Then there is a positive constant $C_{25}$ depending on the initial hypersurface and $\alpha,\delta,\beta$,  such that the principal curvatures of $X(\cdot,t)$ are uniformly bounded from above $$\k_i(\cdot,t)\leq C_{25} \text{ \qquad } \forall1\le i\le n,$$
	and hence, are compactly contained in $\Gamma_+$, in view of Lemma \ref{l6.3}.
\end{lemma}
\begin{proof}
The proof is similar to the proof of Lemma \ref{l3.6}. The auxiliary functions, etc. are the same as Lemma \ref{l3.6}. We mention that $\a+\beta+\delta>1$ or $\a+\beta+\delta\le1$ are irrelevant to the proof of Lemma \ref{l3.6}. We only need to pay attention to whether $F$ has influence in this proof. By $C^0$ and $C^1$ estimate, we have
\begin{equation}
\begin{split}
\mathcal{L}\theta\le&\frac{\beta\Psi G}{F}F^{ij}(h^2)_{ij}+\beta\Psi f^{-\beta-1}f^{ij}(\nn_i(\log h_1^1)\nn_j(\log h_1^1)-p''\nn_iu\nn_ju)\\
&+\beta\frac{\Psi G}{F\k_1}(f^n_{m})^s_{r}h^m_{n;1}h^r_{s;1}-N\beta\frac{1}{\rho}\Psi f^{-\beta-1}f^{ij}g_{ij}\frac{1}{\om^2}+\eta(1-p'u)-N\eta u\om\\
&-N\eta\rho|\nn\rho|^2+G\(c-(\beta-\a+1)\Psi h_{11}\),
\end{split}
\end{equation}
which is negative by the calculation in Lemma \ref{l3.6} for large $\k_1$ after fixing $N_0$ large enough to ensure that
$$\eta(1-p'u)-N_0\eta u\om\le0.$$
Hence in this case any $N\geq N_0$ yields an upper bound for $\k_1$.

In conclusion, $\k_i\leq C_{25}$, where $C_{25}$ depends on the initial hypersurface, $\a$, $\delta$ and $\beta$. Since $F$ is uniformly continuous on the convex cone $\overline{\Gamma}_+$, and $F$ is bounded from below by a positive constant. Lemma \ref{l6.3} and Assumption \ref{a1.1} imply that $\k_i$ remains in a fixed compact subset of $\Gamma_+$, which is independent of $t$.
\end{proof}
\textbf{Remark}: Lemma \ref{l6.4} implies that
$$F\le C_{26}.$$
\begin{proof}
Follows immediately from the monotonicity and homogeneity of $F$.
\end{proof}
The estimates obtained in Lemma \ref{l6.1}, \ref{l6.2}, \ref{l6.3} and \ref{l6.4} depend on $\a$, $\delta$, $\beta$ and the geometry of the initial data $M_0$. They are independent of $T$. By Lemma \ref{l6.1}, \ref{l6.2}, \ref{l6.3} and \ref{l6.4}, we conclude that the equation (\ref{2.20}) is uniformly parabolic. By the $C^0$ estimate (Lemma \ref{l6.1}), the gradient estimate (Lemma \ref{l6.2}), the $C^2$ estimate (Lemma \ref{l6.4}) and the Krylov's and Nirenberg's theory \cite{KNV,LN}, we get the H$\ddot{o}$lder continuity of $D^2\rho$ and $\rho_t$. Then we can get higher order derivative estimates by the regularity theory of the uniformly parabolic equations. Hence we obtain the long time existence and $C^\infty$-smoothness of solutions for the normalized flow (\ref{1.7}). The uniqueness of smooth solutions also follows from the parabolic theory. In summary, we have proved the following theorem.
\begin{theorem}\label{t6.5}
	Let $M_0$ be a smooth, closed and strictly convex hypersurface in $\mathbb{R}^{n+1}$, $n\geq2$, which encloses the origin. If $\alpha,\delta\leq0<\beta,\a+\beta+\delta\ge1$, the normalized flow (\ref{1.7}) has a unique smooth, closed and strictly convex solution $M_t$ for all time $t\geq0$. Moreover, the radial function of $M_t$ satisfies the a priori estimates
	$$\parallel\rho\parallel_{C^{k,\beta}(\mS^n\times[0,\infty))}\leq C,$$
	where the constant $C>0$ depends only on $k,\a,\beta$ and the geometry of $M_0$.
\end{theorem}
\begin{proof of theorem 1.3}
By Lemma \ref{l6.2}, we have that $\vert D\rho\vert\to0$ exponentially as $t\to\infty$. Hence by the interpolation and the a priori estimates, we can get that $\rho$ converges exponentially to $1$ in the $C^\infty$ topology as $t\to\infty$.
\end{proof of theorem 1.3}

\section{Proof of Theorem \ref{xt1.4}}

In this section, we consider the flow (\ref{x1.9}). We study a convex hypersurface parametrized by the inverse Gauss map in this section. For convenience we still use $t$ instead of $\tau$ to denote the time variable if no confusions arise.

We shall derive the $C^0$ and $C^1$ estimates of the normalized flow (\ref{x1.9}). Firstly we derive an important lemma.
\begin{lemma}\label{l7.3}
Let $n\ge2$, $1\le k\le n$, $0<\psi\in C^\infty(\mS^n)$, $\a+\delta+\beta\le1$, $\beta>0$ and $u_0\in C^\infty(\mS^n)$ be positive and uniformly convex. Let $u(\cdot,t)$ be the solution of the normalized flow (\ref{x1.9}) for $t\in[0,T)$. Then there is a positive constant $C_{27}$ depending on the initial hypersurface and $\psi,\a,\delta,\beta,$ such that
$$C_{27}^{-1}\le u^{\a-1}\rho^{\delta}\sigma_{k}^\frac{\beta}{k}\le C_{27}.$$
\end{lemma}
\begin{proof}
Consider the auxiliary function
$$Q=\psi u^{\a-1}\rho^{\delta}\sigma_{k}^\frac{\beta}{k}.$$
It suffices to prove that $Q$ has a uniform bound. Firstly,
$$(\psi u^{\a}\rho^{\delta}\sigma_{k}^\frac{\beta}{k})_{ij}=Q_{ij}u+Q_iu_j+Q_ju_i+Qu_{ij}.$$
Note that
\begin{equation*}
\p_t\rho=\p_t\sqrt{u^2+|Du|^2}=\dfrac{u(Qu-\eta u)+u_k(Qu-\eta u)_k}{\rho}=\rho (Q-\eta)+\dfrac{uu_kQ_k}{\rho}.
\end{equation*}

 we get
\begin{equation}\label{x7.1}
\begin{split}
\p_tQ=&(\a-1)\frac{Q}{u}\p_tu+\delta\frac{Q}{\rho}\p_t\rho+\frac{\beta Q}{k\sigma_{k}}\sigma_{k}^{ij}\((\p_tu)_{ij}+\p_tu\delta_{ij}\)\\
=&(\a+\delta-1)(Q-\eta)Q+\dfrac{\delta uQ}{\rho^2}u_kQ_k\\
&+\frac{\beta Q}{k\sigma_{k}}\sigma_{k}^{ij}(Q_{ij}u+Q_iu_j+Q_ju_i+(Q-\eta)h_{ij})\\
=&(\a+\delta+\beta-1)(Q-\eta)Q+\frac{\beta uQ\sigma_{k}^{ij}}{k\sigma_{k}}Q_{ij}+2\frac{\beta Q}{k\sigma_{k}}\sigma_{k}^{ij}Q_iu_j+\dfrac{\delta uQ}{\rho^2}u_kQ_k.
\end{split}
\end{equation}
If $\a+\delta+\beta-1=0$, by the strong maximum principle we can derive $C_{27}^{-1}\le Q\le C_{27}.$
If $\a+\delta+\beta-1<0$, the sign of the coefficient of the highest order term $Q^2$ is negative and the sign of the coefficient of the lower order term $Q$ is positive. So by the maximum principle we can derive $C_{27}^{-1}\le Q\le C_{27}.$
\end{proof}

Next we give an interesting lemma, although it is not used in this paper if $1\le k<n$. We hope this lemma can be widely used to prove $C^0$ estimates.
\begin{lemma}\label{l7.4}
Let $n\ge2$, $\beta>0$, $0<\psi\in C^\infty(\mS^n)$, Assume

(\rmnum{1}) $1\le k\le n-1$, $\a+\delta+\beta<1$ or $1-\delta-\beta=\a<1$;

or (\rmnum{2}) $k=n$, $\a+\delta+\beta<1$ or $1-\delta-\beta=\a\ne\frac{\beta}{n}+1$.

Let $u_0\in C^\infty(\mS^n)$ be positive and uniformly convex and $u(\cdot,t)$ be the solution of flows (\ref{x1.8}) or (\ref{x1.9}) or (\ref{x1.13}) for $t\in[0,T)$. Then there is a positive constant $C_{28}$ depending on the initial hypersurface and $\psi,\a,\delta,\beta,k$ such that
$$\max_{\mS^n}\frac{\vert Du\vert}{u}(\cdot,t)\le C_{28}, \quad \forall t\in[0,T).$$
\end{lemma}
\begin{proof}
Since our original flow (\ref{x1.8}) and the normalized flow (\ref{x1.9}) or (\ref{x1.13}) differs a time-dependent dilation, $\frac{\vert Du\vert}{u}$ is scaling-invariant. It suffices to show
$$\max_{\mS^n}\frac{\vert Du\vert}{u}(\cdot,t)\le C_{28},$$
where $u$ is the solution to original flow (\ref{x1.8}). Therefore we can also use this calculation in the next section.

Let $\theta=\log u$. By $h_{ij}=u_{ij}+u\delta_{ij}=u(\theta_{ij}+\theta_i\theta_j+\delta_{ij})$, it is straightforward to see
\begin{equation}\label{7.12}
\frac{\p\theta}{\p t}(x,t)=\psi e^{\iota\theta}(1+\vert D\theta\vert)^{\delta/2}\sigma_{k}^\frac{\beta}{k}([D_iD_j\theta+D_i\theta D_j\theta+\delta_{ij}]),
\end{equation}
where $\iota=\a+\delta+\beta-1$. We use the following notations for convenience
$$a_{ij}=D_iD_j\theta+D_i\theta D_j\theta+\delta_{ij}.$$
Note that, by our assumption, $\iota\le0.$

Let $Q=\frac{1}{2}\vert D\theta\vert^2$. Then we only need to derive the upper bound of $Q$. Suppose $Q(\cdot,t)$ attains its spatial maximum at some $x_t\in\mS^n$. Then at $(x_t,t)$,
\begin{align}
0=D_iQ=\sum\theta_k\theta_{ki},\label{7.13}\\
0\ge D^2_{ij}Q=\theta_{kij}\theta_k+\theta_{ki}\theta_{kj}.\label{7.14}
\end{align}
By a rotation, we assume $\theta_1=\vert D\theta\vert$. Then (\ref{7.13}) implies $\theta_{1k}=0,\forall k$. Therefore we can assume by a further rotation that $\{\theta_{ij}\}$ is diagonal at $(x_t,t)$, and so
$$\{a_{ij}\}=\text{diag}(1+\theta_1^2,1+\theta_{22},\cdots,1+\theta_{nn}).$$
We denote $\sigma_k^{ij}=\frac{\p\sigma_k}{\p a_{ij}}$. Then by (\ref{7.12}) and (\ref{7.13}) we have
\begin{equation}\label{7.15}
\begin{split}
\frac{\p_tQ}{\theta_t}=&\frac{\theta_k\theta_{kt}}{\theta_t}\\
=&\theta_k\(\frac{\psi_k}{\psi}+\iota\theta_k+\frac{\beta\sigma_k^{ij}}{k\sigma_k}\theta_{ijk}\).
\end{split}
\end{equation}
By the Ricci identity,$$D_k\theta_{ij}=D_j\theta_{ki}+\delta_{ik}\theta_j-\delta_{ij}\theta_k.$$
Substituting (\ref{7.14}) into (\ref{7.15}), we derive
\begin{equation}\label{7.16}
	\begin{split}
\frac{\p_tQ}{\theta_t}=&\theta_k\(\frac{\psi_k}{\psi}+\iota\theta_k+\frac{\beta\sigma_k^{ij}}{k\sigma_k}(D_j\theta_{ki}+\delta_{ik}\theta_j-\delta_{ij}\theta_k)\)\\
\le&C_{29}\theta_1+\iota\theta_1^2-\frac{\beta\sigma_k^{ii}}{k\sigma_k}\theta_{ii}^2+\frac{\beta\sigma_k^{11}}{k\sigma_k}\theta_1^2-\frac{\beta\sigma_k^{ii}}{k\sigma_k}\theta_1^2.
	\end{split}
\end{equation}
Note that
\begin{equation}\label{7.17}
\begin{split}
-\sigma_k^{ii}\theta_{ii}^2=&-\sum_{i>1}\sigma_k^{ii}(a_{ii}-1)^2\\
=&-\sum_{i>1}\sigma_k^{ii}a_{ii}^2+2\sum_{i>1}\sigma_k^{ii}a_{ii}-\sum_{i>1}\sigma_k^{ii}\\
\le&-\sum_{i>1}\sigma_k^{ii}a_{ii}^2+2k\sigma_k-\sum_{i>1}\sigma_k^{ii},
\end{split}
\end{equation}
\begin{equation}\label{7.18}
\begin{split}
\frac{\beta\sigma_k^{11}}{k\sigma_k}\theta_1^2\le\frac{\beta\sigma_k^{11}}{k\sigma_k}a_{11}\le \beta.
\end{split}
\end{equation}
Substituting (\ref{7.17}) and (\ref{7.18}) into (\ref{7.16}), we derive
\begin{equation}\label{7.19}
	\begin{split}
		\frac{\p_tQ}{\theta_t}
		\le&C_{29}(1+\theta_1)+\iota\theta_1^2-\frac{\beta}{k\sigma_k}\sum_{i>1}\sigma_k^{ii}a_{ii}^2-\frac{\beta}{k\sigma_k}\sum_{i>1}\sigma_k^{ii}-\frac{\beta\sigma_k^{ii}}{k\sigma_k}\theta_1^2.
	\end{split}
\end{equation}
When $\iota<0$, the RHS of (\ref{7.19}) is negative provided $\theta_1$ is sufficiently large. Since $\theta_t>0$, we deduce from (\ref{7.19}) that $Q$ has a upper bound.

 If $\iota=0$, note that at this case we have $C_{27}^{-1}\le u^{\a-1}\rho^\delta\sigma_{k}^\frac{\beta}{k}\le C_{27}$ by Lemma \ref{l7.3} and $u^{\a-1}\rho^\delta\sigma_{k}^\frac{\beta}{k}$ is scaling-invariant.

 (\rmnum{1}) $1\le k\le n-1$. By Lemma \ref{l7.3} we have $(1+\theta_1^2)^\frac{\delta}{2}\sigma_{k}^\frac{\beta}{k}\le C_{27}$. By Newton-Maclaurin inequality, we can get
 \begin{equation}\label{7.20}
\sigma_{k}^{ii}=\frac{C_n^k\times k}{C_n^{k-1}}\sigma_{k-1}\ge C(n,k)\sigma_{k}^{\frac{k-1}{k}}.
 \end{equation}
Substituting (\ref{7.20}) into (\ref{7.19}), we have
\begin{align*}
\frac{\p_tQ}{\theta_t}\le&C_{29}(1+\theta_1)-\frac{\beta}{k} C(n,k)\sigma_{k}^{-\frac{1}{k}}\theta_1^2\\
\le&C_{29}(1+\theta_1)-C_{30}(1+\theta_1^2)^\frac{\delta}{2\beta}\theta_1^2\\
\le&C_{29}(1+\theta_1)-C_{31}\theta_1^{2+\frac{\delta}{\beta}}.
\end{align*}
By $\a<1$ and $\delta=1-\a-\beta$, we have $2+\frac{\delta}{\beta}>1$. Thus, when $\iota=0$, one sees again that the RHS of (\ref{7.19}) is negative provided $\theta_1$ is sufficiently large.

(\rmnum{2}) $k=n$. By Lemma \ref{l7.3} we have $\frac{1}{C_{27}}\le(1+\theta_1^2)^\frac{\delta}{2}\sigma_{n}^\frac{\beta}{n}\le C_{27}$. Note that $\sigma_{n}^{ii}a_{ii}=\sigma_{n}$ for $\forall i$. We have
\begin{equation}\label{u}
\begin{split}
-\frac{\beta}{n\sigma_n}\sum_{i>1}\sigma_n^{ii}a_{ii}^2&=-\frac{\beta}{n}\sum_{i>1}a_{ii}\le-C(n,\beta)\(\prod_{i>1}a_{ii}\)^\frac{1}{n-1}=-C(n,\beta)\frac{\sigma_n^\frac{1}{n-1}}{a_{11}^\frac{1}{n-1}}\\
&\le-C(n,\beta)(1+\theta_1^2)^{-\frac{1}{n-1}(1+\frac{n\delta}{2\beta})},\\
-\frac{\beta\sigma_n^{ii}}{n\sigma_n}\theta_1^2&\le-\frac{\beta\sum_{i>1}\sigma_n^{ii}}{n\sigma_n}\theta_1^2=-\frac{\beta}{n}\sum_{i>1}\frac{1}{a_{ii}}\theta_1^2\le-C(n,\beta)\(\prod_{i>1}\frac{1}{a_{ii}}\)^\frac{1}{n-1}\theta_1^2\\
&=-C(n,\beta)\(\frac{a_{11}}{\sigma_{n}}\)^\frac{1}{n-1}\theta_1^2\le
-C(n,\beta)(1+\theta_1^2)^{1+\frac{1}{n-1}(1+\frac{n\delta}{2\beta})}.
\end{split}
\end{equation}
Note that $-\frac{1}{n-1}(1+\frac{n\delta}{2\beta})$ or $1+\frac{1}{n-1}(1+\frac{n\delta}{2\beta})$ are strictly larger than $\frac{1}{2}$ if $\delta\ne\frac{-(n+1)\beta}{n}$. By (\ref{u}), one sees again that the RHS of (\ref{7.19}) is negative provided $\theta_1$ is sufficiently large and the proof is completed.
\end{proof}
By the above lemmas, we can derive the $C^0$ and $C^1$ estimates.
\begin{lemma}\label{l7.5}
	Let $n\ge2$, $\a,\delta,\beta\in \mR^1$,  $\beta>0$ and $\a+\delta+\beta<1$, $k$ is an integer, $\psi>0$ is a smooth positive function on $\mS^n$. 
Let $u_0\in C^\infty(\mS^n)$ be positive and uniformly convex and $u(\cdot,t)$ be the solution of the normalized flow (\ref{x1.9}) for $t\in[0,T)$. Then there are positive constants $C_{31}$, $C_{32}$, $C_{33}$ depending on the initial hypersurface and $\psi,\a,\delta,\beta,k$ such that
	$$C_{31}^{-1}\le u\le C_{31},\quad \vert Du\vert\le C_{32},\quad C_{33}^{-1}\le\sigma_{k}\le C_{33}.$$
\end{lemma}
\begin{proof}
Let $u_{\min}(t)=\min_{x\in \mS^n}u(\cdot,t)=u(x_t,t)$. For fixed time $t$, at the point $x_t$,  we have $$D_iu=0 \text{ and } D^2_{ij}u\geq0.$$
Note that $h_{ij}=u_{ij}+u\delta_{ij}\geq u\delta_{ij}$, we have  $\sigma_k^\frac{\beta}{k}(h_{ij})\geq\eta u^\beta$, then (\ref{x1.9}) implies
$$\frac{d}{dt}u_{\min}\geq\eta u_{\min}(\psi(x_t)u_{\min}^{\alpha+\delta+\beta-1}-1).$$
Hence $u_{\min}\geq \min\{(\frac{1}{\min_{\mS^n}\psi})^\frac{1}{\alpha+\delta+\beta-1},u_{\min}(0)\}$. Similarly, we have the upper bound of $u$. By Lemma \ref{l2.5} and $\rho^2=u^2+|Du|^2$, we have the upper bound of $|Du|$. By Lemma \ref{l7.3} and the bounds of $u$ and $\rho$, we have the bounds of $\sigma_{k}$.
\end{proof}
Next we derive the uniform bounds of the principal curvature radii $\l_i$, which are the eigenvalues of matrix $[D^2u+u\Rmnum{1}]$. For $k=n$, the $C^2$ estimates for a class of inverse curvature flows  have been derived by  Lemma 8.2 in \cite{CL}.
\begin{lemma}\cite{CL}\label{l7.6}
Let $\beta>0$. Let $u(\cdot,t)$ be a positive, smooth and uniformly convex solution to
\begin{equation*}
\frac{\p u}{\p t}=\Psi(t,x,u,Du){\det}^\frac{\beta}{n}(D^2u+u\Rmnum{1})-\eta(t)u \qquad \text{on }\mS^n\times[0,T),
\end{equation*}
where $\Psi(t,x,u,Du):\mR_{\ge0}\times\mS^n\times\mR_{\ge0}\times\mR^n\rightarrow\mR_{\ge0}$ is a smooth function. Suppose that $\Psi(t,x,u,Du)>0$ whenever $u>0$. If
\begin{align*}
\frac{1}{C_0}\le u(x,t)\le C_0,\qquad &\forall(x,t)\in\mS^n\times[0,T),\\
\vert Du\vert(x,t)\le C_0,\qquad &\forall(x,t)\in\mS^n\times[0,T),\\
0\le\eta(t)\le C_0,\qquad&\forall t\in[0,T),
\end{align*}
for some constant $C_0>0$, then
\begin{equation*}
(D^2u+u\Rmnum{1})(x,t)\ge C^{-1}\Rmnum{1},\qquad \forall(x,t)\in\mS^n\times[0,T),
\end{equation*}
where $C$ is a positive constant depending only on $n,\beta,C_0,u(\cdot,0)$, $\vert\Psi\vert_{L^\infty(U)}$, $\vert1/\Psi\vert_{L^\infty(U)}$, $\vert\Psi\vert_{C^1_{x,u,Du}(U)}$ and $\vert\Psi\vert_{C^2_{x,u,Du}(U)}$ where $U=[0,T)\times\mS^n\times[1/C_0,C_0]\times B^n_{C_0}$ ($B^n_{C_0}$ is the ball
centered at the origin with radius $R$ in $\mR^n$).
\end{lemma}

\begin{lemma}\label{l7.7}
Let $n\ge2$, $\a,\delta,\beta\in \mR^1$, $\beta>0$, $k$ is an integer, $\psi>0$ is a smooth positive function on $\mS^n$. If $\frac{1}{C}\le u\le C$ and $\frac{1}{C}\le\sigma_{k}\le C$ for some constant $C>0$, and

(\rmnum{1}) $1\le k<n$, $\a\le0$ and $(D_iD_j\psi^{\frac{1}{1+\beta-\a}}+\delta_{ij}\psi^{\frac{1}{1+\beta-\a}})$ is positive definite;

or (\rmnum{2}) $1\le k<n$, $\beta+1<\a$ and $(D_iD_j\psi^{\frac{1}{1+\beta-\a}}+\delta_{ij}\psi^{\frac{1}{1+\beta-\a}})$ is negative definite;

or (\rmnum{3}) $k=n$.

Let $u_0\in C^\infty(\mS^n)$ be positive and uniformly convex and $u(\cdot,t)$ be the solution of the normalized flow (\ref{x1.9}) for $t\in[0,T)$. Then there is a positive constant $C_{39}$ depending on the initial hypersurface and $\psi,\a,\delta,\beta,k$ such that the principal curvature radii of $X(\cdot,t)$ are bounded from above and below
	$$C_{39}^{-1}\le\l_i(\cdot,t)\le C_{39}, \quad \forall t\in[0,T) \text{    and    }i=1,\cdots,n.$$
\end{lemma}
\begin{proof}
First, we shall prove that $\l_i$  is bounded from below by a positive constant. The principal radii of curvatures of $M_t$ are the eigenvalues of $\{h_{il}e^{lj}\}$. To derive a positive lower bound of principal curvatures radii, it suffices to prove that the eigenvalues of $\{h^{il}e_{lj}\}$ are bounded from above.

(\rmnum{1}) and (\rmnum{2}) $1\le k\le n-1$. Similar to Lemma \ref{l3.6}, suppose the maximum eigenvalue of $\theta=\log h^{ij}-\log u$ at time $t$ is attained at the point $x_t$ with unit eigenvector $\xi_t\in T_{x_t}\mS^n$, where $h^{ij}$ is the inverse of $h_{ij}=D^2_{ij}u+u\delta_{ij}$. By a rotating the frame $e_1,\cdots,e_n$ at $x_t$, assume that at $x_t$ we have $\xi_t=e_1$. Then at $(x_t,t)$, $\theta=\log h^{11}-\log u$. We mention that we calculate the evolution equation of $\theta$ by tensor property. Meanwhile, at $(x_t,t)$, $[h_{ij}]$ and $[\sigma_{k}^{ij}]$ are diagonal.

Let $\Phi=\psi u^\a\rho^\delta$ and $G=\sigma_k^\frac{\beta}{k}=F^\beta$, the evolution equation of $h^{11}$ can be found in Lemma 3.4 in \cite{DL}. By below (3.8) in \cite{DL}, we have
\begin{align}
	\p_th^{11}\leq&\Phi G^{kl}D_kD_lh^{11}-(\beta+1)\Phi G(h^{11})^2+\Phi \sum_iG^{ii}h^{11}+\eta h^{11}\notag\\
	&+\frac{\beta}{\beta+1}G\frac{(D_1\Phi)^2}{\Phi}(h^{11})^2-GD_1D_1\Phi(h^{11})^2.\label{7.30}
\end{align}
We denote
$$\cL=\p_t-\Phi G^{kl}D_kD_l.$$
By (\ref{x1.9}), we have
\begin{equation}\label{7.31}
\begin{split}
\cL u=&\Phi G-\eta u-\Phi G^{kl}D_kD_lu\\
=&\Phi G-\eta u-\beta\Phi\frac{G}{F} F^{kl}(h_{kl}-u\delta_{kl})\\
=&(1-\beta)\Phi G-\eta u+u\Phi\sum_{i}G^{ii}.
\end{split}
\end{equation}
By (\ref{7.30}) and (\ref{7.31}) we can derive the evolution equation of $\theta$.
\begin{equation}\label{7.33}
\begin{split}
\cL\theta=&\frac{\cL h^{11}}{h^{11}}-\frac{\cL u}{u}\\
\le&c+\Phi G\(-(\beta+1)h^{11}+\frac{\beta}{\beta+1}(D_1\log\Phi)^2h^{11}-\frac{D_1D_1\Phi}{\Phi}h^{11}\)\\
=&c+\Phi Gh^{11}\(-(\beta+1)+\frac{\beta}{\beta+1}(D_1\log\psi+\a D_1\log u+\delta D_1\log\rho)^2\\
&-\frac{\psi_{11}}{\psi}-2\a D_1\log\psi D_1\log u-2\delta D_1\log\psi D_1\log \rho-\a\frac{h_{11}-u}{u}\\
&-\a(\a-1)(D_1\log u)^2-2\a\delta D_1\log uD_1\log \rho-\delta\frac{\rho_{11}}{\rho}-\delta(\delta-1)(D_1\log \rho)^2\).
\end{split}
\end{equation}
By Lemma \ref{l7.4}, we have $\vert D_1\log u\vert\le C_{28}$. Note that at $(x_t,t)$ we have
$D\theta=\frac{Dh^{11}}{h^{11}}-\frac{Du}{u}=0.$ It also means $Dh_{11}=-h_{11}\frac{Du}{u}$. Thus it is direct to calculate
\begin{equation}\label{7.34}
\begin{split}
\rho_1=&\frac{uu_1+u_ku_{k1}}{\rho}=\frac{u_1h_{11}}{\rho},\\
\rho_{11}=&\frac{u_{11}u+u_1^2+u_{k1}^2+u_ku_{k11}}{\rho}-\frac{(uu_1+u_ku_{k1})(uu_1+u_lu_{l1})}{\rho^3}\\
=&\frac{(h_{11}-u)u+u_1^2+(h_{11}-u)^2+u_k(D_kh_{11}-u_1\delta_{1k})}{\rho}-\frac{(u_kh_{k1})(u_lh_{l1})}{\rho^3}\\
=&\frac{-uh_{11}+h_{11}^2-h_{11}\frac{\vert Du\vert^2}{u}}{\rho}-\frac{(u_1h_{11})^2}{\rho^3}.
\end{split}
\end{equation}
By $(\ref{7.34})$, we can find $h^{11}\rho_1$ and $h^{11}\rho_{11}$ have uniform bounds provided $h^{11}$ sufficiently large. Thus, by (\ref{7.33}) and (\ref{7.34}), we have
\begin{equation}\label{7.28}
\begin{split}
\cL\theta\le&c+\Phi Gh^{11}\(-(\beta-\a+1)+\frac{\beta}{\beta+1}(D_1\log\psi+\a D_1\log u)^2\\
&-\frac{\psi_{11}}{\psi}-2\a D_1\log\psi D_1\log u-\a(\a-1)(D_1\log u)^2\)\\
=&c+\Phi Gh^{11}\(\a\(\sqrt{\dfrac{1+\beta-\a}{1+\beta}}D_1\log u-\sqrt{\dfrac{1}{(1+\beta)(1+\beta-\a)}}D_1\log\psi\)^2\\
&-(\beta-\a+1)+\frac{\beta-\a}{1+\beta-\a}(D_1\log\psi)^2-\frac{\psi_{11}}{\psi}\)
\end{split}
\end{equation}
if $\a\le0$.
Since $D_1D_1\psi^{\frac{1}{1+\beta-\a}}+\psi^{\frac{1}{1+\beta-\a}}>0$, we have
$$1-\frac{\beta-\a}{(1+\beta-\a)^2}(D_1\log\psi)^2+\frac{1}{1+\beta-\a}\frac{\psi_{11}}{\psi}>0.$$
Thus, there exists a positive constant $C_{40}$, s.t.
$$1-\frac{\beta-\a}{(1+\beta-\a)^2}(D_1\log\psi)^2+\frac{1}{1+\beta-\a}\frac{\psi_{11}}{\psi}\ge C_{40}$$
by $\psi$ defined on $\mS^n$. In fact, $1-\frac{\beta-\a}{(1+\beta-\a)^2}(D_1\log\psi)^2+\frac{1}{1+\beta-\a}\frac{\psi_{11}}{\psi}$ is also defined on a compact set. Thus it can arrive the minimum $C_{40}>0$ on $\mS^n$.

If $\a>1+\beta$, by (\ref{7.28}) we have
\begin{equation*}
\begin{split}
\cL\theta\le&c+\Phi Gh^{11}\(-\a\(\sqrt{\dfrac{\a-1-\beta}{1+\beta}}D_1\log u+\sqrt{\dfrac{1}{(1+\beta)(\a-1-\beta)}}D_1\log\psi\)^2\\
&-(\beta-\a+1)+\frac{\beta-\a}{1+\beta-\a}(D_1\log\psi)^2-\frac{\psi_{11}}{\psi}\).
\end{split}
\end{equation*}
Since $D_1D_1\psi^{\frac{1}{1+\beta-\a}}+\psi^{\frac{1}{1+\beta-\a}}<0$, we have
$$1-\frac{\beta-\a}{(1+\beta-\a)^2}(D_1\log\psi)^2+\frac{1}{1+\beta-\a}\frac{\psi_{11}}{\psi}<0.$$
Thus, there exists a positive constant $C_{40}$, s.t.
$$1-\frac{\beta-\a}{(1+\beta-\a)^2}(D_1\log\psi)^2+\frac{1}{1+\beta-\a}\frac{\psi_{11}}{\psi}\le -C_{40}$$
by $\psi$ defined on $\mS^n$. In fact, $1-\frac{\beta-\a}{(1+\beta-\a)^2}(D_1\log\psi)^2+\frac{1}{1+\beta-\a}\frac{\psi_{11}}{\psi}$ is also defined on a compact set. Thus it can arrive the maximum $-C_{40}<0$ on $\mS^n$. Note that $1+\beta-\a<0$ at this case.

Combining the two cases, we have
$$\p_t\theta\leq-C_{40}\theta+c.$$
That is, $\theta\leq \frac{1}{C_{39}}$ or $\frac{1}{\l_i}\leq \frac{1}{C_{39}}$ by $C^0$ estimates, where $C_{39}$ depends on the initial hypersurface, $\psi$, $\a$, $\delta$ and $\beta$. Thus $\l_i\ge C_{39}$ for $i=1,\cdots,n$.

(\rmnum{3}) and (\rmnum{4}): $k=n$. By Lemma \ref{l7.5} and Lemma \ref{l7.6}, we can derive the lower bound of $\l_i$.

 Since $F_*(\frac{1}{\l_1},\cdots,\frac{1}{\l_n})=\frac{1}{F(\l_1,\cdots,\l_n)}$ is uniformly continuous on $\{\l\in\overline{\Gamma}^+\vert\l_i\geq C_{39}$ for all $i \}$ and $F_*$ is bounded from below by a positive constant, $F^*\vert_{\p\Gamma_+}=(\frac{\sigma_n}{\sigma_{n-k}})^\frac{1}{k}\vert_{\p\Gamma_+}=0$ implies that $\l_i$ remains in a fixed compact subset of $\bar\Gamma^+$, which is independent of $t$. That is $$\frac{1}{C_{39}}\leq\l_i(\cdot,t)\leq C_{39}.$$
\end{proof}
The estimates obtained in Lemma \ref{l7.5} and \ref{l7.7} depend on $\psi$, $\a$, $\delta$, $\beta$ and the geometry of the initial data $M_0$. They are independent of $T$. By Lemma \ref{l7.5} and \ref{l7.7}, we conclude that the equation (\ref{x1.9}) is uniformly parabolic. By the $C^0$ estimate, the gradient estimate (Lemma \ref{l7.5}), the $C^2$ estimate (Lemma \ref{l7.7}), Cordes and Nirenberg type estimates \cite{B2,CO,LN} and the Krylov's theory \cite{KNV}, we get the H$\ddot{o}$lder continuity of $D^2u$ and $u_t$. Then we can get higher order derivation estimates by the regularity theory of the uniformly parabolic equations. Hence we obtain the long time existence and $C^\infty$-smoothness of solutions for the normalized flow (\ref{x1.9}). The uniqueness of smooth solutions also follows from the parabolic theory.

We give a uniqueness result with $\psi$ in case $\a+\delta+\beta=1$. This lemma states that our method can't be used to solve the case $\a+\delta+\beta=1$. In fact, there is no difference between $\psi$ and $\psi'=c\psi$ as we deal with this problem. Thus it will cause a contradiction if this method can handle the case $\a+\delta+\beta=1$. We assert that monotonic quantities along some one flow must be found  to handle this case.
\begin{lemma}
Let $\a+\delta+\beta=1$ and there exists a solution to $\psi u^{\a-1}\rho^\delta\sigma_{k}^\frac{\beta}{k}=\eta$. If $\psi_1>\psi$ for all $x\in\mS^n$, then there is no solution to this equation with $\psi$ replaced by $\psi_1$; Similarly if $\psi>\psi_2$ for all $x\in\mS^n$, then there is no solution to this equation with $\psi$ replaced by $\psi_2$.
\end{lemma}
\begin{proof}
We assume that $u_0$ is a solution to $\psi u^{\a-1}\rho^\delta\sigma_{k}^\frac{\beta}{k}=\eta$, and show that this equation can't admit a solution if $\psi$ is replaced by $\psi_1$. Otherwise, suppose that $u_1$ satisfies
\begin{equation}\label{x7}
\psi_1 u_1^{\a-1}(u_1^2+|Du_1|^2)^\frac{\delta}{2}\sigma_{k}^\frac{\beta}{k}(D^2u_1+u_1\Rmnum{1})=\eta.
\end{equation}
Let $u_1^\l=\l u_1$, and $\l_*=\inf\{\l\ge0,u_1^\l>u_0\}$. Since $\a+\delta+\beta=1$, $u_1^\l$ satisfies (\ref{x7}) as well. Let $x_0\in\mS^n$ be a point such that $u_1^{\l_*}(x_0)=u_0(x_0)$. By the choice of $\l_*$, we also have, at the point $x_0$, $Du_1^{\l_*}=Du_0$ and $D^2u_1^{\l_*}\ge D^2u_0$. Consequently
\begin{align*}
\sigma_{k}^\frac{\beta}{k}(D^2u_1^{\l_*}+u_1^{\l_*}\Rmnum{1})(x_0)=&\eta\psi_1^{-1}(x_0)(u_1^{\l_*})^{1-\a}((u_1^{\l_*})^2+|Du_1^{\l_*}|^2)^{-\frac{\delta}{2}}(x_0)\\
<&\eta\psi^{-1}(x_0)u_0^{1-\a}(u_0^2+|Du_0|^2)^{-\frac{\delta}{2}}(x_0)\\
=&\sigma_{k}^\frac{\beta}{k}(D^2u_0+u_0\Rmnum{1})(x_0)\\
\le&\sigma_{k}^\frac{\beta}{k}(D^2u_1^{\l_*}+u_1^{\l_*}\Rmnum{1})(x_0),
\end{align*}
thus arriving a contradiction. Similarly one sees that the equation $\psi u^{\a-1}\rho^\delta\sigma_{k}^\frac{\beta}{k}=\eta$ doesn't admit a solution if
$\psi$ is replaced by a function $\psi_2<\psi$ for all $x\in\mS^n$.
\end{proof}

Finally, we complete the proof of Theorem \ref{xt1.4}.

\begin{proof of theorem 1.4}
Firstly, we give the convergence of this flow. By the choice of $C_0$ in Page 6, we have the initial hypersurface $M_0$ satisfied $\psi u^{\a-1}\rho^\delta\sigma_{k}^\frac{\beta}{k}\vert_{M_0}>\eta$. We shall derive this condition preserved along this flow. Let $Q=\psi u^{\a-1}\rho^\delta\sigma_{k}^\frac{\beta}{k}$, by (\ref{x7.1}), we can find the condition $\psi u^{\a-1}\rho^\delta\sigma_{k}^\frac{\beta}{k}\vert_{M_t}>\eta$ is preserved along the normalized flow (\ref{x1.9}) for all $t\ge0$, i.e. $\p_tu>0$ is preserved along the normalized flow (\ref{x1.9}) for all $t\ge0$. Moreover, due to the $C^0$ estimates,
$$\vert u(x,t)-u(x,0)\vert=\vert\int_{0}^t\(\psi u^{\a-1}\rho^\delta\sigma_{k}^\frac{\beta}{k}-\eta\)u(x,s)ds\vert<\infty.$$
Therefore, in view of the monotonicity of $u$, the limit $u(x,\infty):=\lim_{t_i\rightarrow\infty}u(x,t)$ exists and is positive, smooth and $D^2u(x,\infty)+u(x,\infty)\Rmnum{1}>0$. Thus the hypersurface with support function $u(x,\infty)$ is our desired solution to $$\psi u^{\a-1}\rho^\delta\sigma_{k}^\frac{\beta}{k}=\eta.$$

Finally, it suffices to show that the solution of $\psi u^{\a-1}\rho^\delta\sigma_{k}^\frac{\beta}{k}=c$ is unique.

Let $u_1$, $u_2$ be two smooth solutions, i.e.
\begin{equation*}
\psi u_1^{\a-1}\rho_1^\delta\sigma_{k}^\frac{\beta}{k}(D^2u_1+u_1\Rmnum{1})=c,\qquad \psi u_2^{\a-1}\rho_2^\delta\sigma_{k}^\frac{\beta}{k}(D^2u_2+u_2\Rmnum{1})=c.
\end{equation*}
Suppose $M=\frac{u_1}{u_2}$ attains its maximum at $x_0\in\mS^n$, then at $x_0$,
\begin{align*}
0=D\log M=&\frac{Du_1}{u_1}-\frac{Du_2}{u_2},\\
0\ge D^2\log M=&\frac{D^2u_1}{u_1}-\frac{D^2u_2}{u_2}.
\end{align*}
Hence at $x_0$, by $\rho=u\sqrt{1+\vert D\log u\vert^2}$ we get
\begin{equation*}
1=\frac{u_1^{\a-1}\rho_1^\delta\sigma_{k}^\frac{\beta}{k}(D^2u_1+u_1\Rmnum{1})}{u_2^{\a-1}\rho_2^\delta\sigma_{k}^\frac{\beta}{k}(D^2u_2+u_2\Rmnum{1})}=\frac{u_1^{\a+\delta+\beta-1}\sigma_{k}^\frac{\beta}{k}(\frac{D^2u_1}{u_1}+\Rmnum{1})}{u_2^{\a+\delta+\beta-1}\sigma_{k}^\frac{\beta}{k}(\frac{D^2u_2}{u_2}+\Rmnum{1})}\le M^{\a+\delta+\beta-1}.
\end{equation*}
Since $\a+\delta+\beta<1$, $M(X_0)=\max_{\mS^n}M\le1$. Similarly one can show $\min_{\mS^n}M\ge1$. Therefore $u_1\equiv u_2$.

\end{proof of theorem 1.4}

\section{Proof of Theorem 1.5}
In this section, we use another scaling to increase the range of $\a$ and $\delta$. Meanwhile, we derive some results with even prescribed data. The lemma below justifies that (\ref{x1.13}) is the normalized flow of (\ref{x1.8}).
\begin{lemma}\label{l8.1}
Let $\beta>0$. Let $X(\cdot,t)$ be a smooth solution to the flow (\ref{x1.8}) with $t\in[0,T)$, and for each $t\ge0$, $M_t=X(\mS^n,t)$ be a smooth, closed and uniformly convex
hypersurface. Suppose that the origin lies in the interior of convex body $\Om_{t}$ enclosed by $M_t$ for all $t\in[0,T)$. Then $\widetilde{X}(\cdot,\tau)$, given by (\ref{x1.12}), satisfies the normalized flow (\ref{x1.13}).
\end{lemma}
\begin{proof}
Let $\rho(\cdot,t)$ and $\widetilde{\rho}(\cdot,\tau)$ be the radial function of $M_t=X(\mS^n,t)$ and $\widetilde{M}_\tau=\widetilde{X}(\cdot,\tau)$ respectively, and let $\widetilde{\Om}_\tau$ be the convex body enclosed by $\widetilde{M}_\tau$. By virtue of (\ref{x1.12}), (\ref{2.26}) and (\ref{2.27}) and using the change of variables formula, we have
\begin{equation*}
	\begin{split}
\frac{\p\tau}{\p t}=&
\begin{cases}
	{\int\hspace{-0.9em}-}_{\mS^n} \rho^{q-1}\frac{\p\rho}{\p t}(\xi,t)d\xi/{\int\hspace{-0.9em}-}_{\mS^n} \rho^q(\xi,t)d\xi,&\quad q\ne0,\\
	{\int\hspace{-0.9em}-}_{\mS^n} \rho^{-1}\frac{\p\rho}{\p t}(\xi,t)d\xi,&\quad q=0,
\end{cases}\\
=&{\int\hspace{-1.05em}-}_{\mS^n}\dfrac{\p_tu(x,t)}{\rho^{n+1-q}(\sA^*_{\Om_t}(\xi),t)K(\nu^{-1}_{\Om_t}(x))}dx\({\int\hspace{-1.05em}-}_{\mS^n} \rho^q(\xi,t)d\xi\)^{-1}.
	\end{split}
\end{equation*}
It then follows from (\ref{x1.12}), (\ref{x1.15}) and $\p_tu=\psi u^\a\rho^\delta/K^\frac{\beta}{n}$ that
\begin{equation}\label{8.1}
\begin{split}
\frac{\p t}{\p\tau}=&\({\int\hspace{-1.05em}-}_{\mS^n}\dfrac{\psi u^\a\rho^{\delta+n\delta/\beta}(\sA^*_{\Om_{t(\tau)}}(\xi),t(\tau))}{K^{\frac{\beta}{n}+1}(\nu^{-1}_{\Om_{t(\tau)}}(x))}dx\)^{-1}{\int\hspace{-1.05em}-}_{\mS^n} \rho^q(\xi,t)d\xi\\
=&e^{(1-\a-\delta-\beta)\tau}\phi(\tau),
\end{split}
\end{equation}
where $\phi(\tau)$ is given by (\ref{x1.14}) which involves the geometric quantities of $\widetilde{\Om}_\tau$.

In view of (\ref{x1.12}) and (\ref{8.1}), one infers that
\begin{equation*}
\frac{\p\widetilde{X}}{\p\tau}=e^{-\tau}\frac{\p X}{\p t}\frac{\p t}{\p\tau}-\widetilde{X}=\frac{\psi \phi(\tau)\widetilde{u}^\alpha\widetilde{\rho}^\delta}{\widetilde{K}^\frac{\beta}{n}}\nu-\widetilde{X},
\end{equation*}
thus we complete the proof.
\end{proof}
Similar to the above section, $V_q(\Om_t)$ defined in (\ref{x1.16}) is unchanged under the flow (\ref{x1.13}) by scaling.
\begin{lemma}\label{l8.2}
Let $\beta>0$. Let $X(\cdot,t)$ be a smooth solution to the flow (\ref{x1.13}) with $t\in[0,T)$, and for each $t\ge0$, $M_t=X(\mS^n,t)$ be a smooth, closed and uniformly convex
hypersurface. Suppose that the origin lies in the interior of convex body $\Om_{t}$ enclosed by $M_t$ for all $t\in[0,T)$. Then
\begin{equation*}
V_q(\Om_t)=V_q(\Om_0),\qquad \forall t\in[0,T),
\end{equation*}
where $q=n+1+n\delta/\beta$ as in (\ref{x1.15}).
\end{lemma}
\begin{proof}
If $q\ne0$, by (\ref{x1.12}), we have
\begin{equation}
\widetilde{X}=\(\dfrac{\int_{\mS^n}\rho^q_{\Om_0}d\xi}{\int_{\mS^n}\rho^q_{\Om_t}d\xi}\)^\frac{1}{q}X.
\end{equation}
Thus,
\begin{equation*}
\widetilde{V}_q(\Om_{t})=\dfrac{\int_{\mS^n}\rho^q_{\Om_0}d\xi}{\int_{\mS^n}\rho^q_{\Om_t}d\xi}\times\frac{1}{q}\int_{\mS^n}\rho^q_{\Om_t}d\xi=\frac{1}{q}\int_{\mS^n}\rho^q_{\Om_0}d\xi=V_q(\Om_{0})=\widetilde{V}_q(\Om_{0}).
\end{equation*}
If $q=0$, the proof is similar to above.
\end{proof}
\textbf{Remark}: We can also calculate the evolution equation of $\widetilde{V}_q(\Om_{t})$ to prove this lemma.

The next lemma shows that $\cJ_{p,q}$ is monotone under the flow (\ref{x1.13}).
\begin{lemma}\label{l8.3}
Let $\beta>0$. Let $X(\cdot,t)$, $M_t$ and $\Om_{t}$ be as in Lemma \ref{l8.2}. Let $p$ and $q$ be as in (\ref{x1.15}). Then functional $\cJ_{p,q}$ given by (\ref{x1.16}) is non-increasing, that is
\begin{equation*}
\frac{d}{dt}\cJ_{p,q}(X(\cdot,t))\le0,\qquad \forall t\in[0,T),
\end{equation*}
and the equality holds if and only if the support function $u(\cdot,t)$ of $\Om_t$ is a solution to (\ref{1.11}) with $\psi$ replaced by $(\psi\phi(t))^{-n/\beta}$.
\end{lemma}
\begin{proof}
By Lemma \ref{l8.1}, we have the evolution equation of $u$,
\begin{equation}\label{8.3}
\p_tu=\frac{\psi \phi(t)u^\alpha\rho^\delta}{K^\frac{\beta}{n}}-u.
\end{equation}
By Lemma \ref{l8.2}, $V_q(\Om_{t})$ is a constant along the flow. It then follows from (\ref{8.3})
\begin{equation}\label{8.4}
\begin{split}
\frac{d}{dt}\cJ_{p,q}(X(\cdot,t))=&\dfrac{1}{\int_{\mS^n}d\mu_{\psi,\beta}}\int_{\mS^n}u^{p-1}\p_tu(x,t)d\mu_{\psi,\beta}\\
=&\[\int_{\mS^n}d\mu_{\psi,\beta}\int_{\mS^n}\dfrac{\psi u^\a\rho^{\delta+n\delta/\beta}}{K^{\frac{\beta}{n}+1}}dx\]^{-1}\\
&\{\int_{\mS^n}\dfrac{\psi u^{p-1+\a}\rho^\delta}{K^\frac{\beta}{n}}d\mu_{\psi,\beta}\int_{\mS^n}\rho^qd\xi-\int_{\mS^n}u^pd\mu_{\psi,\beta}\int_{\mS^n}\dfrac{\psi u^\a\rho^{\delta+n\delta/\beta}}{K^{\frac{\beta}{n}+1}}dx\},
\end{split}
\end{equation}
where $\rho:=\rho\circ\sA^*_{\Om_{t}}=\rho(\sA^*_{\Om_{t}}(x),t)$ and $K$ is the Gauss curvature of $M_t$ at the point $\nu^{-1}_{\Om_{t}}(x)$. Let
\begin{equation*}
\Rmnum{1}=\int_{\mS^n}\dfrac{\psi u^{p-1+\a}\rho^\delta}{K^\frac{\beta}{n}}d\mu_{\psi,\beta}\int_{\mS^n}\rho^qd\xi-\int_{\mS^n}u^pd\mu_{\psi,\beta}\int_{\mS^n}\dfrac{\psi u^\a\rho^{\delta+n\delta/\beta}}{K^{\frac{\beta}{n}+1}}dx.
\end{equation*}
By (\ref{x1.15}) and (\ref{2.28}), we further deduce that
\begin{equation*}
\Rmnum{1}=\int_{\mS^n}h^\frac{n}{\beta}d\sigma\int_{\mS^n}hd\sigma-\int_{\mS^n}d\sigma\int_{\mS^n}h^{1+\frac{n}{\beta}}d\sigma,
\end{equation*}
where the measure $d\sigma$ and function $h$ are given by
$$d\sigma=u^pd\mu_{\psi,\beta}=u^p\psi^{-\frac{n}{\beta}}dx\quad\text{    and    }\quad h(x,t)=\frac{\psi u^{\a-1}\rho^\delta}{K^\frac{\beta}{n}}.$$
It follows by Hölder inequality that
\begin{align*}
\int_{\mS^n}hd\sigma\le\(\int_{\mS^n}d\sigma\)^\frac{n}{n+\beta}\(\int_{\mS^n}h^{1+\frac{n}{\beta}}d\sigma\)^\frac{\beta}{n+\beta},\\
\int_{\mS^n}h^\frac{n}{\beta}d\sigma\le\(\int_{\mS^n}d\sigma\)^\frac{\beta}{n+\beta}\(\int_{\mS^n}h^{1+\frac{n}{\beta}}d\sigma\)^\frac{n}{n+\beta}.
\end{align*}
Thus,
$$\frac{d}{dt}\cJ_{p,q}(X(\cdot,t))\le0,$$
with the equality if and only if
\begin{equation}\label{8.5}
h(x,t)=c(t)
\end{equation}
for some function $c(t)$. We still need to show $c(t)=\frac{1}{\phi(t)}$ provided (\ref{8.5}) holds. By (\ref{x1.14}) and using the change of variables formula via (\ref{2.27}),
\begin{equation*}
\frac{1}{\phi(t)}=[\int_{\mS^n}h(\sA_{\Om_t}(\xi),t)\rho^q(\xi,t)d\xi][\int_{\mS^n}\rho^q(\xi,t)d\xi]^{-1}=c(t).
\end{equation*}
\end{proof}
Since our original flow (\ref{x1.8}) and the normalized flow (\ref{x1.9}) or (\ref{x1.13}) differs a time-dependent dilation, $\frac{\vert Du\vert}{u}$ is scaling-invariant. By Lemma \ref{l7.4}, we have the following lemma.
\begin{lemma}\label{l8.4}
	Let $n\ge2$, $\beta>0$, $0<\psi\in C^\infty(\mS^n)$, $\a+\delta+\beta<1$ or $\a=1-\delta-\beta\ne\frac{\beta}{n}+1$ and $u_0\in C^\infty(\mS^n)$ be positive and uniformly convex. Let $u(\cdot,t)$ be the solution of the normalized flow (\ref{8.3}) for $t\in[0,T)$. Then there is a positive constant $C_{28}$ depending on the initial hypersurface and $\psi,\a,\delta,\beta,n$ such that
	$$\max_{\mS^n}\frac{\vert Du\vert}{u}(\cdot,t)\le C_{28}, \quad \forall t\in[0,T).$$
\end{lemma}
To prove the $C^0$ estimates, we should use the following generalized Blaschke-Santal$\acute{o}$ inequality. It was first proved in \cite{CHD}. In \cite{CCL} they gave a simple proof.
\begin{lemma}(Blaschke-Santal$\acute{o}$-type inequality \cite{CHD})\label{l8.5} Given $q>0$, let $q^*>0$ be the number given by (\ref{x1.15}). For $s\in(0,q^*]$, $s\ne+\infty$, there is a constant $C_{n,q,s}>0$ such that
\begin{equation*}
\({\int\hspace{-1.05em}-}_{\mS^n}\rho^q_\Om d\sigma_{\mS^n}\)^\frac{1}{q}\({\int\hspace{-1.05em}-}_{\mS^n}\rho^s_{\Om^*} d\sigma_{\mS^n}\)^\frac{1}{s}\le C_{n,q,s} \quad \text{ for all } \Om\in\cK^e_0,
\end{equation*}
where $\cK^e_0$ is the set of all origin-symmetric convex bodies containing the origin in their interiors.
\end{lemma}

With the help of Lemma \ref{l8.4} and Lemma \ref{l8.5}, we obtain the following estimates.
\begin{lemma}\label{l8.6}
Let $\beta>0$. Let $\psi$, $u_0$ and $u(\cdot,t)$ be as in Lemma \ref{l8.4}. Then there are positive constants $C_{41}$, $C_{42}$ depending on the initial hypersurface and $\psi,\a,\delta,\beta,n$ such that
$$C_{41}^{-1}\le u\le C_{41},\quad \vert Du\vert\le C_{42},\quad C_{41}^{-1}\le\rho\le C_{41}, \quad \forall(x,t)\in\mS^n\times[0,T),$$
if (\rmnum{1}) $\a+\delta+\beta<1$ or $\a=1-\delta-\beta\ne\frac{\beta}{n}+1$;

or (\rmnum{2}) $\a=\frac{\beta}{n}+1$, $\delta=-\frac{(n+1)\beta}{n}$,  $\int_{\mS^n}  c_0\psi^{-\frac{n}{\beta}}=\vert\mS^n\vert$ and $\int_\om  c_0\psi^{-\frac{n}{\beta}}<\vert\mS^n\vert-\vert\om^*\vert$ for some  positive constant $c_0$;

 or (\rmnum{3})  $\psi$ and $u_0$ are in addition even function with (1) $\a\le\frac{\beta}{n}+1$ (i.e. $p\ge0$); or (2) $\delta\le-\frac{(n+1)\beta}{n}$ (i.e. $q\le0$); or (3) $q>0$ and $-q^*<p<0$.
\end{lemma}
\begin{proof}
It suffices to prove $u$ and $\rho$ have uniform bounds by $\rho=\sqrt{u^2+\vert Du\vert^2}$. Let $\rho_{\min}=\min_{\mS^n}\rho=u_{\min}$ and $u_{\max}=\max_{\mS^n}u=\rho_{\max}$.

(\rmnum{1}) $\a+\delta+\beta<1$ or $\a=1-\delta-\beta\ne\frac{\beta}{n}+1$. It means that  $p>q$ or $p=q\ne0$. By Lemma \ref{l8.4}, we have $\frac{\vert Du\vert}{u}\le C_{28}$. Thus
\begin{equation*}
\max_{\mS^n}\log u(\cdot,t)-\min_{\mS^n}\log u(\cdot,t)\le C_{43}\max_{\mS^n}\frac{\vert Du\vert}{u}(\cdot,t)\le C_{44}.
\end{equation*}
In other words,
\begin{equation}\label{8.6}
\frac{\max_{\mS^n}u}{\min_{\mS^n}u}\le C_{45}.
\end{equation}
(1) $\a<\frac{\beta}{n}+1$ (i.e. $p>0$). By Lemma \ref{l8.2} and \ref{l8.3}, $\cJ_{p,q}(X(\cdot,t))\le\cJ_{p,q}(X(\cdot,0))$ and $V_q(\Om_{t})=V_q(\Om_{0})$ for all $t$. Hence by (\ref{2.30}),
\begin{equation}\label{8.7}
C_{46}\ge\int_{\mS^n}u^p\psi^{-\frac{n}{\beta}}dx\ge\int_{\{x\in\mS^n:x\cdot x_t>0\}}(x\cdot x_tu(x_t,t))^p\psi^{-\frac{n}{\beta}}dx\ge\max_{\mS^n}u^p/C_{47},
\end{equation}
where $x_t$ is a point at where $u(\cdot,t)$ attains its spatial maximum. Thus we obtain the upper bound of $u$ and $\rho$ by Lemma \ref{l2.5}.

 If $\delta<-\frac{(n+1)\beta}{n}$ (i.e. $q<0$), by a rotation, we may assume that $\rho_\min=\rho(\xi_0)$, where $\xi_0=(0,\cdots,0,1)\in\mS^n$. By (\ref{2.31}), we have
\begin{equation}\label{8.8}
\int_{\mS^n}\rho^qd\xi\ge\int_{\{\xi\in\mS^n:\xi\cdot\xi_0>0\}}\rho^qd\xi\ge\rho_{\min}^q\int_{\{\xi\in\mS^n:\xi_{n+1}>0\}}(\xi_{n+1})^{-q}d\xi\ge\rho_{\min}^q/C_{48}.
\end{equation}
Hence we have the lower bound of $u$ and $\rho$.

If $\delta=-\frac{(n+1)\beta}{n}$ (i.e. $q=0$), by (\ref{2.31}), we have
\begin{equation}\label{8.9}
\begin{split}
{\int\hspace{-1.05em}-}_{\mS^n} \log\rho(\xi,t)d\xi=&\frac{1}{\vert\mS^n\vert}\int_{\{\xi\in\mS^n:\xi_{n+1}\le0\}}\log\rho d\xi+\frac{1}{\vert\mS^n\vert}\int_{\{\xi\in\mS^n:\xi_{n+1}>0\}}\log\rho d\xi\\
\le&\frac{1}{2}\log\rho_{\max}+\frac{1}{2}\log\rho_{\min}-\frac{1}{\vert\mS^n\vert}\int_{\{\xi\in\mS^n:\xi_{n+1}>0\}}\log\xi_{n+1} d\xi\\
\le&\frac{1}{2}\log(\rho_{\max}\rho_{\min})+C_{49}.
\end{split}
\end{equation}
By the upper bound of $\rho$ and (\ref{8.9}) we have the lower bound of $\rho$.

If $\delta>-\frac{(n+1)\beta}{n}$ (i.e. $q>0$), then
\begin{equation*}
\max_{\mS^n}u(\cdot,t)\ge\({\int\hspace{-1.05em}-}_{\mS^n}\rho^q(\xi,t)d\xi\)^\frac{1}{q}=\({\int\hspace{-1.05em}-}_{\mS^n}\rho^q(\xi,0)d\xi\)^\frac{1}{q}.
\end{equation*}
By (\ref{8.6}), we can derive the lower bound of $u$. In summary, if $p>0$, we have the uniform bounds of $u$ and $\rho$.

(2) $\a\ge\frac{\beta}{n}+1$ (i.e. $p\le0$). At this time, $q<0$. By (\ref{8.8}), we have the lower bound of $u$ and $\rho$. We also have
\begin{equation*}
	\min_{\mS^n}\rho(\cdot,t)\le\({\int\hspace{-1.05em}-}_{\mS^n}\rho^q(\xi,t)d\xi\)^\frac{1}{q}=\({\int\hspace{-1.05em}-}_{\mS^n}\rho^q(\xi,0)d\xi\)^\frac{1}{q}.
\end{equation*}
By (\ref{8.6}), we can derive the upper bound of $u$. In summary, we have proved the case (\rmnum{1}).

(\rmnum{2}) $\a=1-\delta-\beta=\frac{\beta}{n}+1$. It means that $p=q=0$. let $M_t$ be a smooth convex solution to the normalized flow (\ref{8.3}) and let $\mathcal{N}$ be a smooth convex solution to $\psi u^{\a-1}\rho^\delta\sigma_n^\frac{\beta}{n}=c$ by the existence of Aleksandrov's problem, where $\a=\frac{\beta}{n}+1$, $\delta=-(n+1)\frac{\beta}{n}$. Let $\mathcal{N}_0=s_0\mathcal{N}$ and $\mathcal{N}_1=s_1\mathcal{N}$, where the constants $s_1>s_0>0$ are so chosen that $\mathcal{N}_0$ is strictly contained in $M_0$ and $M_0$ is strictly contained in $\mathcal{N}_1$. It is easily seen that $\mathcal{N}_0$ and $\mathcal{N}_1$ are invariant along the flow (\ref{8.3}). Applying the comparison principle, we can derive that $\mathcal{N}_0$ is strictly contained in $M_t$ and $M_t$ is strictly contained in $\mathcal{N}_1$. In other words, we derive the uniform bounds of $\rho$ and $u$.

(\rmnum{3}) By Case (\rmnum{1}) above, we let $p<q$ or $p=q=0$. In this case, $f$ and $u_0$ are in addition even function. It is easily seen that $\Om_t$ is origin-symmetric.

 (1) $\a\le\frac{\beta}{n}+1$ (i.e. $p\ge0$). In this case, $q>0$ or $p=q=0$.

 If $\a<\frac{\beta}{n}+1$ (i.e. $p>0$), we have the upper bound of $u$ and $\rho$ by (\ref{8.7}).

  If $\a=\frac{\beta}{n}+1$ (i.e. $p=0$), by a rotation of coordinates we may assume that $\rho_{\max}(t)=\rho(e_1,t).$ Since $\Om_{t}$ is origin-symmetric, the points $\pm\rho_{\max}(t)e_1$ are in $M_t$. Hence
 $$u(x,t)=\sup\{p\cdot x:p\in M_t\}\ge\rho_{\max}(t)\vert x\cdot e_1\vert,\quad \forall x\in\mS^n.$$
 Therefore by Lemma \ref{l8.3},
 \begin{equation}\label{8.10}
 	\begin{split}
 C_{50}\ge\int_{\mS^n}\log u(x,t)dx\ge&\vert\mS^n\vert\log\rho_{\max}(t)+\int_{\mS^n}\log \vert x\cdot e_1\vert dx\\
 \ge&\vert\mS^n\vert\log\rho_{\max}(t)-C_{51}.
 	\end{split}
 \end{equation}
Thus by (\ref{8.10}) we have the upper bound of $u$ and $\rho$.

If $\delta=-\frac{(n+1)\beta}{n}$ (i.e. $q=0$), by (\ref{8.9}) and the upper bound of $\rho$, we have the lower bound of $u$ and $\rho$.  Next we only need to derive the uniform lower bound of $u$.

If $-\frac{(n+1)\beta}{n}<\delta\le0$ (i.e. $0<q\le n+1$), by Lemma \ref{l8.2} and the H$\ddot{o}$lder inequality,
\begin{equation}\label{8.11}
\begin{split}
\({\int\hspace{-1.05em}-}_{\mS^n}\rho^q(\xi,0)d\xi\)^\frac{1}{q}=\({\int\hspace{-1.05em}-}_{\mS^n}\rho^q(\xi,t)d\xi\)^\frac{1}{q}&\le\({\int\hspace{-1.05em}-}_{\mS^n}\rho^{n+1}(\xi,t)d\xi\)^\frac{1}{n+1}\\
&=\(\dfrac{n+1}{\vert\mS^n\vert}\Vol(\Om_t)\)^\frac{1}{n+1},
\end{split}
\end{equation}
where $\Vol(\Om_t)$ denotes the volume of $\Om_t$ and we have used $\vert\frac{d\xi}{d\mu(M)}\vert=\frac{u}{\rho^{n+1}}$ in the last equality. After a rotation if necessary, we can assume $\rho(e_{n+1},t)=\rho_{\min}(t)$. Since $\Om_{t}$ is origin-symmetric, we find that $\Om_{t}$ is contained in the cube
\begin{equation*}
Q_t=\{z\in\mR^{n+1}:-\rho_{\max}(t)\le z_i\le\rho_{\max}(t)\text{ for }1\le i\le n, -\rho_{\min}(t)\le z_{n+1}\le\rho_{\min}(t)\}.
\end{equation*}
Thus, by (\ref{8.11}), we have
\begin{equation}\label{8.12}
C_{52}\le\Vol(\Om_{t})\le2^{n+1}\rho_{\max}^n(t)\rho_{\min}(t).
\end{equation}
By the upper bound of $\rho$ and (\ref{8.12}), we have the lower bound of $u$ and $\rho$.

If $\delta>0$ (i.e. $q>n+1$), by Lemma \ref{l8.2}, we have
\begin{equation}\label{8.13}
\begin{split}
\int_{\mS^n}\rho^q(\xi,0)d\xi=&\rho^q_{\max}(t)\int_{\mS^n}\(\frac{\rho(\xi,t)}{\rho_{\max}(t)}\)^{q}d\xi\le\rho^q_{\max}(t)\int_{\mS^n}\(\frac{\rho(\xi,t)}{\rho_{\max}(t)}\)^{n+1}d\xi\\
=&(n+1)\rho^{q-n-1}_{\max}(t)\Vol(\Om_t)\le C_{53}\rho^{q-1}_{\max}(t)\rho_{\min}(t).
\end{split}
\end{equation}
By the upper bound of $\rho$ and (\ref{8.13}), we have the lower bound of $u$ and $\rho$.

(2) $\delta\le-\frac{(n+1)\beta}{n}$ (i.e. $q\le0$).

If $\delta=-\frac{(n+1)\beta}{n}$ (i.e. $q=0$), by (\ref{2.32}) and Lemma \ref{l8.2}, we have
$$\int_{\mS^n}\log u^*d\xi=C_{54}.$$
Since $\Om_{t}^*$ is origin-symmetric, similar to (\ref{8.10}), we have the upper bound of $u^*$. By (\ref{2.32}) and Lemma \ref{l2.5}, we have the lower bound of $u$ and $\rho$.

If $\delta<-\frac{(n+1)\beta}{n}$ (i.e. $q<0$), by (\ref{8.8}) we have the lower bound of $u$ and $\rho$. Next we only need to derive the uniform upper bound of $u$.

If $\a\le\frac{\beta}{n}+1$ (i.e. $p\ge0$), by (\ref{8.7}) and (\ref{8.10}) we have the upper bound of $u$ and $\rho$. If $\a>\frac{\beta}{n}+1$ (i.e. $p<0$), by Lemma \ref{l8.3}, we have $\int_{\mS^n}u^pd\mu_{\psi,\beta}\ge C_{55}$.  This implies that
\begin{equation*}
\int_{\mS^n}(\rho^*)^{-p}d\mu_{\psi,\beta}\ge C_{55}.
\end{equation*}
Since $p<0$, we have $-p>0$. By (\ref{8.12}) and (\ref{8.13}), we have the lower bound of $\rho^*$. In other words, we have the upper bound of $u$ and $\rho$.

(3)  $q>0$ and $-q^*<p<0$. First, we shall derive the upper bound of $u$. We follow an argument in \cite{CW}. Suppose there are a sequence original-symmetric convex bodies $\Om_{t_j}$ along this flow, and the diameter of $\Om_{t_j}$, $2L(t_j)\rightarrow\infty$ as $t_j\rightarrow T$, where $L(t):=\max_{\mS^n}u(t)$. Let $\eps>0$ be a fixed small constant. Set $S_1(t)=\mS^n\cap\{u(t)\le\frac{1}{\eps}\}$ and $S_2(t)=\mS^n\cap\{u(t)\ge\frac{1}{\eps}\}$. Since $\Om_{t_j}$ is origin-symmetric, we have $u_j(y)\ge L(t_j)|x_0\cdot y|$ for any $y\in\mS^n$, where $u_j$ attains the maximum at $x_0\in\mS^n$. We conclude that
\begin{equation}\label{8.14}
\vert S_1(t_j)\vert\rightarrow0 \quad \text{ as } L(t_j)=\max_{\mS^n}u(t_j)\rightarrow\infty.
\end{equation}
Let $\Om^*$ be the polar set of $\Om$, and $\rho^*(t)=\rho_{\Om_{t}^*}$. Let $s=-p>0$. We have $s<q^*$. Thus by Lemma \ref{l8.3},
\begin{equation}\label{8.15}
	\begin{split}
C_{56}\le&\int_{S_1(t)\cup S_2(t)}u^{-s}dx\\
\le&\int_{S_1(t)}(\rho^*)^sdx+C_{54}\eps^s\\
\le&\(\int_{S_1(t)}(\rho^*)^{q'}dx\)^\frac{s}{q'}\vert S_1\vert^{1-\frac{s}{q'}}+C_{57}\eps^s,
	\end{split}
\end{equation}
for any $s<q'<q^*$. Since by Lemma \ref{l8.5} and \ref{l8.2}, we have
$$\(\int_{S_1(t)}(\rho^*)^{q'}dx\)^\frac{s}{q'}\le C_{58}.$$
If $L(t_j)\rightarrow\infty$, then by (\ref{8.14}),
\begin{equation*}
\int_{S_1(t_j)\cup S_2(t_j)}u^{-s}dx\le C_{59}\eps^s.
\end{equation*}
Thus we arrive a contradiction by letting $\eps\rightarrow0$. By (\ref{8.11}), (\ref{8.12}) and (\ref{8.13}), we can derive the lower bound of $u$. In summary, we derive the uniform bounds of $u$ and $\rho$.
\end{proof}
We show that the Gauss curvature of $M_t$ evolved by (\ref{x1.13}) has a uniform positive
lower bound as long as the flow exists. We shall use the following lemma proved in \cite{CL} to derive a uniform positive lower bound easily.
\begin{lemma}\label{l8.7}\cite{CL}
Let $\beta>0$. Let $u(\cdot,t)$ be a positive, smooth and uniformly convex solution to
\begin{equation*}
\frac{\p u}{\p t}=-\dfrac{\phi(t)G(x,u,Du)}{\det^\frac{\beta}{n}(D^2u+u\Rmnum{1})}+\eta(t)u\quad\text{ on }\mS^n\times[0,T),
\end{equation*}
where $G(x,u,Du):\mS^n\times\mR_{\ge0}\times\mR^n\rightarrow\mR_{\ge0}$ is a smooth function, $\mR_{\ge0}=[0,\infty)$.  Suppose that $G(x,u,Du)>0$ whenever $u>0$. If
\begin{align*}
	\frac{1}{C_0}\le u(x,t)\le C_0,\qquad\qquad &\forall(x,t)\in\mS^n\times[0,T),\\
	\vert Du\vert(x,t)\le C_0,\qquad\qquad &\forall(x,t)\in\mS^n\times[0,T),\\
\vert\eta(t)\vert\le C_0,\qquad\qquad&\forall t\in[0,T),\\
	\phi(t)/(\min_{\mS^n}\det(D^2u+u\Rmnum{1})(\cdot,t))^\frac{\beta}{n}\ge\frac{1}{C_0},\quad&\forall t\in[0,T),
\end{align*}
for some constant $C_0>0$, then
\begin{equation*}
\min_{\mS^n}\det(D^2u+u\Rmnum{1})(\cdot,t)\ge C^{-1},\qquad \forall t\in[0,T),
\end{equation*}
where $C$ is a positive constant depending only on $n,\beta,C_0,u(\cdot,0)$, $\vert G\vert_{L^\infty(U)}$, $\vert1/G\vert_{L^\infty(U)}$, $\vert G\vert_{C^1_{x,u,Du}(U)}$ and $\vert G\vert_{C^2_{x,u,Du}(U)}$ where $U=\mS^n\times[1/C_0,C_0]\times B^n_{C_0}$ ($B^n_{C_0}$ is the ball
centered at the origin with radius $R$ in $\mR^n$).
\end{lemma}
\textbf{Remark}: Lemma \ref{l8.7} was proved directly in \cite{CL} by the evolution equation of $$Q=\frac{\eta(t)u-u_t}{\phi(t)(u-\sigma_0)}=\frac{GK^\beta}{u-\sigma_0},$$
where $\sigma_0=\frac{1}{2}\inf\{u(x,t):(x,t)\in\mS^n\times[0,T)\}>0$.

\begin{lemma}\label{l8.8}
Let $\beta>0$. Let $\psi$, $u_0$, $u(\cdot,t)$, $\a$, $\delta$ and $\beta$ be as in Lemma \ref{l8.6}. Then there is a  positive constant $C_{60}$ depending on the initial hypersurface and $\psi,\a,\delta,\beta,n$ such that
$$\det(D^2u+u\Rmnum{1})(x,t)\le C_{60}, \quad \forall(x,t)\in\mS^n\times[0,T).$$
\end{lemma}
\begin{proof}
Let $\Om_t$ be the convex body whose support function is $u(\cdot,t)$. Let $\Om_t^*$ be the
polar set of $\Om_t$ and $u^*(\xi,t)$ be the support function of $\Om_t^*$. By (\ref{2.32}), (\ref{2.33}), (\ref{2.34}) and (\ref{8.3}), we have
\begin{equation}\label{8.16}
\begin{cases}
\frac{\p u^*}{\p t}(\xi,t)=-\dfrac{\phi(t)G(\xi,u^*,Du^*)}{\det^\frac{\beta}{n}(D^2u^*+u^*\Rmnum{1})}+u^*\quad\text{ on }\mS^n\times[0,T),\\
u^*(\xi,0)=u^*_{0}(\xi),
\end{cases}
\end{equation}
where $u^*_{0}$ is the support function of $\Om^*_{0}$ and
\begin{equation*}
G(\xi,u^*,Du^*)=(u^{*2}+\vert Du^*\vert^2)^\frac{n-n\a+(n+2)\beta}{2n}u^{*(1-\delta-\frac{(n+2)\beta}{n})}\psi\(\dfrac{Du^*+u^*\xi}{\sqrt{u^{*2}+\vert Du^*\vert^2}}\).
\end{equation*}
Then this lemma can be proved by applying Lemma \ref{l8.7} to (\ref{8.16}). It suffices to verify conditions in Lemma \ref{l8.7}.

Note that $\eta(t)=1$ in the current situation, and $\frac{1}{C_0}\le u,\rho\le C_0$ are consequences of Lemma \ref{8.6}. Finally we only need to verify $$\phi(t)/(\min_{\mS^n}\det(D^2u+u\Rmnum{1})(\cdot,t))^\frac{\beta}{n}\ge\frac{1}{C_0}.$$
Recall that $\phi(t)$ is given by (\ref{x1.14}). Since we have the uniform bounds of $u$ and $\rho$, using (\ref{2.27}), (\ref{2.33}) and Lemma \ref{l8.2}, we have
\begin{equation*}
\frac{1}{\phi(t)}\le C_{61}\int_{\mS^n}\dfrac{d\xi}{\det^\frac{\beta}{n}(D^2u^*+u^*\Rmnum{1})}\le\dfrac{C_{61}\vert\mS^n\vert}{\min_{\mS^n}\det^\frac{\beta}{n}(D^2u^*+u^*\Rmnum{1})(\cdot,t)}.
\end{equation*}
Hence we complete the proof by making use of Lemma \ref{l8.7}.
\end{proof}
\begin{lemma}\label{l8.9}
Let $\beta>0$. Let $\psi$, $u_0$, $u(\cdot,t)$, $\a$, $\delta$ and $\beta$ be as in Lemma \ref{l8.6}. Then there is a  positive constant $C_{62}$ depending on the initial hypersurface and $\psi,\a,\delta,\beta,n$ such that
\begin{equation}\label{8.17}
C_{62}^{-1}\le\phi(t)\le C_{62}, \quad \forall t\in[0,T).
\end{equation}
\end{lemma}
\begin{proof}
The first inequality in (\ref{8.17}) is a consequence of (\ref{x1.14}), Lemma \ref{l8.2}, \ref{l8.6} and \ref{l8.8}.

On the other hand, by virtue of (\ref{2.28}) and Lemma \ref{l8.6},
\begin{equation}\label{8.18}
\vert\mS^n\vert=\int_{\mS^n}d\xi=\int_{\mS^n}\dfrac{u}{\rho^{n+1}K}dx\le C_{63}\int_{\mS^n}\dfrac{dx}{K}.
\end{equation}
Hence, by using (\ref{x1.14}), Lemma \ref{l8.2} and \ref{l8.6} again,
\begin{align*}
\frac{1}{\phi(t)}\ge C_{64}\int_{\mS^n}\dfrac{dx}{K^{\frac{\beta}{n}+1}}\ge C_{65}\(\int_{\mS^n}\dfrac{dx}{K}\)^{\frac{\beta}{n}+1}\ge C_{66},
\end{align*}
where the last inequality is due to (\ref{8.18}). Therefore the second inequality in (\ref{8.17}) holds.
\end{proof}
With the help of the above lemmas, we can prove Theorem \ref{xt1.5}.

\begin{proof of theorem 1.5}

Since equation (\ref{8.3}) is parabolic, we have the short-time existence. Let $T$ be the
maximal time such that $u(\cdot,t)$ is a positive, smooth and uniformly convex solution to
(\ref{8.3}) for all $t\in[0,T).$

Lemma \ref{l8.6} and {\ref{l8.9}} enable us to apply Lemma \ref{l7.6} to (\ref{8.3}) and thus deduce a uniformly positive lower bound estimate for the smallest eigenvalue of $[u_{ij}+u\delta_{ij}]$. This together with Lemma \ref{l8.8} shows that
$$ C^{-1}\Rmnum{1}\le(D^2u+u\Rmnum{1})(x,t)\le C\Rmnum{1},\qquad \forall(x,t)\in\mS^n\times[0,T),$$
where $C>0$ depends only on $n$, $\psi$, $\a$, $\delta$, $\beta$ and $u_0$. This implies that (\ref{8.3}) is uniformly parabolic. By the same argument as the argument below Lemma \ref{l7.7}, we obtain the long time existence and $C^\infty$-smoothness of solutions for the normalized flow (\ref{8.3}). The uniqueness of smooth solutions also follows from the parabolic theory.

By the monotonicity of $\cJ_{p,q}$ (see Lemma \ref{l8.3}), and noticing that
$$\vert\cJ_{p,q}(X(\cdot,t))\vert\le C,\qquad\forall t\in[0,\infty),$$
we conclude that
$$\int_{0}^\infty\vert\frac{d}{dt}\cJ_{p,q}(X(\cdot,t))\vert dt\le C.$$
Hence there is a sequence of $t_i\rightarrow\infty$ such that $\frac{d}{dt}\cJ_{p,q}(X(\cdot,t))\rightarrow0$. By Lemma \ref{l8.3}, we see that $u(\cdot,t_i)$ converges smoothly to a positive, smooth and uniformly function $u_\infty$ solving (\ref{1.11}) with $\psi$ replaced by $c_0\psi^{-n/\beta}$, where $c_0=\lim_{t_i\rightarrow\infty}(\phi(t_i))^{-n/\beta}$. When $p> q$, we have proved the uniqueness of solutions to (\ref{1.11}) in Proof of Theorem 1.4. If $p=q$, the uniqueness of solutions to (\ref{1.11}) have been proved in \cite{CL} Lemma 4.2. Therefore the convergence of $u(\cdot,t)$ is for any sequence of times.
\end{proof of theorem 1.5}

\section{Proof of Corollary 1.6, 1.7 and Theorem 1.8}
Notice that the solution to $u^{\alpha-1}\rho^\delta F^{\beta}(D^2u+u\Rmnum{1})=\eta$ remains invariant under flow (\ref{1.7}). By Theorem \ref{1.2} and \ref{1.3}, we have the convergence of the flow (\ref{1.7})  if $F(D^2u+u\Rmnum{1})=\sigma_{k}^\frac{1}{k}(D^2u+u\Rmnum{1})$ with (\rmnum{1}) $\forall \a,\delta\le0,\beta>0$; (\rmnum{2}) $\forall \a\le0<\beta\le1-\a-\delta$.
\begin{proof of corollary 1.6 and 1.7}
	
 We can find that if $q=k+1$, (\ref{1.12}) is (\ref{1.9}). Now we consider this equation
$$\sigma_{k}(D^2u+u\Rmnum{1})=c u^\frac{k(1-\a)}{\beta}\rho^\frac{-k\delta}{\beta},$$
for $\forall 1\le k\le n$. Thus $p-1=\frac{k(1-\a)}{\beta}$ and $k+1-q=\frac{-k\delta}{\beta}$. We can find that $\a\le0$ isn't needed in Theorem \ref{t1.2} if $F=\sigma_{n}^{1/n}$ by Lemma \ref{l7.6}. In fact, we only need the condition $\a\le0$ in $C^2$ estimates if $F=\sigma_{n}^{1/n}$.

(\rmnum{1}) If $\forall \a,\delta\le0,\beta>0$, then $\forall p,q,$  $ p-1>0,k+1-q\ge0$ is correct.

(\rmnum{2}) If $\forall0<\beta\le1-\a-\delta$, we can find $k+p-q=\frac{k}{\beta}(1-\a-\delta)\ge k$. Therefore $\forall p,q,$  $p\ge q$ is correct. If $k\ne n$, we have $p>1$ in addition by $\a\le0$.

In summary, Corollary \ref{t1.4} and \ref{t1.5} have been proved.
\end{proof of corollary 1.6 and 1.7}

\begin{proof of theorem 1.8}

By Theorem \ref{xt1.4} and \ref{xt1.5}, we have the existence and partial uniqueness of the solution of
$$\sigma_{k}(D^2u+ug_{\mS^n})=(u^2+|Du|^2)^\frac{k+1-q}{2}u^{p-1}\psi'(x).$$
Note that $\a,\delta,\beta,\psi$ in conditions of Theorem \ref{xt1.4} and \ref{xt1.5} corresponds with $\a,\delta,\beta,\psi$ in equation $\psi u^{\a-1}\rho^\delta\sigma_{k}^\frac{\beta}{k}=c$. At this time, we can find $\sigma_{k}=(c\psi)^{-\frac{k}{\beta}}u^{\frac{k(1-\a)}{\beta}}\rho^{-\frac{k\delta}{\beta}}$. Thus if we let $(D_iD_j\psi^{\frac{1}{1+\beta-\a}}+\delta_{ij}\psi^{\frac{1}{1+\beta-\a}})$ be positive definite or negative definite, then $(D_iD_j\psi'^{\frac{1}{p+k-1}}+\delta_{ij}\psi'^{\frac{1}{p+k-1}})$ be positive definite or negative definite respectively by $p-1=\frac{k(1-\a)}{\beta}$ and $\psi'=(c\psi)^{-\frac{k}{\beta}}$. Similarly, we let $\int_{\mS^n}c_0\psi^{-\frac{n}{\beta}}=\vert\mS^n\vert$ and $\int_\om c_0\psi^{-\frac{n}{\beta}}<\vert\mS^n\vert-\vert\om^*\vert$ in \ref{xt1.4} and \ref{xt1.5} by the existence of solution of Alexsandrov's problem. How to derive the range of $p,q$ is similar to Corollary \ref{t1.4} and \ref{t1.5}.
\end{proof of theorem 1.8}

\section{Reference}


\end{document}